\documentclass{amsart}
  \usepackage[top=1in, bottom=1in, left=1.375in, right=1.375in]{geometry}
 \usepackage{amsmath,amssymb}
 \usepackage{algpseudocode}
 \usepackage{algorithm}
 \usepackage{algorithmicx}
 \usepackage{caption}
 \usepackage{subcaption}
 \usepackage{amsthm}
 \numberwithin{equation}{section}
 \usepackage{amsfonts}
 \usepackage{graphicx}
 \usepackage{hyperref}

\usepackage{amsmath}
\usepackage{amssymb}
\usepackage{amsthm}
\usepackage{amsmath}
\usepackage{setspace}
\usepackage{xcolor}
\usepackage{fancyhdr}
\usepackage{amssymb}
\usepackage{amsthm}
\usepackage{listings}
    \usepackage{cite}
    \usepackage{tabularx}
    \usepackage{graphicx}
    \usepackage{epstopdf}    
    \usepackage{epsfig}
    \usepackage{float}
    \usepackage{ listings}
    \usepackage{appendix}
 \theoremstyle{plain}
  \newtheorem{thm}{Theorem}
 \newtheorem{prop}{Proposition}[section]
 \newtheorem{lem}[prop]{Lemma}
 
 \theoremstyle{definition}
 \newtheorem{definition}[prop]{Definition}
 
 \theoremstyle{remark}
 \newtheorem{remark}[prop]{Remark}

 \let\pa=\partial
 \let\al=\alpha
 \let\b=\beta
 \let\d=\delta
 \let\g=\gamma
 \let\e=\varepsilon

 \let\lam=\lambda
 
 \let\s=\sigma
 \let\f=\frac
  \let\infim = \inf
 \let\inf = \infty
 \let \les = \lesssim
 \let \gtr = \gtrsim
 \let\om=\omega
 
 \let \th = \theta
  \let \vth = \vartheta
 \let \vp = \varphi
 \let\G= \Gamma
\let\B = \Big
 \let\D=\Delta
 \let\Lam=\Lambda
 \let\S=\Sigma
 \let\Om=\Omega
 \let \Ups = \Upsilon
 \let\td = \tilde
 \let\wt=\widetilde

 \let\teq \triangleq
 
 \let\pa=\partial
 \let \bsh = \backslash

 \def \dist{{\mathrm dist}}

 \def\cA{{\mathcal A}}
 
 \def\cC{{\mathcal C}}

 \def\cH{{\mathcal H}}

 \def\cL{{\mathcal L}}

 \def\cP{{\mathcal P}}
 
 \def\cR{{\mathcal R}}
 
 \def\cT{{\mathcal T}}

 \def\cW{{\mathcal W}}

 \def\na{\nabla}
 \def\la{\langle}
 \def\ra{\rangle}
\def\lt{\left}
\def\rt{\right}
\def\one{\mathbf{1}}

 \newcommand{\beq}{\begin{equation}}
 \newcommand{\eeq}{\end{equation}}
  \newcommand{\bal}{\begin{aligned} }
  \newcommand{\eal}{\end{aligned}}
 \newcommand{\ben}{\begin{eqnarray}}
 \newcommand{\een}{\end{eqnarray}}
 \newcommand{\beno}{\begin{eqnarray*}}
 \newcommand{\eeno}{\end{eqnarray*}}

 \newcommand{\uu}{\mathbf{u}}
 \newcommand{\vv}{\mathbf{v}}
 
 \newcommand{\xx}{\mathbf{x}}

  \newcommand{\LL}{\mathbf{L}}

 \newcommand{\R}{\mathbb{R}}

  \newcommand{\BT}{\mathbb{T}}
  \newcommand{\BZ}{\mathbb{Z}}

 \newcommand{\supp}{\mathrm{supp}}

 \author{Jiajie Chen and Thomas Y. Hou}
 \address{Applied and Computational Mathematics, California Institute of Technology, Pasadena, CA 91125, USA. Email: jchen@caltech.edu}
 \date{\today}

\title[Euler stability ]{
On stability and instability of $C^{1,\alpha}$ singular solutions to the 3D Euler and 2D Boussinesq equations}

\begin{document}
\begin{abstract}
Singularity formation of the 3D incompressible Euler equations is known to be extremely challenging \cite{majda2002vorticity,gibbon2008three,kiselev2018,Elg22,constantin2007euler}. In \cite{elgindi2019finite} (see also \cite{elgindi2019stability}), Elgindi proved that the 3D axisymmetric Euler equations with no swirl and $C^{1,\alpha}$ initial velocity develops a finite time singularity. Inspired by Elgindi's work,  we proved that the 3D axisymmetric Euler and 2D Boussinesq equations with $C^{1,\alpha}$ initial velocity and boundary develop a stable asymptotically self-similar (or approximately self-similar) finite time singularity \cite{chen2019finite2} in the same setting as the Hou-Luo blowup scenario \cite{luo2014potentially,luo2013potentially-2}. On the other hand, the authors of \cite{vasseur2020blow,lafleche2021instability}  recently showed that blowup solutions to the 3D Euler equations are hydrodynamically unstable. The instability results obtained in \cite{vasseur2020blow,lafleche2021instability} require some strong regularity assumption on the initial data, which is not satisfied by the $C^{1,\alpha}$ velocity field. In this paper, we generalize the analysis of \cite{elgindi2019finite,chen2019finite2,vasseur2020blow,lafleche2021instability} to show that the blowup solutions of the 3D Euler and 2D Boussinesq equations with $C^{1,\alpha}$ velocity are unstable under the notion of stability introduced in \cite{vasseur2020blow,lafleche2021instability}. These two seemingly contradictory results reflect the difference of the two approaches in studying the stability of 3D Euler blowup solutions. The stability analysis of the blowup solution obtained in \cite{elgindi2019finite,chen2019finite2} is based on the stability of a dynamically rescaled blowup profile in space and time, which is nonlinear in nature. The linear stability analysis in \cite{vasseur2020blow,lafleche2021instability} is performed by directly linearizing the 3D Euler equations around a blowup solution in the original variables. It does not take into account the changes in the blowup time, the dynamic changes of the rescaling rate of the perturbed blowup profile and the blowup exponent of the original 3D Euler equations using a perturbed initial condition when there is an approximate self-similar blowup profile. Such information has been used in an essential way in establishing the nonlinear stability of the asymptotically self-similar blowup profile in \cite{elgindi2019finite,chen2019finite2,elgindi2019stability}.

\end{abstract}

 \maketitle

\vspace{-0.1in}
\section{Introduction}

Whether the 3D incompressible Euler equations can develop a finite time singularity from smooth initial data with finite energy is one of the most challenging open questions in nonlinear partial differential equations \cite{majda2002vorticity,gibbon2008three,kiselev2018,Elg22,constantin2007euler}. 
In \cite{luo2014potentially,luo2013potentially-2}, the authors provided convincing numerical evidence that the 3D incompressible Euler equations with smooth initial data and boundary develop a finite time singularity. This work has inspired a number of subsequent theoretical studies,see e.g.\cite{kiselev2013small,choi2015finite,kiselev2015finite,choi2014on,
chen2019finite,chen2019finite2,chen2021HL,elgindi2018finite,elgindi2017finite}. Inspired by Elgindi's seminal work on singularity formation of the 3D axisymmetric Euler equations with no swirl and   $C^{1,\alpha}$ velocity \cite{elgindi2019finite}, we have proved rigorously that the axisymmetric Euler and the 2D Boussinesq equations with $C^{1,\alpha}$ initial velocity of finite energy and boundary develop a stable asymptotically self-similar (or approximately self-similar) finite time singularity \cite{chen2019finite2}. There has been some important progress on singularity formation and small-scale creation in incompressible fluids. We refer to \cite{kiselev2018,Elg22} for excellent surveys. On the other hand, in two recent papers \cite{vasseur2020blow,lafleche2021instability}, the authors showed that blow-up solutions to the 3D Euler equations are hydrodynamically unstable. The instability results obtained in \cite{vasseur2020blow,lafleche2021instability}  require some strong regularity assumption on the initial data, which is not satisfied by the $C^{1,\alpha}$ velocity. In this paper, we generalize the analysis of \cite{elgindi2019finite,chen2019finite2,elgindi2019stability,vasseur2020blow,lafleche2021instability}  to prove that the $C^{1,\al}$ blowup solutions of the 3D Euler and the 2D Boussinesq equations \cite{elgindi2019finite,chen2019finite2,elgindi2019stability} are unstable under the notion of stability introduced in \cite{vasseur2020blow,lafleche2021instability}. 

These two seemingly contradictory results reflect the difference of the two approaches in studying the stability of singular solutions to the 3D Euler equations. The stability analysis in \cite{vasseur2020blow,lafleche2021instability} is based on the linearized Euler equations around a blowup solution in the original physical variables. However, the perturbed solution of the linearized Euler equations is completely different from the perturbed solution of the original 3D Euler equations using a perturbed initial condition. If the perturbed initial condition leads to a blowup time $T^*$ that is smaller than the blowup time $T$ of the background blowup solution, i.e. $T^* < T$, the perturbed solution of the linearized Euler equations would not be able to capture this effect and will remain regular for  $t \in [T^*,T)$. 
On the other hand, if $T^* > T$, then the  perturbed solution of the linearized Euler equations cannot be extended beyond $T$ due to the singularity of the background singular solution. But the solution of the original Euler equations is still regular for $t \in [T,T^*)$. Thus, the linearized Euler equations do not capture the singular behavior of the original Euler equations close to the blowup time due to a small perturbation in the initial data. This seems to be one of the main sources of instability induced by the framework of studying stability of a singular solution to the 3D Euler equations using the linearized Euler equations. Note that the blowup time $T^*$ depends nonlinearly on the perturbed initial data \cite{chen2019finite2,elgindi2019stability}.


The nonlinear stability of the asymptotically self-similar (or approximately self-similar) blowup profile using the dynamic rescaling formulation or the modulation technique in \cite{elgindi2019finite,chen2019finite2,elgindi2019stability} is very different from the linear stability performed in \cite{vasseur2020blow,lafleche2021instability}. The dynamic rescaling formulation or the modulation technique involve a nonlinear transform of the physical equations by rescaling the solution dynamically in the spatial  and the temporal variables. The dynamic rescaling formulation allows us to incorporate the changes of the blowup time, the blowup profile and the blowup exponent (see $\b$ below) by choosing suitable rescaling parameters that come from the scaling symmetry of \eqref{eq:euler} or \eqref{eq:bous}. Since the linearization around an approximate blowup profile is performed {\it after} we make this nonlinear transform, the linear stability under this framework is nonlinear in nature.

We remark that the authors of \cite{lafleche2021instability} also studied the profile instability of a self-similar blowup solution to the 3D Euler equations in \cite{lafleche2021instability}. 
More specifically, given a background self-similar blowup solution $u(x, t) = (T- t)^{\al} U( t, \f{x}{ (T-t)^{\b}} )$, the authors assumed that the perturbed solution of the linearized equation \eqref{eq:euler_lin} takes  the same form  $v(x, t) = (T- t)^{\al} V( t, \f{x}{ (T-t)^{\b}} )$.
 Thus, the perturbed solution of the linearized equation does not capture the change in the blowup time and the dynamic changes of the rescaling rate of the perturbed profile and the 
blowup exponent $\b$ of the original 3D Euler equations using a perturbed initial condition. Therefore, the perturbed profile of the linearized Euler equations cannot be used to study the stability of the self-similar blowup profile of the original 3D Euler equations close to the blowup time using a perturbed initial condition.


The 3D incompressible Euler equations read 
\beq\label{eq:euler}
  \uu_{t} + \uu \cdot \nabla \uu = -\nabla p, \quad  \nabla \cdot \uu = 0,
\eeq
where $\uu$ is the velocity field and $p$ is the scalar pressure.
In \cite{vasseur2020blow}, the authors studied the stability of a singular solution $\uu(t)$ of the 3D Euler equations by analyzing the growth of the perturbation $\vv(t)$ using the following linearized Euler equations around $\uu(t)$:
\beq\label{eq:euler_lin}
\vv_t + \uu \cdot \na \vv + \vv \cdot \na \uu + \na q = 0, \quad \na \cdot \vv = 0.
\eeq

In a subsequent paper \cite{lafleche2021instability}, the authors generalized their earlier results to the axisymmetric Euler equations. Recall that a vector field $f(x)$ is axisymmetric \cite{majda2002vorticity} if it can be represented as 
\beq\label{axi}
f(x) = f^r(r, z) e_r + f^{\vth}(r, z) e_{\vth} + f^z(r, z) e_z,
\eeq
where 
$(r,\vth, z)$ are the cylindrical coordinate with basis $e_r = ( \cos \vth, \sin \vth, 0 )$, $ e_{\vth} = ( - \sin \vth, \cos \vth, 0)$, $e_z = (0, 0, 1)$. For a solution $\uu$ with axisymmetric initial data $\uu_0$, the axisymmetry property is preserved dynamically by the Euler equations \eqref{eq:euler}.

\subsection{Main results}



We consider singular solutions $\uu$ to \eqref{eq:euler} in a domain $D$ with the following symmetry in $z$
\beq\label{eq:sym}
\tag{Sym}
\uu = u^r e_r + u^{\vth} e_{\vth} + u^z e_z, \quad  u^r, u^{\vth} \mathrm{ \ are \ even \ in \ } z , \quad u^{z} \mathrm{ \ is \ odd \ in \ } z.
\eeq

Denote by $X$ the set of axisymmetric functions with symmetry given in \eqref{eq:sym},  $H_X^1(D) = H^1( D) \cap X$. Let $\vv$ be the solution of the linearized Euler equations \eqref{eq:euler_lin} with initial data $\vv_0$. Following \cite{lafleche2021instability}, we define the growth factors $\lam_{p, \s, D }(t)$ and $\lam^{sym}_{p, \s, D }(t)$ as follows: 
\beq\label{eq:instab1}
\lam_{p, \s, D }(t) = \sup_{ \vv_0 \in H^1( D) , \vv_0 \neq 0 } \f{ || r^{-\s} \vv(t, \cdot) ||_{L^{p}( D  ) } }{ || r^{-\s} \vv_0 ||_{L^{p}(D ) } } , \quad
\lam^{sym}_{p, \s, D }(t) = \sup_{ \vv_0 \in H_X^1( D), \vv_0 \neq 0 } 
\f{ || r^{-\s} \vv(t, \cdot) ||_{L^{p}( D ) } }{ || r^{-\s} \vv_0 ||_{L^{p}( D ) } } .
\eeq
Note that $\lam^{sym}_{p, \s}(t) \leq \lam_{p, \s}(t)$ since $H_X^1(D)$ is a subclass of axisymmetric functions in $H^1( D )$. 

In the first main result, we consider \eqref{eq:euler} in a cylinder $D = \{ (r, z) : r \leq 1,  z \in \BT\}$ periodic in $z$ (axial direction) with period $2$, where $r$ is the radial variable and $\BT = \R / (2 \BZ)$. This setting is the same as that in \cite{luo2014potentially,luo2013potentially-2,chen2019finite2}. 
We prove that the blowup solution constructed in \cite{chen2019finite2} is linearly unstable under the notion of stability introduced in \cite{lafleche2021instability}, even in the symmetry class \eqref{eq:sym}.

\begin{thm}\label{thm:Euler_HL}
There exists $\al_0 > 0$ such that for any $0 < \al < \al_0$, the 3D axisymmetric Euler equations \eqref{eq:euler} in the cylinder $(r, z) \in [0, 1] \times \BT$  develops a singularity at finite time $T_*$ from some $C^{1,\al}$ initial data $\uu_0$ with finite energy. Moreover, there exists $R_{2,\al} < \f{1}{4}$, such that the solution $\uu$ \eqref{axi} satisfies 
$u^r, u^z, u^{\vth} \in L^{\inf}( [0, T] ,  C^{50}( \S) )$ for any compact domain $\S \subset  \{ (r, z) : r \in (0, 1) , z \neq 0 \} \cap B_{ (1,0)}( R_{2,\al})$ and $T < T_*$. For any $p \in [1, \infty)$ and $\s \in \R$, we have 
\[
 \lim_{t \to T_*} \lam^{sym}_{p, \s, D}(t) = \infty. 
\]

\end{thm}
Note that the range of $\s$ is larger than that in \cite{lafleche2021instability}.
We can prove the whole range of $\s$ since the singular solution \cite{chen2019finite2} is supported near $(r, z) = (1, 0)$, which allows us to construct a unstable solution  supported near $(r, z) = (1, 0)$. Thus, the weight $r^{-\s}$ in \eqref{eq:instab1} is essentially equal to $1$.

In the second main result, we consider the singular solution in $\R^3$ constructed by Elgindi \cite{elgindi2019finite} (see also \cite{elgindi2019stability}) and prove a similar instability result for a smaller range of parameter $\s < -1$.
\begin{thm}\label{thm:Euler_R3}
There exists $\al_0 > 0$  such that for any $0 < \al < \al_0$, the 3D axisymmetric Euler equations \eqref{eq:euler} in $\R^3$  develops a singularity at finite time $T_*$ from some $C^{1,\al}$ initial data $\uu_0$ with finite energy and without swirl. Moreover, the solution $\uu$ \eqref{axi} satisfies $u^{\vth} \equiv 0, u^r, u^z  \in L^{\inf}( [0, T], C^{50}( \S) )$ for any compact domain $\S \subset \{ (r, z): r >0, z \neq 0 \} $ and $T < T_* $. For any $p \in (2, \infty)$ and $\s \in ( - \f{2 (p-1)}{ p}, -1 )$, we have 
\[
 \lim_{t \to T_* } \lam^{sym}_{p, \s, \R^3}(t) = \infty. 
\]
\end{thm}

Note that for $p \in [1, 2]$, the interval $( - \f{2 (p-1)}{ p}, -1 )$ is empty.

Next, we generalize the instability results to the 2D Boussinesq equations in $\R_2^+$
\beq\label{eq:bous}
\bal
\om_t +  \uu \cdot \na \om  &= \th_{x},  \quad  \th_t + \uu \cdot  \na \th  =  0 , 
\eal
\eeq
where the velocity field $\uu = (u , v)^T : \R_+^2 \times [0, T) \to \R^2_+$ is determined via the Biot-Savart law
\beq\label{eq:biot}
 - \D \psi = \om , \quad  u =  - \psi_y , \quad v  = \psi_x,
\eeq
with no flow boundary condition $v(x, 0) = 0$. Given a singular solution $(\th, \uu)$, the linearized equations of \eqref{eq:bous} in the velocity-density formulation around $(\th, \uu)$ read
\beq\label{eq:bous_lin0}
\bal
&\pa_t \eta + \uu \cdot \na \eta + \vv \cdot \na \th  = 0, \\
&\pa_t \vv + \uu \cdot \na \vv + \vv \cdot \na \uu + \na q = -  (0, \eta)^T, \quad  \mathrm{div} \ \vv = 0.  \\
\eal
\eeq
Denote ${\bf w} = (\eta, \vv)$ and define
\[
\g^{sym}_{p}(T) = \sup_{  || {\bf w}_0||_{L^{p}} \leq 1 } || {\bf w} ||_{L^p} , \quad  \mathrm{ with } \
|| {\bf w} ||_{L^p} \sim || \eta ||_{L^p} + || \vv ||_{L^p}, 
\]
with the symmetry property that $v_1(x, y)$ is odd in $x$ and $v_2(x, y), \eta_0$ are even in $x$.

We have the following instability result for the singular solution constructed in \cite{chen2019finite2}.

\begin{thm}\label{thm:bous}
There exists $\al_0 > 0$ such that for $0< \al < \al_0$, the 2D Boussinesq equations \eqref{eq:bous} in $D = \R_2^+$ develops a singularity at finite time $T_*$ from some initial data $ \om_0 \in C_c^{\al}( \R^2_+),  \th_0 \in C_c^{1,\al}(\R^2_+) $. The initial data satisfy that $\om_0(x, y)$ is odd in $x$, $\th_0(x, y)$ is even in $x$, and $ \uu_0$ has finite energy $||\uu_0||_2 < +\infty$. Moreover, the solution satisfies $(\uu, \th ) \in L^{\inf}( [0, T], C^{50}( \S) ) $ for any $T < T_*$ and any compact domain $\S \subset \{ (x, y): x \neq 0, y > 0 \}$. For any $ p \in (1, \inf)$, we have 
\[
\lim_{t \to T_*}  \g^{sym}_p(t) = \inf.
\vspace{-0.05in}
\]
\end{thm}

For the same $C^{1,\al}$ blowup solution to these equations in Theorems \ref{thm:Euler_HL}-\ref{thm:bous}\;, stability of the asymptotically (or approximately) self-similar blowup profile has been established in \cite{elgindi2019finite,chen2019finite2,elgindi2019stability} using the dynamic rescaling formulation \cite{mclaughlin1986focusing,  landman1988rate} or the modulation technique  \cite{merle1997stability,kenig2006global}.

\subsection{Comparison of the stability and instability results}
Given that the same blowup solution of the 3D Euler equations can be both linearly unstable under one definition and nonlinearly stable under a different definition, it is important to have a better understanding how we define stability and how to quantify instability. 
First of all, we would like to emphasize that the instability results in Theorems \ref{thm:Euler_HL}-\ref{thm:bous} measure the \textit{absolute} instability, i.e. the growth of the perturbation relative to the initial perturbation. The fact that instability develops using the linearized equation is not due to the violation of certain symmetry conditions for the perturbation. In fact, the perturbation in Theorems \ref{thm:Euler_HL}-\ref{thm:bous} satisfies the same symmetry as the blowup solution, e.g., \eqref{eq:sym}.
This rapid growth is not surprising since the background singular solution $\uu$ blows up and contributes to a singular forcing term $\vv \cdot \na \uu$ to the linearized equations \eqref{eq:euler_lin}. Such instability is quite common in several nonlinear PDEs. In Section \ref{sec:Riccati}, we will use a nonlinear PDE of Riccati type and the inviscid Burgers' equation to show that a similar forcing term generates linear instability for these equations. 

Since we consider the stability of a blowup solution, we believe that it is more reasonable to study the \textit{relative} stability or instability, which measures the relative growth of the perturbation compared with the growth of the background singular solution. 
In Section \ref{sec:Ric_PDE}, we use a nonlinear PDE of Riccati-type to illustrate that the blowup profile is very unstable when we compare the growth of its perturbation relative to its initial perturbation using the linearized equation similar to \eqref{eq:euler_lin}. In fact, the growth rate of the perturbation can be much faster than that of the background blowup solution. On the other hand, by incorporating the changes of the blowup time, the blowup profile and the blowup exponent via the dynamic rescaling formulation, we can establish the nonlinear stability of the blowup solution. This stability results show that the relative growth of the perturbed profile compared with the growth of the background blowup profile remains small up to the blowup time.


More importantly, the nonlinear stability results presented in \cite{elgindi2019finite,chen2019finite2,elgindi2019stability}
quantify the \textit{relative} stability: for a small initial perturbation to the blowup profile, some weighted norm $X$ of the perturbation remains \textit{relatively} small up to the blowup time. These estimates and the embedding inequalities imply that the growth of the perturbation of the vorticity $|| \td \om||_{L^{\inf}}$ remains much smaller than the growth of the blowup solution $||  \om||_{L^{\inf}}$ up to the blowup time. The $L^{\inf}$ norm of the vorticity is of fundamental importance since it controls the blowup of \eqref{eq:euler}. 
Moreover, this stability result implies that for a small initial perturbation, the change of the blowup time $T_*$ is very small. Thus, if a blowup solution has stability similar to that obtained in \cite{elgindi2019finite,chen2019finite2,elgindi2019stability}, one can perform reliable numerical computations to provide compelling evidence of finite time blowup despite unavoidable numerical errors \cite{luo2013potentially-2,luo2014potentially,hou2022potential,Hou-euler-2021,Hou-nse-2021}.

Studying stability of the blowup based on the self-similar variables, dynamic rescaling formulation, or the modulation technique has been used in many other equations, such as the nonlinear heat equations \cite{merle1997stability}, the Burgers' equation \cite{collot2018singularity}, the complex Ginzburg-Landau equation \cite{plechavc2001self,masmoudi2008blow}, the nonlinear Schr\"odinger equation \cite{merle2005blow}, the generalized KdV equation \cite{martel2014blow}, and the compressible Euler equations \cite{buckmaster2019formation,buckmaster2019formation2}. 
On the other hand, there are also some instability results of the blowup based on these approaches. For example, the authors in \cite{collot2018singularity} proved that many blowup profiles of the 1D Burgers' equation have a finite number of unstable directions. See also the blowup of the nonlinear Schr\"odinger equation \cite{merle2022blow} and the blowup of compressible fluids \cite{merle2019implosion} with finite many potential unstable directions.






%
\vspace{-0.05in}
\subsection{Main ideas in the instability analysis}
%
There are several main ideas in proving the main instability results stated in Theorems \ref{thm:Euler_HL}, \ref{thm:Euler_R3}. One of the main difficulties in proving Theorems \ref{thm:Euler_HL}, \ref{thm:Euler_R3} is to relax the regularity assumptions in the arguments \cite{lafleche2021instability,vasseur2020blow} by using the properties of the singular solutions in \cite{elgindi2019finite,chen2019finite2}. 
We then construct an axisymmetric approximate solution to \eqref{eq:euler_lin} and follow the arguments in \cite{lafleche2021instability} to prove the main theorems. 

For the 2D Boussinesq equations, we use ideas similar to the 3D Euler equations to relax the regularity assumption in \cite{shao2022instability} and then apply the argument in \cite{shao2022instability} to prove Theorem \ref{thm:bous}.

\vspace{0.1in}
\noindent
{\bf Relaxing the regularity assumption.}
 In \cite{lafleche2021instability}, the regularity assumption $\uu \in C^0( [0, T), H^s) \cap C^1( [0, T) , H^{s-1}) $ with $s > \f{7}{2}$ is to ensure 
 
(a) 
the solvability of the bicharacteristics-amplitude ODE system \cite{lafleche2021instability,vasseur2020blow,friedlander1991dynamo};

(b) that the poloidal component of the vorticity $ \om_p = \om^r e_r + \om^z e_z$ satisfies $ \f{1}{r^a}  \om_p \in L^{\inf}$ for some $a > 0$, which is used in \cite{lafleche2021instability} to connect the blowup criteria with the instability.

To relax the regularity assumption for (a), we make an important observation that the singular solution $\uu$ constructed in \cite{elgindi2019finite,chen2019finite2} is smooth away from the symmetry axis and the boundary. The $C^{1, \al}$ low regularity is used essentially near the singularity, the symmetry axis, and the boundary to weaken the advection. The higher-order interior regularity of the solution $\uu$ can be propagated by using careful higher-order weighted energy estimates and the elliptic estimates with weights degenerated near the symmetry axis and the boundary \cite{elgindi2019finite,chen2019finite2}. In particular, in a compact interior domain, the weighted energy norms are comparable to the standard Sobolev norms, which allows us to establish higher-order interior regularity of the solution using the embedding inequalities. See Theorems \ref{thm:bous_blowup}-\ref{thm:euler_R3_blowup}.

Using the higher-order interior regularity, we can solve the bicharacteristics-amplitude ODE system, which is local in nature, in the interior of the domain and construct smooth solution to the modified bicharacteristics-amplitude ODE system. See Lemma \ref{lem:traj} and Proposition \ref{prop:lam}\;.

\begin{remark}\label{rem:oversight}
In \cite{chen2019finite2}, we proved the blowup results for the 3D axisymmetric Euler equations with initial data $(u_0^{\vth})^2,  u_0^r, u_0^z \in C^{1,\alpha}$ and $\omega_0^{\vth} \in C^{\alpha}$. Though the velocity $u^r, u^z$ in the axisymmetric setting is $C^{1,\alpha}$, our interpretation that the velocity is $C^{1,\alpha}$ is not correct since $u^{\vth}$ is not $C^{1,\alpha}$. This oversight can be fixed easily with minor changes in the construction of the approximate steady state and the truncation of the approximate steady state. 
These changes do not affect the nonlinear stability estimates of the 3D Euler equations, 
see the update arXiv version of \cite{chen2019finite2}.
\end{remark}

\noindent
{\bf Blowup quantities.}
An important step in \cite{lafleche2021instability,vasseur2020blow} is to show that the growth factor $\lam_{p, \s}$ \eqref{eq:instab1} controls $|| \om||_{\inf}$, which blows up for a singular solution \cite{beale1984remarks}.
The singular solutions in \cite{elgindi2019finite,chen2019finite2} are self-similar 
or approximately self-similar. In addition to $|| \om||_{\inf}$,
there are several other blowup quantities. By comparing some of these blowup quantities and the growth  factor $\lam_{p, \s}$ \eqref{eq:instab1}, we can simplify the proof in \cite{lafleche2021instability} and further relax some constraints. 
For example, in the proof of Theorem \ref{thm:Euler_HL}, we use the property that $|| \om_p||_{\inf}$ (the poloidal component) blows up and thus do not rely on the blowup criterion 
on $|| \om_p  / r^a||_{\inf}$ for some $a > 0$ established in \cite{lafleche2021instability}. 
This relaxes the condition (b). 


The singularity considered in \cite{elgindi2019finite} develops near the axis $r = 0$ and has zero swirl $u^{\th}\equiv 0$, which implies $\om_p \equiv 0$. 
Thus we cannot follow the argument in \cite{lafleche2021instability} to prove Theorem \ref{thm:Euler_R3}. Instead, we use the bicharacteristics-amplitude ODE system and the flow structure near the singularity in \cite{elgindi2019finite} to show that the growth $\lam_{\s, p}(t)$ controls another blowup quantity.



\vspace{0.1in}
\noindent
{\bf Axisymmetric velocity.}
Another important step in proving Theorems \ref{thm:Euler_HL} and \ref{thm:Euler_R3} is to construct 
an axisymmetric solution to \eqref{eq:euler_lin}. We remark that the initial data of \eqref{eq:euler_lin} constructed in \cite{lafleche2021instability} is not axisymmetric under the canonical notion \eqref{axi} \cite{majda2002vorticity}, see Remark \ref{rem:non_axi} for more discussions.  
We use the PDE (Eulerian) form of the 
bicharacteristics-amplitude ODE system to construct the amplitude $b(t, x)$ and the phase $S(t, x)$ in the WKB construction of the approximate solution to \eqref{eq:euler_lin}. 
The initial data $b(0, x), \xi(0, x)$ are axisymmetric flows in the whole domain, which are constructed by extending some constant initial data $b_0 , \xi_0 = \na S_0 \in \R^3$  of the bicharacteristics-amplitude ODE system.
The axisymmetry properties of $b(t, x), \xi(t, x)$ are preserved dynamically by the equations. 
We further show that $b(t, x)$ controls the solution to the bicharacteristics-amplitude ODE system and captures the growth of the vorticity.
Based on these functions, we construct the axisymmetric velocity using the formula in \cite{lafleche2021instability,vasseur2020blow}.


%
%
%
\vspace{0.1in}
\noindent
 {\bf Symmetry of the unstable solution.}
The singular solutions constructed in \cite{chen2019finite2,elgindi2019finite} are symmetric with respect to some axis, e.g., \eqref{eq:sym}, and the flow does not cross the symmetry axis or the symmetry plane.
This allows us to first construct unstable solution in the upper half domain following \cite{lafleche2021instability}, and then extend it naturally to a symmetric solution to the linearized Euler equations using linear superposition. 
Therefore, we can further restrict the perturbation in \eqref{eq:instab1} to the natural symmetry class.


\vspace{0.1in}
The rest of the paper is organized as follows. 
In Section \ref{sec:Riccati}, we use several nonlinear PDEs, including a simple nonlinear PDE of Riccati type and the inviscid Burgers' equation, to demonstrate the difference between the notion of stability introduced in \cite{vasseur2020blow,lafleche2021instability} and the stability based on dynamically rescaling formulation or modulation technique. Section \ref{sec:main} is devoted to prove the main theorems of this paper. Some important properties that we use in proving the main theorems will be established for the 2D Boussinesq equations in Section \ref{sec:high} and for the 3D Euler equations in Sections \ref{sec:euler_blowup}, respectively. Some technical lemmas are deferred to the Appendix.

\section{Comparison of stability vs instability through several nonlinear PDEs}
\label{sec:Riccati}

In this section, we will use several examples to demonstrate that under the notion of stability introduced in \cite{vasseur2020blow}, linear instability of a blowup solution is quite common in several nonlinear equations, even for those nonlinear equations whose blowup solutions can be shown to be nonlinearly stable using a suitable functional space and the dynamic rescaling formulation.

\subsection{The 3D Euler equations}\label{sec:Euler_unstable}

We first consider the 3D Euler equations. Suppose that $\uu(x, t)$ is a singular solution of the 3D Euler equations that blows up at a finite time $T$ with $|| \uu||_{L^2} < +\inf$. Clearly, we have $\pa_{i} \uu_0 \neq 0$ for all $i$. If $\pa_{i} \uu_0 \equiv 0$ for some $i$, the initial velocity $\uu_0$ would have reduced to the two dimensional Euler equations, which could not blow up in a finite time. 

For a domain without boundary, e.g. $\BT^3$ or $\R^3$, the linearized equation \eqref{eq:euler_lin} has exact solutions $\vv = \pa_i \uu $ for $i=1,2,3$, which was observed in \cite{vasseur2020blow} for the Navier Stokes equations. Suppose that $X$ is some functional space equipped with a norm that is stronger than the $L^{\inf}$ norm, e.g. $X = L^{\inf}, C^{k, \al}, k \geq 0, \al \in (0, 1)$, or $X = H^s, s > \f{3}{2}$, and it satisfies $ \na \uu_0 \in X$. Since $\int_0^t ||\na \uu(s)||_{\inf} ds$ controls the blowup of the solution, we obtain 
\[
 \infty = \limsup_{t \to T}  \sum_{i=1}^3 \f{ || \pa_i \uu (t)||_{L^{\inf}} }{ || \pa_i \uu_0||_X}
 \les  \limsup_{t \to T}  \sum_{i=1}^3 \f{ || \pa_i \uu (t)||_{X}}{ || \pa_i \uu_0||_X}
 \les   \limsup_{t \to T} \sup_{ \vv_0 \in X, \vv_0 \neq 0}  \f{ || \vv(t)||_{X}}{ || \vv_0||_X}.
\]
Under the notion of stability introduced in \cite{vasseur2020blow}, the blowup is linearly unstable in the norm of $X$. Yet, this instability result is a direct consequence of the blowup criterion and does not use further properties of the blowup solution, e.g., the blowup profile and the blowup exponent.

\subsection{1D models for the 3D Euler equations}

Consider the De Gregorio model \cite{DG90,DG96} and the generalized Constantin-Lax-Majda model \cite{OSW08}
\beq\label{eq:gCLM}
\om_t + a  u \om_x = u_x \om , \quad  u_x(\om) = H \om , \quad x \in \R  \mathrm{ \ or \ } S^1,
\eeq
where $H$ is the Hilbert transform and $a$ is a parameter. If $a=1$, \eqref{eq:gCLM} becomes the De Gregorio model. We consider the following linearized equation for a singular solution $\bar \om(t)$ that develops a finite time singularity at $T$
\beq\label{eq:gCLM_lin}
\pa_t \om + a \bar u \om_x + a u  \bar \om_x = \bar u_x \om + u_x \bar \om, \quad u_x = H \om.
\eeq
It is easy to see that $ \om = \pa_x \bar \om$ is a solution to \eqref{eq:gCLM_lin}. Following \cite{vasseur2020blow}, we introduce the growth factor 
\beq\label{eq:gCLM_lam}
\lam_p(t) \teq \sup_{ \om_0 \in L^p, \om_0 \neq 0}  \f{ || \om(t)||_p}{ || \om_0||_p}, \quad p \in (1, \inf).
\eeq

For $a=1$, in a joint work with Huang \cite{chen2019finite}, we constructed a finite time blowup of the De Gregorio model (\eqref{eq:gCLM} with $a=1$) from $C_c^{\inf}$ initial data. The singular solution satisfies 
\[
\bar \om(x, t) = C_{\om}(t)^{-1} \Om( C_{\om}(t) x, t) ,  \quad C_{\om}(0) = 1, \quad  \lim_{t \to T} C_{\om}(t) = 0,  \quad \Om(x, t) = \bar \Om(x) + \td \Om(x, t),
\]
where $C_{\om}(t)$ is decreasing, $\Om( \cdot, t) \in C^{\inf}$, $\bar \Om$ is the approximate self-similar profile, and $\td \Om(x, t)$ is a small perturbation. In particular, the estimates in \cite{chen2019finite} imply
\[
| \td \Om(x, t)| \les  |x|^{3/2}, \quad   |\bar \Om(x) - A x | \les x^2
\]
for some $A \neq 0$, where the implicit constants are time-independent. Therefore, for some small $\d > 0$, we get $ |\Om(\d, t)| \geq \f{A}{2} \d > 0$ for $t\in [0, T)$. For any $p\in [1, \inf)$, we obtain
\[
 || \pa_x \bar \om ||_{L^p} = C_{\om}(t)^{-1/p} || \Om_x(\cdot, t)||_p
 \gtr_p  C_{\om}(t)^{-1/p} \int_0^{\d} |\Om_x( y , t)|  dy
 \gtr_p C_{\om}(t)^{-1/p} | \Om( \d, t) |
 \gtr_p C_{\om}(t)^{-1/p}  A \d .
\]
Since $ || \pa_x \bar \om_0 ||_{L^p} \neq 0$ and $C_{\om}(t) \to 0$ as $t \to T$, we yield $\lam_p(t) \to \infty$ \eqref{eq:gCLM_lam} as $t \to T$. 

Similarly, we can obtain that the smooth blowup solutions of \eqref{eq:gCLM} on $\R$ with small $|a|$ \cite{chen2019finite,Elg17}, on $\R$ with $a$ close to $\f{1}{2}$ \cite{chen2020singularity}, and on $S^1$ with $a$ slightly less than $1$ \cite{chen2020slightly} are linearly unstable in the $L^p$ norm with $p \in [1, \inf)$ under the notion of stability introduced in \cite{vasseur2020blow}.

On the other hand,  nonlinear stability of these blowup solutions in some weighted $H^1$ norms has been established in \cite{chen2019finite,chen2020singularity,chen2020slightly,Elg19} using the dynamic rescaling formulation \cite{mclaughlin1986focusing,landman1988rate}. The nonlinear stability in \cite{chen2019finite,chen2020singularity,chen2020slightly,Elg19} is established by analyzing the stability of the asymptotically (or approximate) self-similar blowup profile, which is very different from the linear stability in \cite{vasseur2020blow}.

Similar discussions on the stability of the blowup solution in the dynamic rescaling equations and the instability of the blowup solution in the linearized equation apply to the singular solution of De Gregorio model on $S^1$ \cite{chen2021regularity} and of the Hou-Luo model \cite{chen2021HL}.


\subsection{A nonlinear Riccati PDE and the inviscid Burgers' equation}

In the next two subsections, we consider the blowup solutions of the 
inviscid Burgers' equation
\beq\label{eq:Burger}
\pa_t u + u u_x = 0 , \quad x \in \R
\eeq
and a nonlinear PDE of Riccati type
\beq\label{eq:Ric}
 \pa_t u (t, x) = u^2(t, x), \quad x \in \R.
\eeq
We will show that the blowup solutions of these two nonlinear PDEs are unstable under the notion of stability introduced in \cite{vasseur2020blow}. See Theorems \ref{thm:burger}, \ref{thm:Ric_instab}. On the other hand, using the dynamic rescaling formulation, we can prove the nonlinear stability of the blowup solutions to \eqref{eq:Ric} in Theorem \ref{thm:Ric_stab}. 
In Section \ref{sec:Ric_PDE}, we will use \eqref{eq:Ric} to illustrate the importance of studying the stability of the asymptotically (or approximate) self-similar blowup profile using suitable rescaling and renormalization rather than studying the stability of the blowup solution itself.


Following \cite{vasseur2020blow}, we define the growth factor 
\beq\label{eq:Ric_instab}
\lam_p(t) = \sup_{ v_0 \neq 0, v_0 \in L^p } \f{ || v(t)||_{L^p}}{ || v_0||_{ L^p}}
\eeq
for the solution $v$ to the linearized equations of \eqref{eq:Burger} or \eqref{eq:Ric} around a singular solution.



We first consider the Burgers' equation. It is well-known that \eqref{eq:Burger} blows up (develops a shock) in finite time $T$ for initial data $u_0 \in C_c^{\inf}$ satisfying $ u_0(0) =0$ and that $\pa_x u_0$ is minimal at $0$ with $\pa_x u_0(0) < 0$. Let $v$ be a solution to the linearized equation of \eqref{eq:Burger} around the blowup solution $u$
\beq\label{Burger_lin}
v_t + \pa_x(  u v) = v_t +  u v_x + u_x v =  0.
\eeq
It has been shown in \cite{vasseur2020blow} that the blowup is linearly stable in $L^1$ in the sense that $\lam_p(t) \leq 1$ \eqref{eq:Ric_instab} up to the blowup time. However, this stability result does not generalize to $L^p$ with $p> 1$. In particular, we have the following instability result. 
\begin{thm}\label{thm:burger}
Suppose that the initial data $u_0 \in C^{1}$ of \eqref{eq:Burger} satisfies that $u_0(0) = 0$,  $\pa_x u_0$ is minimal at $0$ with $\pa_x u_0(0) < 0$. Then the solution $u$ blows up in finite time $T_* = - \f{1}{ u_{0, x}(0)}$. Moreover,  for any $p \in (1, \inf)$, we have
\[
\lim_{t \to T_*} \lam_p(t) \to \infty.
\]
\end{thm}

Since the linearized equation \eqref{Burger_lin} contains a singular forcing term $u_x v$, it is not surprising that $v(t)$ can blow up in some $L^p$ norm. In the following proof, since the equation is local, we localize the perturbation $v$ to the region where $-u_x$ blows up to show that $v$ can grow rapidly. 

We remark that the stability of the blowup of \eqref{eq:Burger} has been studied in details in \cite{collot2018singularity} using the modulation technique. The stability of the blowup of \eqref{eq:Burger} has been used to establish shock formation in the 2D and the 3D compressible Euler equations \cite{buckmaster2019formation,buckmaster2019formation2}.


\begin{proof}
Fix $T <  T_* = - \f{1}{u_{0,x}(0)} = \f{1}{ |u_{0,x}(0)|}$. 
It is easy to obtain that $u(t, 0) = 0$ for any $ t < T_*$. Note that $u_x(t, 0)$  satisfies the ODE
\beq\label{eq:Burger_ux0}
\pa_t u_x(t, 0) = - ( u_x(t, 0) )^2, \quad u_{0,x}( 0) < 0 , \quad u_x(t, 0) = \f{ - |u_{0, x}(0)| }{ 1 - t |u_{0, x}(0)| } = - \f{1}{T_*- t},
\eeq
where we have used $T_* =  | u_{0, x}(0)|^{-1}$ in the last equality. It follows the blowup result and $u_x(t, 0) \leq u_{0,x}( 0) < 0 $. Since $u(t) \in C^0( [0, T], C^1)$, there exists $\d >0$ such that 
\beq\label{eq:Burger_ux}
u_x(t, x) \geq  - \f{1}{2} u_x( t, 0 ) > 0 , \quad x  \in  [-\d, \d],
\eeq
for any $t \leq T$, which implies 
\beq\label{eq:Burger_flow}
 u(t, x) \leq 0 , \ x \in [0, \d] , \quad u(t, x) \geq 0, \ x \in [-\d, 0],
\eeq
for any $t \leq T$. Consider $v_0 \in C^{\inf}, v_0 \neq 0, \supp(v_0) \subset [-\d, \d]$. Due to \eqref{eq:Burger_flow},  $\supp( v(t)) $ remains in $[-\d, \d]$ for $t \leq T$. Performing $L^p$ estimate on \eqref{Burger_lin} and using integration by parts,  we obtain 
\[
  \f{1}{p} \f{d}{dt} || v||_{L^p}^p 
  = \int_{\R} - (u v)_x \cdot |v|^{p-2} v  dx 
= \int_{\R} - u_x |v|^p -  u v_x  |v|^{p-2} v  dx 
= \int_{\R} - u_x |v|^p + \f{1}{p}  u_x  |v|^p  dx .
\]
Since $\supp( v(t)) \subset [-\d, \d]$, using \eqref{eq:Burger_ux}, we further obtain 
\[
\f{1}{p}\f{d}{dt} || v||^p_{L^p} = (1- \f{1}{p}) \int_{[-\d, \d]} - u_x |v|^p   dx \geq
(1 - \f{1}{p}) \f{ - u_x(t, 0)}{2} \int_{[-\d, \d]}  |v|^p   dx 
= (1 - \f{1}{p}) \f{ - u_x(t, 0)}{2} || v||_{L^p}^p.
\]
Solving the above ODE and using \eqref{eq:Burger_ux0}, we prove 
\[
\bal
|| v(T)||_{L^p} &\geq || v_0||_{L^p} \exp\B( \f{1}{2}(1 - \f{1}{p})\int_0^T - u_x( t, 0)  dt \B)
=  || v_0||_{L^p} \exp\B( - \f{1}{2}(1 - \f{1}{p}) \log(T_* - T) \B). \\
\eal 
\]
From the definition of $\lam_p(t)$, we yield 
\[
\lam_p(T) \geq \exp\B( - \f{1}{2}(1 - \f{1}{p}) \log(T_* - T) \B) =(T_* - T)^{ - \f{1}{2}(1 - \f{1}{p}) }.
\]
Since $p > 1$, taking $T \to T_*$, we obtain the desired result.
\end{proof}

\subsection{The Riccati-type PDE}\label{sec:Ric_PDE}

It is easy to show that if the initial data $u_0$ of \eqref{eq:Ric} satisfies $\max (u_0)  > 0$, the PDE blows up at finite time $T(u_0) = \f{1}{ \max (u_0)}$. Moreover, the equation can develop a self-similar blowup
\beq\label{eq:Ric_self}
 \bar u (t, x) = \f{1}{1 - t+ x^2} = \f{1}{1-t} \bar U( \f{ x}{ (1 -t)^{1/2} }), \quad \bar U = \f{1}{1 + x^2}. 
\eeq

The linearized equation around the blowup solution $\bar u$ \eqref{eq:Ric} reads
\beq\label{eq:Ric_lin}
\pa_t v = 2 \bar u v .
\eeq

Denote $P_{\e}$
\beq\label{eq:Ric_init}
 P_{\e} \teq\{ u : u = C( \bar U + V_0 ), \ C> 0 , \  |  V_0 | \leq \e \min(1, |x|^3) \}.
\eeq
We will study the stability of the blowup solution of \eqref{eq:Ric} for initial data in $P_{\e}$. Let us motivate the class $P_{\e}$. 
For initial data $u_0$ close to \eqref{eq:Ric_self}, we have $u_0(x) = u_0(0) u_1(x)$ with $u_1(0) = 1$ and $u_1$ being a perturbation of $\bar U$. Since the solution $u$ first blows up at $\arg\max u_0$ and $ \bar U(x) = 1 - x^2 + O(x^4)$ near $x=0$, we require that $V_0$ vanishes to higher order $O(|x|^3)$ near $x=0$ and $\e$ is small so that the maximum of $u_0$ does not shift. The vanishing order can be relaxed to $ |v| \les |x|^{2+ \d} $ for any $\d > 0$, and the stability result similar to that in Theorem \ref{thm:Ric_stab} holds.

To further study the instability of the blowup profile $\bar U$ \eqref{eq:Ric_self} to \eqref{eq:Ric}, 
we consider the following ansatz of the linearized solution \eqref{eq:Ric_lin} and the rescaled growth factor $\Lam_p(t)$ similar to that for the 3D Euler equations in \cite{lafleche2021instability} 
\beq\label{eq:Ric_ansatz}
v(t, x) = \f{1}{1 - t} V( \f{x}{ (1-t)^{\b} }, t), \quad \b = \f{1}{2},  \quad \Lam_p( v, t) = \f{ || V(t)||_{L^p} }{ || V_0||_{L^p}}.
\eeq
Since the blowup exponent $\f{1}{1-t}$ is factored out,  $\Lam_p$ can be seen as measuring the \textit{relative} linear instability between $V$ and the background profile $\bar U$ \eqref{eq:Ric_self}, while $\lam_p$ \eqref{eq:Ric_init} measures the \textit{absolute} linear instability. We have the following instability results.



\begin{thm}\label{thm:Ric_instab}
For any $ v_0 \in C_c^0$ with $v_0(0)> 0$ and any $p \in [1, \infty]$, we have 
\[
|| v(t) ||_{L^p} \gtr C(v_0, p) (1- t)^{-2 + \f{1}{2p} }, \quad 
\lim_{ t \to 1}  || v(t)||_p = \infty, \quad  \lim_{t \to 1} \f{ || v(t)||_{L^p}}{ || u(t) ||_{L^p} } = \infty.
\]
As a result, we have $\lam_p(t) \to \infty, \Lam_p(v, t) \to \infty$ as $t \to 1$. 
\end{thm}

In the above theorem, we can choose perturbation $v_0$ with $u_0 =  \bar u + v_0 \in P_{\e}$ \eqref{eq:Ric_init}. On the other hand, we can prove stability of the blowup for $u_0 \in P_{\e}$ in Theorem \ref{thm:Ric_stab}. 

We remark that the above instability results are not surprising since $\bar u$ in the forcing term $\bar u v$ \eqref{eq:Ric_lin} blows up. The problems of using the ansatz \eqref{eq:Ric_ansatz} to study the stability of the blowup profile $\bar U$ \eqref{eq:Ric_self} are the following. For initial data $u_0$ perturbed from $\bar u$, we expect that the blowup time $T$ changes and the blowup exponent $\b$ in \eqref{eq:Ric_ansatz} can also change. Moreover, to observe the blowup profile, we need to rescale the solution using a different rescaling rate in the spatial variable. These lead to the following ansatz of the singular solution $u$ from initial data $u_0$ near $\bar u$
\beq\label{eq:Ric_ansatz2}
 u(x, t) = \f{1}{\td T- t} U\B(   \mu   \f{x}{ (\td T - t)^{ \td \b }}, t \B), \quad \td \b \approx \f{1}{2}, \quad \td T \approx 1, \quad \mu \approx 1 , \quad U \approx \bar U.
\eeq

However, in \eqref{eq:Ric_ansatz}, the parameter $\td T, \td \b,  \mu$ are all fixed. Moreover, in the above ansatz, due to the composition, the parameters $\td \b, \td T, \mu$ depend on the initial data and perturbation in a nonlinear fashion. Thus, they cannot be captured by the linearized equation \eqref{eq:Ric_lin} around $\bar u$. Without incorporating the perturbation of these parameters, it is not expected to observe the stability of the profile.

\begin{remark}
There is some progress on modulating the instability caused by the change of the blowup time $T(u)$ using the Calkin semi-norm \footnote{
Pages 9 and 13 in the slide of the talk "Instability of finite time blow-ups for incompressible Euler" 
by Alexis Vasseur in New Mechanisms for Regularity, Singularity, and Long Time Dynamics in Fluid Equations. \url{https://www.birs.ca/events/2021/5-day-workshops/21w5110/videos/watch/202107301430-Vasseur.html}}
\beq\label{eq:calkin}
\lam_p(t)= \infim_{ K \mathrm{ \ compact \ operator} } \sup_{ || v_0||_{L^p} \leq 1 } || v(t) - K v_0 ||_{L^p}. 
\eeq
Suppose that $\uu$ is a solution that blows up at $T_*$, and $u_{\e}$ is another solution from a perturbed initial data $ \uu_{\e}^0 = \uu + \e v_0$ and blows up at $T_{\e}^*$.
Considering the change of the blowup time, if $\e$ is very small, one would expect that the solution $v(t)$ to \eqref{eq:euler_lin} takes the form 
\[
v(t, x) \approx \f{ \uu_{\e}(t, x) - \uu(t + T^* - T_{\e}^* , x) }{\e}
= \f{ \uu_{\e}(t, x) - \uu(t, x)}{\e} + \f{ \uu(t, x) - \uu(t + T^* - T_{\e}^* , x) }{\e}
\teq I + II.
\]
The first part measures the change of the profile, and the second part measures the effect due to the change of the blowup time. Using the Taylor expansion, for small $\e$, one has
\[
II \approx  c[v_0] \pa_t \uu , \quad  c[v_0] \approx \f{ T_{\e}^*  - T^*}{\e} .
\]

If $\uu$ blows up with a rate $(T_* - t)^{-\al}$, then $\pa_t \uu$ can blow up even faster with a rate $(T_* - t)^{-\al-1}$. This can lead to the fast growth of $v(t, x)$. If $c[v_0]$ is a linear operator of $v_0$, since $\pa_t \uu(t, x)$ is a given function, then $c(v_0) \pa_t \uu$ is a rank-one operator. Therefore, to modulate this effect, the speaker proposed to subtract $v(t)$ by $K v_0$ for a compact operator $K$, and then $v(t) - K v_0$ measures the more interesting quantity $I$. See more details from the link in the footnote.

However, since the blowup time $T(u_0)$ depends nonlinearly on the initial data $u_0$, $c[v_0]$ may not be linear in $v_0$. For example, in \eqref{eq:Ric}, the blowup time is  $T(u_0) = \f{1}{ \max( \max u_0, 0) }$. 
Consider initial data $\bar u_0$ with $\bar u_0(0) = 1$ and $\bar u_0(x) < 1$ for $x \neq 0$. 
For any perturbation $v_0$, we yield 
\[
  \lim_{\e \to 0}  \f{  - T( \bar u_0 + \e v_0) + T( \bar u_0)}{\e} 
  = \lim_{\e \to 0}   \f{ \max( \bar u_0 + \e v_0)  - 1}{ \e \cdot \max( \bar u_0 + \e v_0) }
  = \lim_{\e \to 0} \f{   \max( \bar u_0 + \e v_0)  - 1 }{\e} 
  = \max( v_0(0), 0).
\] 
The functional $ \max( v_0(0), 0)$ is a nonlinear functional of $v_0$. In general, the blowup time $T$ and the corresponding functional can depend on the initial data in a nonlinear fashion, and thus the modification using the Calkin semi-norm \eqref{eq:calkin} may not remove the instability due to the change of the blowup time. 
\end{remark}

Using the dynamic rescaling formulation, we prove the stability of the blowup of \eqref{eq:Ric}.

\begin{thm}\label{thm:Ric_stab}
There exists an absolute constant $\e > 0$, such that for any $u_0 \in P_{\e} \cap L^{\inf}$, we have
\beq\label{eq:Ric_self2}
u( x, t(\tau)) = \f{1}{ T - t(\tau)} U\B( T^{1/2} \f{ x}{ (T- t(\tau))^{1/2} } , \tau \B) , \quad T = \f{1}{C} = \f{1}{u_0(0)},  \quad t(\tau) =  T( 1 - e^{-\tau}),
\eeq
for any $\tau \in [0, \infty)$. Moreover, we have the following stability estimate
\beq\label{eq:Ric_stab}
     || ( U(\cdot, \tau) - \bar U ) ( |x|^{-3} + 1)||_{L^{\inf}} 
     \leq      || ( U_0 - \bar U ) (|x|^{-3} + 1)||_{L^{\inf}} e^{- \f{\tau}{4}} .
\eeq
\end{thm}

The formula \eqref{eq:Ric_self2} and estimate \eqref{eq:Ric_stab} are consistent with the ansatz \eqref{eq:Ric_ansatz2}. For  initial data  $u$ different from $\bar u$ \eqref{eq:Ric_self}, we have a different blowup time $T$ and we need to adjust the rescaling rate $T^{1/2} (T-t)^{-1/2}$ in the spatial variable. To study the stability of the blowup profile, we rescale the spatial variable, the temporal variable, and normalize the amplitude of the solution according to the initial data. These rescaling relations and renormalization are nonlinear and thus are not captured by the ansatz \eqref{eq:Ric_ansatz} and the linearized equation  \eqref{eq:Ric_lin}.

We first prove Theorem \ref{thm:Ric_instab} and then Theorem \ref{thm:Ric_stab}\;.
\begin{proof}[Proof of Theorem \ref{thm:Ric_instab}]
Recall $\bar u = \f{1}{1 - t+ x^2}$ from \eqref{eq:Ric_self}. Using \eqref{eq:Ric_lin}, we obtain 
\[
\bal
v(t, x) & = v_0(x) \exp\B( \int_0^t 2 \bar u( s, x) ds  \B)
= v_0(x) \exp\B( \int_0^t  2 \f{1 }{1-s + x^2} ds  \B) \\
& = v_0(x) \exp(  2 \log(1 + x^2) - 2 \log(1 - t + x^2) )
= v_0(x) \f{ (1 + x^2)^2}{ (1 - t + x^2)^2 }  \\
& = \td v(x) \f{1}{(1-t)^2} ( \bar U ( \f{x}{ (1-t)^{1/2}} ) )^2, \quad \td v = v_0(x) (1 + x^2)^2,
\eal
\]
where $\bar U$ is given in \eqref{eq:Ric_self}. In particular, $v$ blows up with a rate $(1-t)^{-2}$, which is even faster than that of $\bar u$. We remark that the exponent $2$ in $ (1-t)^{-2}$ is generic and does not relate to the coefficient $2$ in \eqref{eq:Ric_lin} or the formulation of \eqref{eq:Ric}. We obtain the same exponent if we consider $u_t = c u^2$ for other constant $c>0$ instead of \eqref{eq:Ric}. 

Since $v_0(0) > 0$ and $v_0 \in C_c^0$, there exists $c, \d >0$ such that $v_0(x) \geq c$ for $
|x| \leq \d$. For any $p \in [1, \infty)$, we have
\[
\int_{\R} |v|^p dx \geq c \int_{ |x| \leq \d} (1-t)^{-2p} ( \bar U ( \f{x}{ (1-t)^{1/2}} ) )^{2p} dx
= c (1-t)^{-2p+ 1/2} \int_{ |y| \leq \d (1-t)^{-1/2}}  | \bar U(y)|^{2p} d y.
\]

For $t$ sufficiently close to $1$, we get 
\[
|| v(t)||_{L^p} \gtr C(v_0, p) (1-t)^{-2 + \f{1}{2p}}.
\]

Recall $\bar u$ from \eqref{eq:Ric_self}. A direct calculation yields  $|| \bar u(t) ||_p = C_p (1-t)^{-1 + \f{1}{2p}}$ for some absolute constant $C_p>0$. For $p \in [1, \infty) $, these estimates imply the result in Theorem \ref{thm:Ric_instab}\;. For $p=\infty$, the calculation is even simpler and thus is omitted. 
\end{proof}

\begin{remark}
For smooth initial data $v_0$ with $v_0(0) > 0$, since $v(t)$ blows up even faster than $u(t)$, it is expected that the relative instability $ || v(t)||_X / || \bar u||_X$ occurs in many norms $X$, e.g., the Sobolev norms $W^{k, p}$ and the Holder norms $C^{k, \al}$. This relative instability is generic 
for \eqref{eq:Ric_lin}. Thus, using the linearized equation \eqref{eq:Ric_lin} around a blowup solution $\bar{\uu}$ is not suitable to study the stability of the profile \eqref{eq:Ric_self}. 
\end{remark}

\paragraph{\bf{Dynamic rescaling formulation}}

To study the stability of the blowup, we use the dynamic rescaling formulation. Suppose that $u$ is a solution to \eqref{eq:Ric}. Then it is easy to show that 
\beq\label{eq:Ric_res1}
U(x, \tau) = C_{\om}(\tau) u( C_l(\tau) x, t(\tau))
\eeq
is the solution to the dynamic rescaling equation 
\beq\label{eq:Ric_dyn}
\pa_{\tau} U + c_l x \pa_x U = c_{\om} U + U^2,
\eeq
where 
\beq\label{eq:Ric_res2}
  C_{\om}(\tau) = C_{\om}(0) \exp\lt( \int_0^{\tau} c_{\om} (s)  d \tau\rt), \ C_l(\tau) = \exp\lt( \int_0^{\tau} -c_l(s) ds    \rt) , \  t(\tau ) = \int_0^{\tau} C_{\om}( s) ds .
\eeq

We have the freedom to choose the scaling parameters $c_l, c_{\om}$ dynamically.

Note that a similar dynamic rescaling formulation was employed in \cite{mclaughlin1986focusing,  landman1988rate} to study the nonlinear Schr\"odinger (and related) equation. 
This formulation is closely related to the modulation technique, 
which has been developed by Merle, Raphael, Martel, Zaag and others, see e.g. \cite{merle1997stability,kenig2006global,merle2005blow,martel2014blow,merle2015stability}.
It has been a very effective tool to study singularity formation for many problems like the nonlinear Schr\"odinger equation \cite{kenig2006global,merle2005blow}, the nonlinear wave equation \cite{merle2015stability}, the nonlinear heat equation \cite{merle1997stability}, the generalized KdV equation \cite{martel2014blow}, compressible Euler equations \cite{buckmaster2019formation,buckmaster2019formation2}. Recently, it has been used to establish singularity formation in 
3D incompressible Euler equations \cite{chen2019finite2,elgindi2019finite,elgindi2019stability}, and related De Gregorio model \cite{chen2019finite,chen2021regularity}, the Hou-Luo model \cite{chen2021HL}, and the gCLM model \cite{chen2019finite,chen2020singularity,chen2020slightly,Elg19}.

\begin{proof}[Proof of Theorem \ref{thm:Ric_stab}]
Firstly, since $\bar U$ is a self-similar solution to \eqref{eq:Ric}, it is easy to see that 
$ \bar U , \bar c_l = \f{1}{2}, \bar c_{\om} = -1 $
is the steady state to \eqref{eq:Ric_dyn}. For any $u_0  = C ( \bar U + v ) \in P_{\e}$ \eqref{eq:Ric_init}, we choose $C_{\om}(0) = C^{-1} = u_0(0)^{-1}$. We summarize these parameters below
\beq\label{eq:Ric_para1}
\bar c_l = \f{1}{2} ,\quad \bar c_{\om} = -1, \quad C_{\om}(0) = C^{-1 } = u_0(0)^{-1}.
\eeq
Then from \eqref{eq:Ric_res1}, we get $U_0(x) = C_{\om}(0)u_0(x) = \bar U + V_0$. Denote $U(x, \tau) = \bar U(x) + V(x, \tau), c_l(\tau) = \bar c_l + \td c_l, c_{\om} = \bar c_{\om} + \td c_{\om}$. Substituting $U(x, \tau) = \bar U(x) + V(x, \tau)$ into \eqref{eq:Ric_dyn}, we obtain the equation for $V$ as follows:
\beq\label{eq:Ric_lin2}
\pa_{\tau} V + \bar c_l x \pa_x V + \td c_l x \pa_x \bar U + \td c_l x \pa_x V 
=\bar c_{\om} V +  \td c_{\om} \bar V + \td  c_{\om} V +  2 \bar U V + V^2,
\eeq
with $V(x,0) = V_0, |V(x, 0)| \les \min(1, |x|^3)$.
We choose normalization conditions on $c_l(\tau), c_{\om}(\tau)$ such that 
\[
V( 0, \tau) \equiv 0, \quad \pa_{xx} V(0, \tau) \equiv 0,
\]
for sufficiently smooth $V$. These requirements motivate the following conditions 
\beq\label{eq:Ric_normal}
\td c_{\om}(\tau)\equiv 0 , \quad \td c_l(\tau) \equiv 0. 
\eeq
Thus, we can simplify \eqref{eq:Ric_lin2} as follows
\[
\pa_{\tau} V + \bar c_l  x \pa_x V = \bar c_{\om } V + 2 \bar U V + V^2.
\]
The above equation implies that $\pa_{x}V(0, \tau) = 0$ is also preserved. Thus, we have $|V(x,\tau)| \les |x|^3$ near $x=0$. We choose $\rho = |x|^{-3} + 1$ and estimate $V \rho$
\[
\pa_{\tau} ( V\rho) + \bar c_l x\pa_x( V \rho) 
= \bar c_{\om} V \rho + 2 \bar U V \rho + \bar c_l x \pa_x \rho V + V^2 \rho
= ( \bar c_{\om} + 2 \bar U + \f{ \bar c_l x \rho_x}{\rho} )  \rho V + \rho V^2 \; .
\]

Denote $E = || V \rho ||_{L^{\inf}}$. Since $x \pa_x \rho = -3 |x|^{-3}$, a direct calculation yields 
\[
D \teq \bar c_{\om} + 2 \bar U + \f{ \bar c_l x \rho_x}{\rho}
= -1 + \f{2}{1 + x^2} + \f{1}{2} \cdot \f{-3 |x|^{-3}}{ |x|^{-3} + 1}  
= 1  - \f{2x^2}{1 + x^2} - \f{3}{2} + \f{3}{2} \cdot \f{|x|^3}{1 + |x|^3}.
\]
Using Young's inequality, we get $4 x^2(1 + |x|^3) - 3 |x|^3 (1 + x^2) = |x|^5 + 4 x^2  - 3 |x|^3
=|x|^5 + 2 x^2 + 2x^2 - 3|x|^3 \geq 3 (4)^{1/3} |x|^3 - 3 |x|^3 \geq 0$. It follows 
\[
 3 \cdot \f{|x|^3}{1 + |x|^3} \leq 4 \f{x^2}{1 + x^2}, \quad D \leq -\f{1}{2}.
\]
Performing $L^{\inf}$ estimate on $V \rho$ and using $ |V| \leq || V\rho||_{\inf} = E$, we establish 
\[
\f{1}{2} \f{d}{d \tau} || V \rho||_{L^{\inf}}^2
\leq -\f{1}{2}  || V \rho||_{L^{\inf}}^2 + || V\rho||_{L^{\inf}}^3, \quad 
\f{d}{d\tau} E(\tau) \leq (-\f{1}{2} +  E) E .
\]
Now, we pick $\e = \f{1}{8}$. Recall the assumption of the initial perturbation $V_0$ from \eqref{eq:Ric_init}. We yield $|V_0| (1 + |x|^{-3})
\leq \e \min(1, |x|^3)(1 + |x|^{-3}) \leq 2\e \leq \f{1}{4} $. Solving the above inequality, we prove
\[
\f{d}{d \tau} E(\tau) \leq -\f{1}{4} E(\tau), \quad  E(\tau) \leq E(0) e^{-\tau / 4},
\]
which is exactly the stability estimate \eqref{eq:Ric_stab}. 

Using \eqref{eq:Ric_para1} and \eqref{eq:Ric_normal}, we can compute the rescaling parameters \eqref{eq:Ric_res2}
\[
\bal
c_l(\tau) \equiv \f{1}{2}, \quad c_{\om}(\tau) \equiv -1, \quad C_{\om}(\tau) = C_{\om}(0) \exp(-\tau), \ C_l(\tau) = \exp( \f{ - \tau}{2}), \\
 t(\tau) =  \int_0^{\tau} C_{\om}(s) ds= C_{\om}(0) (1 - e^{-\tau}), \quad T = t(\inf) = C_{\om}(0) = \f{1}{u_0(0)}, \quad T - t(\tau) =  C_{\om}(\tau).
 \eal
\]
It follows $C_l(\tau) = (\f{ T-t(\tau)}{T} )^{1/2}$.
Plugging the above relations into \eqref{eq:Ric_res1}, we prove 
\[
u(x, t(\tau)) = C_{\om}(\tau)^{-1} U( C_l(\tau)^{-1} x, \tau)
= \f{1}{ T-t} U\B( \f{ x T^{1/2}}{ (T-t)^{1/2}} ,\tau \B), \quad T = \f{1}{ u_0(0)} ,
\]
which is exactly \eqref{eq:Ric_self2}.
\end{proof}

\section{Proof of main theorems}\label{sec:main}

In this Section, we first discuss several important properties of the singular solutions constructed in \cite{chen2019finite2,elgindi2019finite}. We will generalize the arguments and estimates in \cite{chen2019finite2} to prove some of these properties and defer the proofs to Sections \ref{sec:high} and \ref{sec:euler_blowup}. Using these properties of the blowup solutions, we will prove Theorems \ref{thm:Euler_HL}-\ref{thm:bous} by generalizing the arguments in \cite{lafleche2021instability,shao2022instability}.


\vspace{0.1in}
\paragraph{\bf{Notations}}
We first introduce some notations to be used in the analysis.
We use $(r, \vartheta, z)$ to denote the cylindrical coordinate in $\R^3$. The associated basis is 
\beq\label{eq:polar_basis}
e_r = ( \cos \vth, \sin \vth, 0 ), \quad e_{\vth} = ( - \sin \vth, \cos \vth, 0), \quad e_z = (0, 0, 1).
\eeq
For $x$ with coordinate $(x_r, x_{\vth}, x_z)$ and $A \subset \R^3$, we use $\td x, \td A$ to denote the poloidal component 
\beq\label{eq:polo}
  \td x = (x_r, x_z) , \quad \td A = \{ \td x : x \in A \}.
\eeq

The poloidal component of the axisymmetric vorticity $\om$ is defined as follows 
\beq\label{eq:polo_w}
\om_p \teq \om^r e_r + \om^z e_z,  \quad \om = \om^r e_r  + \om^{\vth} e_{\vth} + \om^z e_z.
\eeq

In the analysis of the axisymmetric Euler equations, 
for any 2D domain $\S$ of $(r, z)$, we abuse the notation and use
\beq\label{eq:nota_abuse}
  x \in \S  \quad \mathrm{  if   } \quad \td x = ( x_r, x_z ) \in \S.
\eeq
For example,  $x \in B_{(1,0)}(\d)$ means $(x_r, x_z ) \in  B_{(1,0)}(\d)$, or equivalently, $x$ in the annulus $B_{(1,0)}(\d) \times \R / (2\pi \BZ) $. We abuse this notation since the flow is axisymmetric and thus many variables, e.g., $u^r, u^z, u^{\vth}, \om^{\vth}$, depend on $(r, z)$ only.


\subsection{The WKB expansion and the bicharacteristics-amplitude ODEs}\label{sec:review}


The main idea in \cite{vasseur2020blow,lafleche2021instability} is to construct an approximate solution  to \eqref{eq:euler_lin} using a WKB expansion 
\beq\label{eq:WKB00}
v(t, x) \approx b(t, x) \exp( \f{ i S(t, x)}{\e} )
\eeq
for sufficiently small $\e$, where $b(t, x) \in \R^3$ and $S$ is a scalar, and the following bicharacteristics-amplitude ODE system \eqref{eq:bichar1}-\eqref{eq:bichar3} \cite{lafleche2021instability,vasseur2020blow}
\begin{align}
\dot \g_t  & = \uu(t, \g_t) , \quad \g_0 = x_0, 
\label{eq:bichar1} \\
\dot \xi_t &= -( \na \uu)^T(t, \g_t)   \xi_t, \label{eq:bichar2} \\
\dot b_t & = - (\na \uu)(t, \g_t) b_t + 2 \f{ \xi_t^T  ( \na \uu)(t, \g_t)   b_t}{ |\xi_t|^2} \xi_t, \label{eq:bichar3}
\end{align}
with initial data $(x_0, \xi_0, b_0)$, where $(\na u)_{ij} = \pa_j u_i$. The regularity assumption $\uu \in C^0( [0, T], H^s), s > 9/2$ in \cite{vasseur2020blow} is mainly used to guarantee the solvability of the above ODEs with smooth dependence on the initial data.

The ODE system \eqref{eq:bichar1}-\eqref{eq:bichar3} has been derived in \cite{friedlander1991dynamo} to define the fluid Lyapunov exponent and used to study the stability of steady states of the Euler equations \cite{friedlander1991instability,friedlander1997nonlinear}. The WKB expansion \eqref{eq:WKB00} was developed in \cite{Vishik1996spectrum} to study the spectrum of small oscillations in an ideal incompressible fluid. 
It has also been used to study the local stability conditions for the Euler equations \cite{lifschitz1991local}.

For the sake of completeness, in Appendix \ref{app:WKB}, we begin with the WKB expansion \eqref{eq:WKB00} and then explain the use of the bicharacteristics-amplitude ODE system \eqref{eq:bichar1}-\eqref{eq:bichar3}, which arise naturally in the construction of the approximate solution. We also explain the connections among the WKB expansion, the bicharacteristics-amplitude ODE system \eqref{eq:bichar1}-\eqref{eq:bichar3}, and the growth of the unstable solution. From the review in Appendix \ref{app:WKB}, we have a few remarks.

\begin{remark}
(a)  From the proof in \cite{vasseur2020blow} and the simplified derivations in Appendix \ref{app:WKB}, the WKB construction and the high frequency \eqref{eq:WKB00} are mainly used to construct an approximate solution to \eqref{eq:euler_lin} with a small error in the $L^p$ norm but not used to show the growth of the unstable solution.

(b) The growth of the solution $v$ and the linear instability are coupled with the growth of the vorticity via the ODE system \eqref{eq:bichar1}-\eqref{eq:bichar3} and \eqref{eq:growth}.

(c) As we mentioned in Section \ref{sec:Euler_unstable}, for a domain without boundary, $\pa_i \uu, i=1,2,3$ are the exact solutions to \eqref{eq:euler_lin} and blow up in a functional space $X$ equipped with a norm stronger than the $L^{\inf}$ norm. These  
simple instability results do not use \eqref{eq:WKB00} and \eqref{eq:bichar1}-\eqref{eq:bichar3}. 

(d) The argument in \cite{vasseur2020blow} has an advantage that several nonlocal terms become local. It is based on the characteristics and is local in nature. Due to this local property, we can relax the regularity assumptions in the proof in \cite{vasseur2020blow} for the singular solutions in \cite{chen2019finite2,elgindi2019finite} and generalize it to prove Theorems \ref{thm:Euler_HL}-\ref{thm:bous}. 
\end{remark}

\subsection{Properties of the singular solutions}

The singular solution to the 2D Boussinesq equations \eqref{eq:bous}-\eqref{eq:biot} constructed in \cite{chen2019finite2} satisfies the following properties. The $\cC^k$ norm in the following theorem is defined in \eqref{norm:ck}. The reader should not confuse it with the standard $C^k$ norm.

\begin{thm}\label{thm:bous_blowup}
Let $\omega$ be the vorticity and $\theta$ be the density in the 2D Boussinesq equations described by \eqref{eq:bous}-\eqref{eq:biot}. There exists $\al_0 > 0$ such that for $0 < \al < \al_0$, the unique local solution of the 2D Boussinesq equations in the upper half plane develops a focusing asymptotically self-similar singularity in finite time $T_*$ for some initial data $\om_0 \in C_c^{\al}(\R_+^2) , \th_0 \in C_c^{1, \al}(\R_+^2)$.
Moreover, we have $\lim_{t \to T_*} || \na \th(t)||_{\inf} = \inf$, the velocity field is in $C^{1, \al}$ with finite energy. For any $T < T_*$ and any compact domain $\S$ in the interior of  $ \{ (x, y) : x \neq 0, y > 0  \} $, we have $\th_0 \in C^{50}(\S)$ and 
$\om, \na \th , \f{1}{ \sqrt{x^2 + y^2}} \uu \in L^{\inf}( [0, T] , \cC^{50} \cap C^{50}( \S) ), 
 \uu  \in L^{\inf}( [0, T],  C^{50}(\S) ) $.
\end{thm}

The regularity $\cC^{50}, C^{50}$ can be further improved to $\cC^k, C^k$ with larger $k$ directly by choosing smaller $\al_0$. The first part of the theorem about the blowup has been proved in \cite{chen2019finite2}. To prove the regularity in the interior of the domain, we generalize the weighted energy estimates for the perturbation and the estimates of the approximate steady state in \cite{chen2019finite2} to sufficiently high order. Since the weighted norms used in \cite{chen2019finite2} and the energy estimates, e.g. $\cH^k$ (see \eqref{norm:Hk}), are comparable to the standard Sobolev norms $H^k$ in the interior of the domain, we establish the interior regularity using the embedding inequalities. See Section \ref{sec:high} for the proof.

In \cite{chen2019finite2}, the 3D axisymmetric Euler equations are studied in a cylinder $D = \{ (r,z) : r \in [0,1], z \in \BT \}, \BT = \R / ( 2 \BZ)$ that is periodic in $z$. Here, $r, z$ are the cylindrical coordinates in $\R^3$. The equations are given below:
\beq\label{eq:euler1}
\pa_t (ru^{\vth}) + u^r (r u^{\vth})_r + u^z (r u^{\vth})_z = 0, \quad 
\pa_t \f{\om^{\vth}}{r} + u^r ( \f{\om^{\vth}}{r} )_r + u^z ( \f{\om^{\vth}}{r})_z = \f{1}{r^4} \pa_z( (r u^{\vth})^2 ),
\eeq
where $\om^{\vth}$ is the angular vorticity and $u^{\vth}$ is the angular velocity. The radial and the axial components of the velocity can be recovered from the Biot-Savart law
\beq\label{eq:euler2}
-(\pa_{rr} + \f{1}{r} \pa_{r} +\pa_{zz}) \td{\psi} + \f{1}{r^2} \td{\psi} = \om^{\vth}, 
 \quad  u^r = -\td{\psi}_z, \quad u^z = \td{\psi}_r + \f{1}{r} \td{\psi}  
\eeq
with a no-flow boundary condition on the solid boundary $r = 1$
\beq\label{eq:euler21}
\td{\psi}(1, z ) = 0
\eeq
and a periodic boundary condition in $z$. For the 3D Euler equations, we have the following results 
\begin{thm}\label{thm:euler_blowup}
There exists $\al_0 > 0$ such that for $0 < \al < \al_0$, the unique local solution of the 3D axisymmetric Euler equations in the cylinder $D = \{ r,z \in [0, 1] \times \BT \}$ given by \eqref{eq:euler1}-\eqref{eq:euler21} develops a singularity in finite time $T_*$ for some initial data $\om_0^{\vth} \in C^{\al}(D) , u_0^{\vth} \in C^{1, \al}(D)$.
The initial data $\om_0^{\vth}, u_0^{\vth}$ are supported away from the symmetry axis $r=0$ with $u_0^{\vth} \geq 0$, $\om_0^{\vth}$ is odd in $z$, $u_0^{\vth}$ is even in $z$, and the velocity field $\uu_0$ in each period has finite energy.

Moreover, the singular solution satisfies the following properties. 

(a) The poloidal component $\om_p = \om^r e_r + \om^z e_z$ blows up $\lim_{t\to T_* } || \om_p(t)||_{\inf} = \inf$.



(b) There exists constants $0 < 4 R_{1, \al} < R_{2, \al} < \f{1}{4}$ such that for any particle within the support of $\om^{\vth}_0, u_0^{\vth}$, its trajectory up to the blowup time is within $B_{ (1, 0) }( R_{1,\al}) \cap D$.

(c) For any compact domain $\S$ in $ \{ (r, z) : r \in (0, 1) , z \neq 0 \} \cap B_{ (1,0)}( R_{2,\al})$ and $T < T_*$, we have 
$ u^{\vth}_0 \in  C^{50}(\S), \om^{ \vth}, (u^{\vth})^2, u^r, u^z, u^{\vth} \in L^{\inf}( [0, T] ,  C^{50}( \S) )$.

\end{thm}

Except for result (c), the above theorem has been mostly proved in \cite{chen2019finite2}. We recall from Remark \ref{rem:oversight} that the oversight $u_0^{\vth} \notin C^{1,\al}$ in \cite{chen2019finite2} has been fixed in the updated arXiv version of \cite{chen2019finite2}. See also Remark \ref{rem:change2}.
The parameter $R_{2,\al}$ and domain $B_{(1,0)}(R_{2,\al})$ in the above theorem relate to the localized elliptic estimate. 
In particular, the cutoff function to localize the estimate is $1$ in $B_{(1,0)}(R_{2,\al})$. 
One of the main difficulties in the proof is to show that $u^{\vth}$ is smooth in $\S$. This does not follow from $(u^{\vth})^2 \in C^{50}(\S)$ since $u^{\vth}$ has compact support and can degenerate in $\S$. We use the property that $ r u^{\vth}$ is transported along the flow to prove that it is smooth. See Section  \ref{sec:euler_blowup} for the proof.

The singular solution constructed in \cite{elgindi2019finite,elgindi2019stability} enjoys the following properties, which follow from the estimates in  \cite{elgindi2019finite,elgindi2019stability}.
\begin{thm}\label{thm:euler_R3_blowup}
There exists $\al_0 > 0$ such that for $0 < \al < \al_0$, the unique local solution of the axisymmetric Euler equations \eqref{eq:euler1}-\eqref{eq:euler2} in $\R^3$ without swirl $u^{\vth} \equiv 0$ develops a singularity in finite time $T_*$ for some initial data $\om_0^{\vth} \in C_c^{\al}(\R^3)$ odd in $z$ with finite energy $|| \uu_0||_{L^2} < +\inf$. In addition, we have $u^r_r( t,0 ,0 ) > 0$ and 
\beq\label{eq:one_pt}
 \int_0^{T_*} u_r^r( t, 0, 0)  dt = \infty .
\eeq
For any compact domain $\S \subset \{ (r, z): r >0, z \neq 0 \} $ and $T < T_* $, we have $ \om^{\vth} , u^r, u^z \in  L^{\inf}( [0, T], C^{50}( \S) )$. 
\end{thm}

In the blowup results in Theorem \ref{thm:bous_blowup} and \ref{thm:euler_blowup}, $\na \uu$ also blows up at the singularity point. Since the blowup of $\na \uu$ implies the blowup of the solution,  \eqref{eq:one_pt} can be seen as a blowup criterion for the singular solution in \cite{elgindi2019finite}. A similar one-point blowup criterion has been established to prove global regularity of the De Gregorio model for a large class of initial data in \cite{chen2021regularity}.

In the remaining part of this Section, we prove Theorems \ref{thm:Euler_HL}-\ref{thm:bous} using the important properties of the blowup solution in Theorems \ref{thm:bous_blowup}-\ref{thm:euler_R3_blowup} and the argument in \cite{lafleche2021instability,shao2022instability}. We first prove Theorem \ref{thm:Euler_HL}.

\subsection{Trajectory and the bicharacteristics-amplitude ODE}

Due to the periodicity in $z$, we consider the domain within one period
\beq\label{eq:domain_HL}
D_1 \teq \{ (r,z) : r \in [0,1], |z| \leq 1 \} .
\eeq
We further decompose $D_1$ into the two parts and introduce $\Ups$
\beq\label{eq:domain_pm}
\bal
& D_1^+ \teq \{ (r,z) : r \in [0,1], z  \in [0, 1] \} , \quad  D_1^- \teq \{ (r,z) : r \in [0,1], z  \in [-1, 0] \} , \\
& \Upsilon \teq \{ (r, z) : r = 1 \mathrm{ \ or \ } z = 0 \}.
\eal
\eeq
The set $\Upsilon$ denotes the boundary of the cylinder $D$ and the symmetry plane $z = 0$.

Let $\uu$ be the velocity in Theorem \ref{thm:euler_blowup}. In the cylindrical coordinate $(r, \vth, z)$\eqref{eq:polar_basis}, 
we have $\uu = u^r e_r + u^{\vth} e_{\vth} + u^z e_z$. Since the singular solutions $\om^{\vth}, u^z$ in Theorem \ref{thm:euler_blowup} are odd in $z$ and we impose the no flow boundary condition \eqref{eq:euler21}, we obtain
\beq\label{eq:bc_vanish}
\uu(t) \cdot n \B|_{\Upsilon} = u^r(t) \cdot n^r + u^z(t) \cdot n^z  = 0,
\eeq
where $n$ is the normal vector of $\Ups$. Let $\td \g_t =  (r_t, z_t)$ \eqref{eq:polo} be the $(r, z)$ component of $\g_t$ in \eqref{eq:bichar1}. Since the flow is axisymmetric, we have
\beq\label{eq:char_polo}
  \f{d}{dt} r_t = u^r( r_t, z_t, t), \quad   \f{d}{dt} z_t = u^z( r_t, z_t, t),
  \quad \f{d}{dt} \td \g_t = (u^r, u^z)(\td \g_t, t).
\eeq
Thus, the angular coordinate $x_{0,\vth}$ of the initial data $x_0$ does not affect $\td \g_t$, and $\td \g_t$ depends on $\td x_0 = (r_0, z_0)$ only. Therefore, we have
\beq\label{eq:char_polo2}
\td \g_t( \td x_0 ) = \td \g_t( x_0 ) = (r_t, z_t), \quad  
\wt{ \g_t^{-1}(x) } =  ( \wt{\g_t} )^{-1} (x) =( \wt{\g_t} )^{-1} ( \td x) .
\eeq
We have the following results for the system \eqref{eq:bichar1}-\eqref{eq:bichar3}.

\begin{lem}\label{lem:traj}
Let $\g_t$ be the solution to \eqref{eq:bichar1} with initial data $x_0$, $T_*$ be the blowup time, $T < T_*$, and $D_1^{\pm}$ be the domains defined in \eqref{eq:domain_pm}. (a) For any $x_0 \in \Ups$ and $ t \in [0, T_*)$, the trajectory $\g_t$ remains in $\Ups$; for any $x_0 \in D_1^{\pm} \bsh \Ups$ and $t \in [0, T_*)$, we have $\g_t \in D_1^{\pm} \bsh \Ups$. For any $t \in [0, T]$, $\g_t$ is invertible, and $\g_t, \g_t^{-1}$ are Lipschitz in time and the initial value. 


Let $R_{1,\al}, R_{2,\al}$ be the radius in Theorem \ref{thm:euler_blowup}.

(b) Suppose that $x_0 \in ( D_1^{\pm} \bsh \Ups) \cap  \supp( \om_0 )  $. There exists $\d( \td x_0, T) \in (0, \f{1}{8})$ depending on $\td x_0, T$ and a compact set $\Sigma_2$, such that for any $ t\in[0, T]$, we have
\beq\label{eq:traj_tube3}
\td \g_t(  B_{\td x_0 }(  \d) ) \cup B_{ \td \g_t( \td x_0)}( \d)
\subset  \Sigma_2  \subset ( D_1^{\pm} \bsh \Ups) \cap B_{(1,0)}(R_{2,\al}) 
.\eeq
As a result, for initial data $z_0$ with $\td z_0 \in B_{ \td x_0}(  \d)$ and any $b_0, \xi_0$, there exist unique solutions $(\g_t, b_t, \xi_t)$ to \eqref{eq:bichar1}-\eqref{eq:bichar3} on $t\in [0, T]$. For $t \in [0, T]$, the functions $(\g_t, b_t, \xi_t)$ are Lipschitz in time and $C^4$ with respect to initial data $z_0$ with $\td z_0 \in B_{\td x_0}( \d)$ and $b_0, \xi_0$, and $\g_t^{-1}(x)$ is Lipschitz in time and $C^4$ in $x$ with $\td x \in \td \g_t( B_{\td x_0}(  \d)  ) \cup B_{\td \g_t( \td x_0)}(\d)$.
\end{lem}

In the above Theorem, we have used the notation \eqref{eq:nota_abuse}. For example,  $x_0 \in D_1^{\pm} \bsh \Ups$ means $\td x_0 \in D_1^{\pm} \bsh \Ups$.
The domain of $x$ with $\td x \in B_{\td x_0}(\d)$ is the annulus $(r, z, \vth) \in B_{\td x_0}(\d) \times \R / (2\pi \BZ)$.

The ideas of the above Lemma are simple. Firstly, for any $x_0 \in D_1^{\pm} \bsh \Ups$, the trajectory $\g_t$ with $t \in [0, T]$ remains in $D_1^{\pm} \bsh \Ups$. Using the Lipschitz property of $\td \g_t, \td \g_t^{-1}$, we can find a neighborhood of $\td \g_t$ that still remains in $D_1^{\pm} \bsh \Ups$. We further restrict $\td x_0$ sufficiently close to $(1, 0)$ and use the property that $\uu(x)$ is smooth for $x$ with  $\td x \in D_1^{\pm} \bsh \Ups \cap B_{(1,0)}(R_{2,\al})$ from Theorem \ref{thm:euler_blowup} to solve \eqref{eq:bichar1}-\eqref{eq:bichar3}. 

\begin{proof}

Recall the notation $\td x = (r, z)$ from \eqref{eq:polo}. Due to $\uu \in C^0( [0, T_*), C^{1,\al} )$ and the non-penetrated property \eqref{eq:bc_vanish}, the results in (a) follow directly from the Cauchy-Lipschitz theorem. 

Without loss of generality, we consider the domain $D_1^+ \bsh \Ups$. For any $x_0 \in ( D_1^{+} \bsh \Ups)  \cap  \supp( \om_0 )$, from result (b) in Theorem \ref{thm:euler_blowup} and \eqref{eq:char_polo2}, we know 
\beq\label{eq:traj_pf1}
\td \g_t(\td x_0) \in ( D_1^+ \bsh \Ups) \cap B_{(1,0)}(R_{1,\al}), \quad t \in [0, T].
\eeq
Since $\td \g_t(\td x_0)$ is continuous in $t$, using compactness, we have $\mathrm{dist}( \td \g( \td x_0, [0, T]), \Ups) > 0$.  Let $L_{\g}$ be the Lipschitz constant of $\g_t, \g_t^{-1}$ on $[0, T]$. Denote 
\[
d_1 = \mathrm{dist}( \td \g( \td x_0, [0, T]), \Ups), \quad  \d_1 \teq \f{1}{2} \min( d_1 ,  R_{1,\al}) > 0, 
\quad \d =\min( \f{\d_1}{ 2(  L_{\g}+1)}, \f{1}{16}).
\]


For $y = \td \g_t( \td x), \td x \in B_{ \td x_0}( \d)$, using \eqref{eq:traj_pf1}, we yield
\[
\bal
&|y - \td \g_t( \td x_0)|  \leq L_{\g} | \td x -\td x_0 | \leq L_{\g} \d < \f{\d_1}{2}, 
\quad \dist(y, \Ups) \geq \dist( \td \g_t( \td x_0), \Ups) - \f{\d_1}{2} > \f{\d_1}{2}, \\
 & |y - (1,0)| < | \td \g_t( \td x_0) - (1,0)| + \f{\d_1}{2}  \leq \f{3}{2} R_{1,\al}.
\eal
\]
It follows that $ y \in  D_1^+ \bsh \Ups \cap B_{(1,0)}( \f{3}{2} R_{1,\al})$. We define the compact set
\beq\label{eq:traj_comp}
\Sigma_2 = \{ \td x : \mathrm{dist}(\td x, \Ups) \geq \f{1}{4} \d_1 \} \cap \bar D_1^+ \cap \bar B_{(1,0)}( 2 R_{1,\al}).
\eeq
Recall from Theorem \ref{thm:euler_blowup} that $R_{2,\al} > 4 R_{1,\al}$. The above derivations imply $  \td \g_t( B_{ \td x_0}( \d) ) \subset \Sigma_2$. The proof of $ B_{ \td \g_t( \td x_0)}( \d)  \subset \Sigma_2$ follows from the same argument and is easier. We obtain \eqref{eq:traj_tube3}.


Now, we consider \eqref{eq:bichar1}-\eqref{eq:bichar3} for initial data $z_0 $ with $\td z_0 \in B_{\td x_0}(  \d)$ and $b_0, \xi_0$. Since $\Sigma_2$ is a compact set in $(D_1^+ \bsh \Ups) \cap B_{(1,0)}(R_{2,\al})$, from Theorem \ref{thm:euler_blowup}, we have $u^r, u^z, u^{\vth} \in L^{\inf}( [0, T],  C^{50}(\Sigma_2) )$. Since $\td \g_t( B_{\td x_0}(  \d)) , B_{\td x_0}(  \d) \subset \Sigma_2$ and $\uu(x)$ is smooth for $x$ with $\td x \in \S_2$, using the Cauchy-Lipschitz theorem, there exist unique solutions $(\g_t, b_t, \xi_t)$ to \eqref{eq:bichar1}-\eqref{eq:bichar3} on $t\in [0, T]$, and $\g_t, b_t, \xi_t$ are Lipschitz in time and $C^4$ with respect to the initial data. 

Next, we consider the backward equation. Denote $\d_2 = \f{ \d}{ L_{\g} + 1}$. Fix $t \leq T$. For any $s \in [0, t]$, from \eqref{eq:char_polo2} and \eqref{eq:traj_tube3}, we get
\[
\td \g_s^{-1} \td \g_t( B_{\td x_0}(\d_2)) = \td \g_{t-s}( B_{\td x_0}(\d_2)) \subset \S_2, 
\quad \td \g_s^{-1} B_{ \td \g_t(\td x_0)}(\d_2) 
\subset B_{ \td \g_{t-s}(\td x_0) }(L_{\g} \d_2 )
\subset B_{ \td \g_{t-s}(\td x_0) }(\d) \subset \S_2 .
\]
From Theorem \ref{thm:euler_blowup} and $u^r, u^z, u^{\vth} \in L^{\inf}( [0, T],   C^{50}( \Sigma_2) )$, we can solve \eqref{eq:bichar1} backward on $[0, t]$ for initial data $x_t$ with $\td x_t \in \td \g_t( B_{ \td x_0}(  \d_2 ) ) \cup B_{ \td \g_t(\td x_0)}(\d_2) \subset \Sigma_2$, and $\g_t^{-1}$ is Lipschitz in time and $C^4$ in the initial data.

Finally, due to the inclusion 
\[
\td \g_t( B_{\td x_0}(\d_2)) \cup ( B_{ \td \g_t( \td x_0) }(\d_2))
\subset \td \g_t( B_{\td x_0}(\d)) \cup ( B_{ \td \g_t( \td x_0) }(\d)) \subset \S_2,  \quad t \in [0, T],
\]
we prove result (b) for $\S_2$ defined in \eqref{eq:traj_comp} and $\d = \d_2$. 
\end{proof}


\subsection{Relaxation of \texorpdfstring{$\b_{\s}(t)$}{Lg}}

Recall the definition of $\b_{\s}(t)$ from \cite{lafleche2021instability}
\beq\label{eq:def_beta0}
\b_{\s}(t) = \sup_{ (x_0, b_0, \td \xi_0) \in D_1 \times \R^3 \times S^1, b_0 \cdot \xi_0 =0, |b_0| = r_0^{\s}} | r_t^{-\s} b_t(x_0, \td \xi_0, b_0)|,
\eeq
where $D_1$ is the domain for the Euler equations \eqref{eq:domain_HL}. Here, the notation $\xi_0 = \td \xi_0 \in S^1$ means that the initial data $\xi_0 $ satisfies $ \xi_{0} \cdot e_{\vth(x_0)} = 0$ and  $ ( \xi_0 \cdot e_{r(x_0)} )^2 +( \xi_{0} \cdot e_z)^2 = 1 $, where $ e_{r(x_0)}, e_{\vth(x_0)}, e_z$ are the basis \eqref{eq:polar_basis} associated with $x_0$. Since $\xi_{0} \cdot e_{\vth(x_0)}=0$, it relates to the notation \eqref{eq:polo}.


We focus on the case $\s = 0$ and relax the domain $D_1$ \eqref{eq:domain_HL} to $(D_1 \bsh \Ups ) \cap \supp(\om_0)$
\beq\label{eq:def_beta}
\b(t) = \sup_{ (x_0, b_0, \td \xi_0) \in ( D_1 \bsh \Ups) \cap \supp(\om_0) \times \R^3 \times S^1, b_0 \cdot \xi_0 =0, |b_0| = 1 } |  b_t(x_0, \td \xi_0, b_0)|,
\eeq
where $\om_0$ is the vorticity of the singular solution in Theorem \ref{thm:euler_blowup}. From Lemma \ref{lem:traj}, for any $ t< T^*, x_0 \in  D_1\bsh \Ups  , b_0 \in \R^3, \td \xi_0 \in S^1$, $b_t(x_0, \td \xi_0, b_0 )$ is well defined.

We have the following result, which modifies  Proposition 2 in \cite{lafleche2021instability}.
\begin{prop}\label{prop:beta}
Assume that $\uu$ is the singular solution in Theorem \ref{thm:euler_blowup}, $\om$ is the associated vorticity, and $\om_p$ is the poloidal component \eqref{eq:polo_w}. For any $ t \in (0, T^*)$, we have 
\[
||  \om_p( t, \cdot) ||_{\inf} \leq ||  \om_p^{in} ||_{\inf} \b(t)^2.
\]
\end{prop}

\begin{proof}

We assume  $||  \om_p(t)||_{\inf} > 0$. Otherwise, the result is trivial. Since $\om(t) \in C^{\al}$ and $|\om|$ is even in $z$, using continuity and symmetry, we get 
\[
||  \om_p( t, \cdot) ||_{\inf} = \sup_{ x \in D_1^+ \cap \supp(\om(t))} | \om_p(t, x)| = \sup_{ x \in ( D_1^+ \bsh \Ups ) \cap \supp(\om(t))} |  \om_p(t, x)|.
\]

Now, for each $(t, x_t) \in (0, T^*) \times ( D_1^+ \bsh \Ups) $ with $|\om(t, x_t)| > 0$, we can solve \eqref{eq:bichar1} backward on $[0, t]$ with initial data $\g_t = x_t$. Since $x_t \in D_1^+\bsh \Ups$ and $|\om(t, x_t)| > 0$, using \eqref{eq:euler} and a simple energy estimate along the trajectory implies $|\om(0, x_0) | > 0$. Thus, we get $x_0 \in \supp( \om_0)$. From Lemma \ref{lem:traj}, we further obtain $x_0 \in ( D_1^+ \bsh \Ups) \cap \supp( \om_0)$. Then we can solve \eqref{eq:bichar1}-\eqref{eq:bichar3} with initial data $x_0$ and any $b_0 , \xi_0$ and solve \eqref{eq:bichar1}-\eqref{eq:bichar3} backward with initial data $x_t$ and any $b_t, \xi_t$.

We relax the definition of $\b(t)$ since it suffices to consider $x_0 \in 
( D_1^+ \bsh \Ups) \cap \supp( \om_0) \subset
( D_1 \bsh \Ups) \cap \supp( \om_0)$ instead of all $x_0 \in D_1$. The rest of the proof follows the same argument in \cite{lafleche2021instability}. 
\end{proof}

Next, we show that for the singular solution in Theorem \ref{thm:euler_blowup}, Proposition 3 in \cite{lafleche2021instability} remains true. Recall the definition of $\lam_{p, \s}^{sym}$ from \eqref{eq:instab1}. We drop the domain $D$ to simplify the notation.
\begin{prop}\label{prop:lam}
Let $ t \in (0, T_*), p \in [ 1, \inf)$. Assume that $\uu$ is the singular solution in Theorem \ref{thm:euler_blowup}. Then we have $ \b(T) \les_{\s}  \lam^{sym}_{p, \s}(T)$ for any $\s \in \R$. 
\end{prop}


One of the difficulties in the proof is to construct an \textit{axisymmetric} solution to \eqref{eq:euler_lin}.

\begin{remark}\label{rem:non_axi}
The approximate solution and the initial data $v_{\e, \d}^{in}$ to \eqref{eq:euler_lin} constructed in \cite{lafleche2021instability} 
\beq\label{eq:WKB_euler}
v_{\e, \d} = \e \mathrm{ curl }\B(\f{b\times \xi}{ |\xi|^2} \vp e^{iS / \e}   \B) =
i \vp b e^{i S / \e} + \e  c(x) e^{iS / \e} \teq A + B,  \quad c(x) = \mathrm{curl}( \f{ b \times \xi}{ |\xi|^2}  \vp )
\eeq
are not axisymmetric, where $b(t, x), \xi(t,x) \in \R^3$, $S, \vp$ are scalar functions, and $\e$ is a small parameter. See equation (21) in \cite{lafleche2021instability}. To illustrate this point, we  study the initial data more carefully. According to the construction in the proof of Proposition 3 in \cite{lafleche2021instability}, for $t= 0$, we have $b(0, x) \equiv b_0,  \xi(0, x) \equiv \xi_0 $ for some 
\beq\label{eq:non_axi0}
|b_0 | = 1, \quad |\xi_0| = 1, \quad b_0 \cdot \xi_0 = 0.
\eeq
In particular, $b, \xi$ are constant vectors. Moreover, $\vp, S$ are independent of the angular variable $\vth$ \cite{lafleche2021instability}, i.e. $\vp(x) = \vp(r, z), S(x) = S(r, z)$. Hence, we get 
\beq\label{eq:non_axi1}
c(x) = \na \vp \times \f{b_0 \times \xi_0}{ |\xi_0|^2}
= \na \vp \times s_0, \quad s_0 \teq \f{b_0 \times \xi_0}{ |\xi_0|^2}, \quad  \pa_{\vth} A = 0. 
\eeq
Suppose that $v_{\e, \d}$ is axisymmetric \eqref{axi}. Then $ v_{\e, \d } \cdot \eta$ does not depend on $\vth$ for $\eta = e_r, e_{\vth}, e_z$ \eqref{eq:polar_basis}. Using these properties, \eqref{eq:non_axi1}, and $\pa_{\vth} e_r = e_{\vth}, \pa_{\vth} e_{\vth} = - e_r$, we get
\[
0 = \pa_{\vth} ( v_{\e, \d} \cdot e_r ) 
=  \pa_{\vth} ( (A+B) \cdot e_r) 
= A \cdot e_{\vth} + \pa_{\vth}( B \cdot e_r) 
=  A \cdot e_{\vth} + \e e^{iS / \e} \pa_{\vth} c(x) \cdot e_r-  B \cdot e_{\vth}.
\]
Since the second and the third term have size $O(\e)$ and $\e$ is taken to $ \e \to 0$ in \cite{lafleche2021instability}, for sufficiently small $\e$,  $A \cdot e_{\vth}$ and $\pa_{\vth}( B \cdot e_r)$ must be $0$. Similarly, we get $A \cdot e_r = 0, \  \pa_{\vth}( B \cdot e_z) = 0$. 
Since the direction of $A$ is given by $b_0$, it follows that $b_0 = (0, 0, b_{0, z}) = b_{0, z} e_z$. Note that $\vp( x) = \vp(r, z)$ and $\na \vp = \pa_r \vp(r, z) e_r + \pa_z \vp(r, z) e_z $. 
From \eqref{eq:non_axi1} and $\pa_{\vth}( B \cdot e_z) = 0$, we get 
\[
\bal
 c(x) \cdot e_z  &=  ( \pa_r \vp  e_r \times s_0 + \pa_z \vp e_z \times s_0 ) \cdot e_z
= \pa_r \vp \cdot ( e_r \times s_0 ) \cdot e_z,  \\
 0 = \pa_{\vth}( B \cdot e_z ) = \e e^{ iS / \e}  \pa_{\vth}( c(x) \cdot e_z )
&=   \e e^{iS / \e}\pa_r \vp \cdot ( \pa_{\vth} e_r \times s_0) \cdot e_z
=  \e e^{iS / \e} \pa_r \vp \cdot (  e_{\vth} \times s_0) \cdot e_z.
\eal
\]
Since $b_0 =b_{0, 3} e_z$, we get $s_0 \cdot e_z =0, e_{\vth} \cdot e_z = 0$, which implies that $ e_{\vth} \times s_0$ and $e_z$ are parallel. Then the above identity implies $e_{\vth} \times s_0  = 0$. Since $s_0$ is a constant vector and $\vth$ is arbitrary, we further obtain $s_0 = 0$, which contradicts \eqref{eq:non_axi0} and \eqref{eq:non_axi1}.

\end{remark}


The proof of Proposition \ref{prop:lam} consists of several steps. Firstly, given $x_0, b_0, \xi_0$, we construct axisymmetric flows $\xi(t, x), b(t, x)$ and function $S(t, x)$ using the PDE form of \eqref{eq:bichar1}-\eqref{eq:bichar3} such that $\xi(0, x_0) = \xi_0, b(0, x_0) = b_0, \na S = \xi$. Since the singular solution $\uu$ in Theorem \ref{thm:euler_blowup} is only $C^{1,\al}$, these functions $\xi, b, S$ are not smooth enough to apply the argument in  \cite{lafleche2021instability} to prove Proposition \ref{prop:lam}. Our key observation is that the solution \eqref{eq:WKB_euler} leading to the instability \cite{lafleche2021instability} is constructed locally along the trajectory of $x_0$. Thus, we can apply Lemma \ref{lem:traj} and Theorem \ref{thm:euler_blowup} to localize $\uu$ and obtain a much smoother localized velocity $\uu \cdot \chi$.
Then we can obtain smooth $b, \xi, S$ and an axisymmetric velocity field given by \eqref{eq:WKB_euler}.
Finally, we show that $b(T, x)$ can control $\b(T)$ using the axisymmetric property of $b$. The remaining proof follows the argument in \cite{lafleche2021instability}.

Before we present the proof, we need a simple Lemma for axisymmetric flows.
\begin{lem}\label{lem:axi}
Suppose that $A(x), B(x)$ are axisymmetric flows \eqref{axi}, and $C(x) = C(r, z)$ is independent of $\vth$. Then $A \times B$, $C(x) A$, $\na \times A, \pa_r A, \pa_z A, \pa_{\vth} A$ are axisymmetric flows, and $A \cdot B$ is independent of $\vth$.
\end{lem}

\begin{proof}
Since $e_r, e_{\vth}, e_z$ \eqref{eq:polar_basis} are orthonormal basis, a simple calculation implies that $A \times B, C(r, z) A$ are axisymmetric and that $A \cdot B$ is independent of $\vth$. The property that the curl operator does not change axisymmetry is standard. For example, if the velocity $\uu$ is axisymmetric, the vorticity $\om = \na \times \uu$ is also axisymmetric. The same reasoning and calculation apply to $\na \times A$. Since $\pa_r \eta = 0, \pa_z \eta = 0 $ for $\eta = e_r, e_{\vth}, e_z$ and $\pa_{\vth} A 
= A^r(r, z) e_{\vth} - A^{\vth}(r, z) e_{r} $ for $A = A^r e_r + A^{\vth} e_{\vth} + A^z e_z$, we 
conclude that $\pa_rA, \pa_z A, \pa_{\vth} A$ are axisymmetric. 
\end{proof}

\begin{proof}[Proof of Proposition \ref{prop:lam}]

Recall the poloidal component \eqref{eq:polo},\eqref{eq:char_polo2}
\beq\label{eq:lam_polo}
\td x = (r , z), \quad \td \g_t = (r_t, z_t), \quad \td A = \{\td a: a \in A \}.
\eeq
We fix $T < T_*$. Suppose that $\b(T) > 0$. Otherwise, the proof is trivial.
Using the definition of \eqref{eq:def_beta} and result (b) in Theorem \ref{thm:euler_blowup}, for any $\eta > 0$, we can choose $(x_0, \xi_0, b_0)$ such that 
\beq\label{eq:lam_init}
x_0 \in (D_1 \bsh \Ups) \cap \supp(\om_0) \subset B_{(1,0)}(1/4), \quad r_0 \neq 0, \ \xi_0 = \td \xi_0,  \ \xi_0 \cdot b_0 = 0,
\eeq
and 
\beq\label{eq:lam_pf0}
0 < \b(T) \leq (1 + \eta) |b_T( x_0, \td \xi_0, b_0)|.
\eeq
We have $r_0 \neq 0$ since $x_0 \in B_{(1,0)}(1/4)$ implies $r_0 \geq \f{3}{4}$. Denote 
\beq\label{eq:vth0}
\vth_0 = x_{0,\vth}.
\eeq
Without loss of generality, we assume $x_0 \in D_1^+$. 
From Lemma \ref{lem:traj}, there exists $\d > 0$ and a compact set $\S_2$ such that \eqref{eq:bichar1}-\eqref{eq:bichar3} have a unique solution $( \g_t, b_t, \xi_t)$ on $[0, T]$ for initial data $x$ with $\td x  \in B_{\td x_0}( \d), b_0, \xi_0$ and 
\beq\label{eq:v_comp1} 
\td \g_t( B_{\td x_0 }(\d) ) \cup  B_{\td \g_t (\td x_0) }(\d) \subset \S_2 
\subset D_1^{+} \bsh \Ups \cap B_{(1,0)}(R_{2,\al}) , \quad t \in [0, T].
\eeq

\subsubsection{\bf{Construction of axisymmetric functions}}\label{sec:v_axi}

Our goal is to construct smooth (at least $C^4$) axisymmetric flows $\xi(t, x), b(t, x)$ satisfying \eqref{axi} and function $S(t, x)$ such that 
\begin{align}
&\xi(0, \td x, \vth_0 ) = \xi_0, \quad b(0,\td x, \vth_0 ) = b_0, \quad \xi(t, x) \cdot b(t, x) \equiv 0,  \label{eq:v_axi1} \\
&\xi(t, \g_t(\td x , \vth_0 )) = \xi_t( \td x, \vth_0,\xi_0 ), \quad b( t, \g_t(\td x, \vth_0))
= b_t( \td x, \vth_0, \xi_0, b_0) , \label{eq:v_axi12} \\
& \na S(t, x) = \xi(t, x) ,  \quad \pa_{\th} S(t, x) = \xi \cdot e_{\th} = 0, 
\label{eq:v_axi13} 
\end{align}
for any $\td x \in B_{\td x_0}(\d), t \in [0, T]$, where $\vth_0=x_{0,\vth}$ \eqref{eq:vth0} and $(\td x, \vth_0 )$ means  $(r, \vth_0, z)$ in the cylindrical coordinate. Thus, $b(t, x), \xi(t, x)$ can be seen as the
axisymmetric extensions of the solutions $\xi_t, b_t$ to the ODE \eqref{eq:bichar1}-\eqref{eq:bichar3} with initial data $(\td x, \vth_0),\xi_0, b_0$. We construct initial data as follows 
\beq\label{eq:v_axi_init}
\xi(0, x) = \xi_0^r e_r + \xi_0^z e_z, \quad b(0, x) = b_0^r e_r + b_0^{\vth} e_{\vth} + b_0^z e_z, 
\eeq
where $e_{r(x_0)} = ( \cos \vth_0, \sin \vth_0, 0 ), \ e_{\vth(x_0)} = ( - \sin \vth_0, \cos \vth_0, 0 )$, and 
\[
\xi_0^r = \xi_0 \cdot e_{r(x_0)} ,\   \xi_0^z = \xi_0  \cdot e_z,
\quad b_0^r = b_0 \cdot e_{r(x_0)} ,  \ b_0^{\vth} = b_0 \cdot e_{ \vth(x_0)}, \ b_0^z = b_0 \cdot e_z.
\]
The initial data $\xi(0, x), b(0, x)$ are axisymmetric and only depend on $x_{\vth}$ \eqref{eq:polar_basis}. From Lemma \ref{lem:axi}, $|\xi(0, x)|, |b(0,x)|, \xi(0, x) \cdot b(0, x)$ are independent of $\vth$. Using  \eqref{eq:lam_init} and \eqref{eq:v_axi_init}, we have 
\beq\label{eq:v_axi_init2}
\xi(0, \td x, \vth_0 ) = \xi_0, \  b(0,   \td x, \vth_0 ) = b_0 , \  |\xi(0, x)| = 1, \  | b(0, x)| = 1, \  \xi(0, x) \cdot b(0, x) = \xi_0 \cdot b_0 = 0.
\eeq

\subsubsection*{\bf{Localization of the velocity}}
We want to construct $\xi(t, x), b(t, x)$ using \eqref{eq:bichar2}-\eqref{eq:bichar3} with the above initial data. Yet, the singular solution $\uu$ is only $C^{1,\al}$ and the resulting solutions $\xi, b$ are not smooth enough. To fix this problem, we localize the velocity. From \eqref{eq:v_comp1}, using compactness, we can find a smooth cutoff function $\chi_T(r, z)$ such that 
\beq\label{eq:v_axi_u1}
\chi_T(\td x) = 1 , \quad \td x \in \S_2, \quad 
\S_2 \subset \supp(\chi_T) =\S_3  \subset D_1^{+} \bsh \Ups  \cap  B_{(1,0)}(R_{2,\al}) ,
\eeq
where $\S_3$ is another compact domain. Now, we modify the velocity $\uu$ as follows 
\beq\label{eq:v_axi_u2}
\uu_T(t, x) \teq \uu(t, x) \chi_T(r, z).
\eeq
From Lemma \ref{lem:axi} and Theorem \ref{thm:euler_blowup},  $\uu_T$ is axisymmetric and $\uu_T
 \in L^{\inf}( [0, T], C^{50}(D))$ is smooth in the whole domain. 

\subsubsection*{\bf{Constructions of $b, \xi, S$}}
Consider the PDE (Eulerian) formulations of \eqref{eq:bichar2}-\eqref{eq:bichar3} with the modified velocity $\uu_T$
\beq\label{eq:v_axi2}
\bal
\pa_t \xi + \uu_T \cdot \na \xi & = - (\na \uu_T)^T \xi , \quad
\pa_t b + \uu_T \cdot \na b  = - (\na \uu_T) b  + \f{2 \xi^T (\na \uu_T) b}{|\xi|^2} \xi 	
\eal
\eeq
and initial data $\xi(0,\cdot), b(0,\cdot)$. We will show that the evolution preserves the axisymmetry of $\xi, b$. For any axisymmetric functions $g, f$, using $\pa_{\vth} e_r = e_{\vth}, \pa_r e_{\vth} = -e_r$, we have 
\[
g \cdot \na f = (g^r \pa_r + \f{g^{\vth}}{r} \pa_{\vth} + g^z \pa_z) f 
= \sum_{\al = \al, \vth, z}( g^r \pa_r + g^z \pa_z) f^{\al} \cdot e_{\al}
+ \f{g^{\vth}}{r} ( f^r e_{\vth} - f^{\vth}  e_{r}),
\]
which is axisymmetric. Therefore, we obtain 
\[
\uu_T \cdot \na \xi,  \quad  (\na \uu_T) \xi = \xi \cdot \na \uu_T, \quad
\uu_T \cdot \na b , \quad (\na \uu_T) b
\]
are axisymmetric. Lemma \ref{lem:axi} implies that $\xi \cdot (\na \uu_T) b, |\xi|^2 = \xi \cdot \xi$ are independent of $\vth$. Thus $ \f{ \xi^T  (\na \uu_T) b}{|\xi|^2} \xi$ is axisymmetric. Using the identity 
\[
- (\na \uu_T)^T \xi =  ( \na \uu_T - (\na \uu_T)^T) \xi - (\na \uu_T ) \xi 
= (\na \times \uu_T) \times \xi - (\na \uu_T ) \xi  
\]
and Lemma \ref{lem:axi} again, we conclude that $- (\na \uu_T)^T \xi$ is axisymmetric. Therefore, the equations \eqref{eq:v_axi2} preserves axisymmetry. From \eqref{eq:v_axi2}, it is easy to see that 
\[
\pa_t( \xi \cdot b) + \uu_T \cdot \na ( \xi \cdot b) = 0.
\]
Recall the initial data \eqref{eq:v_axi_init}. From \eqref{eq:v_axi_init2}, we have $\xi(0, x) \cdot b(0, x) \equiv 0$.
The above transport equation implies that $ \xi(t, x) \cdot b(t, x) = 0$ in \eqref{eq:v_axi1}.

Next, we prove the identities in \eqref{eq:v_axi12}. First, for initial data $x$ with $\td x \in B_{\td x_0}(\d)$, due to \eqref{eq:v_comp1} and $\uu_T = \uu$ in $\S_2$ \eqref{eq:v_axi_u1}, \eqref{eq:v_axi_u2}, the flow maps on $[0, T]$ generated by $\uu_T$ and $\uu$ are identical. Hence, we obtain
 \[
\uu( t, \g_t(x)) = \uu_T(t, \g_t(x)), \quad (\na \uu)(t, \g_t(x)) = (\na \uu)( t, \g_t(x)).
\]
Using \eqref{eq:v_axi2} and the flow map $\g_t$ \eqref{eq:bichar1}, we have 
\[
\f{d}{dt}  \xi(t, \g_t(x)) = - (\na \uu)^T \xi(t, \g_t(x)), \quad 
\f{d}{dt} b(t, \g_t(x)) = - (\na \uu) b(t, \g_t(x)) + \f{2 \xi^T (\na \uu) b}{|\xi|^2} \xi(t, \g_t(x))
\]
where $\na \uu$ is evaluated at $(t, \g_t(x))$. Thus, $\xi(t, \g_t(x))$ and $b(t, \g_t(x))$ satisfy the same ODE \eqref{eq:bichar2}-\eqref{eq:bichar3} for $\xi_t, b_t$. According to Lemma \ref{lem:traj} and the discussion below \eqref{eq:vth0}, we can solve these ODEs for initial data $x$ with $\td x \in B_{\td x_0}(\d)$. Using \eqref{eq:v_axi_init2}, we get
\[
\xi(0, \g_0(\td x, \vth_0)) = \xi_0 = \xi_t( \td x, \vth_0, \xi_0) |_{t=0}, 
\quad b(0, \g_0(\td x, \vth_0)) = b_0 = b_t( \td x, \vth_0, \xi_0) |_{t=0}.
\]
Using the uniqueness of ODEs, we obtain \eqref{eq:v_axi12}.

To construct $S$, following \cite{vasseur2020blow,lafleche2021instability} we solve the transport equation with the modified velocity $\uu_T$
\beq\label{eq:v_axi_S}
 \pa_t S + \uu_T \cdot \na S = 0, \quad S(0, x) = r \xi_0^r + z \xi_0^z.
\eeq

The equation for $\na S$ reads
\[
\pa_t (\na S) + \uu_T \cdot \na ( \na S) = - (\na \uu_T)^T  (\na S ), \quad 
(\na S)(0, x) = \xi_0^r  e_r + \xi_0^z e_z = \xi(0, x).
\]
Comparing the above equations with \eqref{eq:v_axi2}, we yield $\na S(t, x) = \xi(t, x)$ for any $x$ and $t \in [0, T]$. 

Next, we consider $\pa_{\vth} S$. Since $\na S = \xi$ and $\uu_T$ are axisymmetric, using Lemma \ref{lem:axi}, we get 
\[
 \pa_{\vth}( \uu_T \cdot \na S) 
 =\pa_{\vth} (\uu_T \cdot \xi )= 0.
\]
Using \eqref{eq:v_axi_S} and $(\pa_{\vth} S )( 0, x)= 0$, we yield 
\[
\pa_t \pa_{\vth} S = 0, \quad \pa_{\vth} S (t, x) \equiv 0.
\]
This proves \eqref{eq:v_axi13}.



Since $\uu_T \in L^{\inf}([0, T], C^{50}(D) )$,  $\xi(t, x), b(t, x), S(t, x)$ are smooth and at least $C^4$ in $x$. 

\subsubsection{\bf{Control of $b(T, x)$} }\label{sec:b_control}
We will show that $b(T, x)$ can control $\b(T)$ via \eqref{eq:lam_pf0}. 

Recall the poloidal notation \eqref{eq:lam_polo}. Let $x_T = \g_T(x_0)$ and $L_{\g} \geq 1$  be the Lipschitz constant of $\g_t, \g_t^{-1}$ on $[0, T] \times D_1$. From \eqref{eq:lam_pf0} and \eqref{eq:v_axi12}, we get 
\[
 0 <  |b_T(x_0, b_0, \xi_0) | = |b(T, x_T)|.
\]
Using the  continuity of $b(T, \cdot)$, there exists small $\d_2$ with 
\beq\label{eq:lam_del}
\d_2  \in (0,  \f{ \d}{4 (L_{\g} + 1)^3} )
\eeq
such that 
\beq\label{eq:lam_pf4}
(1-\eta) |b_T(x_0, b_0, \xi_0)| = 
(1-\eta) |b(T, x_T )| 
\leq \infim_{ \td x \in B_{ \td x_T(\d_2) } } |b(T, \td x, x_{T,\vth} )|
= \infim_{ x \in A_{x_T}(\d_2) } |b(T, x)|.
\eeq
where we have used the continuity of $b(T, x)$ in the inequality, and the axisymmetry property that $|b(T, x)|$ is independent of $\vth$ in the third equality. Here, $A_{x_T}( \d_2)= \{ x : \td x \in B_{\td x_T}(\d_2) \}$  is an annulus. 
The above inequality reproduces Equation (19) in \cite{lafleche2021instability}.

\subsubsection{\bf{Construction of the axisymmetric velocity $v_{\e,\d}$}}

We follow \cite{lafleche2021instability,vasseur2020blow} to construct a cutoff function $\vp$ so that we can localize $b(T, x)$ to the domain where it is large 
using \eqref{eq:lam_pf4}. Let $\vp_T(x) = \vp_T(r, z)$ be a smooth function supported in $A_{x_T}( \d_2)$ with $|| \vp_T ||_p = 1$. For any $t \in [0, T]$, we define
\beq\label{eq:lam_vp}
\vp(t, x) \teq \vp_T( \g_T \circ \g_t^{-1}(x)).
\eeq
Since $ \vp_T$ is independent of $\vth$, using \eqref{eq:char_polo} and \eqref{eq:char_polo2}, we know that the $(r, z)$ component of $\g_T \circ \g_t^{-1}(x) $ only depends on $\td x$. Thus, we yield 
\[
\vp(t, x) = \vp_T( \td \g_T \circ \td \g_t^{-1}(  \td x))
\]
and $\vp(t, x)$ is independent of $\vth$. 

\begin{remark}
We can also solve $\vp(t, x)$ using the PDE similar to \eqref{eq:v_axi2}, \eqref{eq:v_axi_S}
\beq\label{eq:lam_vp2}
\pa_t \vp + \uu_T \cdot \na \vp = 0, \quad \vp(T, x) = \vp_T(x).
\eeq
Tracking the support of $\vp$ and using the argument similar to that in the proof of \eqref{eq:v_axi12}, one can show that these two constructions are the same.
\end{remark}

Using \eqref{eq:lam_vp} and \eqref{eq:char_polo2}, for $x \in \supp ( \vp(t, \cdot))$, we have $| \td \g_T \circ \td \g_t^{-1}(x) - \td \g_T(x_0)| \leq \d_2.$ Since $ \td \g_T \circ \td \g_t^{-1}$ has Lipschitz constant $L_{\g}^2$, from \eqref{eq:lam_del}, we get 
\[
| \td x -  \td \g_t(x_0)|  \leq L_{\g}^2 |  \td \g_T \circ \td  \g_t^{-1}(x) -\td \g_T(x_0)| \leq L_{\g}^2 \d_2,
\quad  \wt {\supp ( \vp(t, \cdot))} \subset B_{ \td \g_t(x_0)}( L_{\g}^2 \d_2) \subset B_{ \td \g_t(x_0) }(\d/2).  
\]



Using \eqref{eq:v_comp1},  we further obtain 
\beq\label{eq:lam_supp}
	\wt{ \supp ( \vp(t, \cdot) ) } 
\subset  \S_2 \subset
	 ( D_1^+ \bsh \Ups ) \cap B_{(1,0)}( R_{2,\al})\subset 	B_{(1,0)}(1/4) , \ t \in [0, T].
\eeq


For fixed $\eta, \d_2$, from Lemma \ref{lem:traj}, the function $\vp $ is Lipschitz in time and $C^4$ in $x$ on $[0, T] \times D_1$.  Moreover, from \eqref{eq:v_axi_u1}, \eqref{eq:v_axi_u2}, we get
\beq\label{eq:v_axi_u3}
\uu_T(t, x) = \uu(t, x),  \quad x \in \supp( \vp(t, \cdot)), \quad t \in [0, T].
\eeq


Now, we follow \cite{lafleche2021instability,vasseur2020blow} to construct an approximate solution \eqref{eq:WKB_euler} via the WKB expansion. Since $\xi, b$ are axisymmetric flows and $S(t, x), \vp(t,x), |\xi(t, x)|$ are independent of $\vth$, using Lemma \ref{lem:axi} repeatedly, we yield that 
\[
 b \times \xi,  \quad \f{b\times \xi}{|\xi|^2}  \vp e^{iS/\e}, \quad 
v_{\e, \eta} = \e \na\times ( \f{b\times \xi}{|\xi|^2}  \vp e^{iS/\e})
\]
are axisymmetric. We remark that $|\xi(t, x)|^{-1}$ is uniformly bounded on $[0, T] \times D_1$, which can be proved using the Lagrangian version of \eqref{eq:v_axi2}, the boundedness of $|\na \uu_T|$, and $|\xi(0, \cdot)| = |\xi_0| =1 $ \eqref{eq:v_axi_init}. Due to \eqref{eq:lam_supp}, $v_{\e, \eta}$ is supported in the interior of $D_1$ and $v_{\e, \eta} \cdot n = 0$ on $\pa D_1$.




Since $\supp(v_{\e, \eta})\subset \supp(\vp)$, from \eqref{eq:v_axi_u3}, the localization of $\uu$ in \eqref{eq:v_axi_u1} and \eqref{eq:v_axi_u2} does not change the estimates of $v_{\e, \eta}$ in \cite{lafleche2021instability,vasseur2020blow}. Following the argument in \cite{lafleche2021instability,vasseur2020blow}, we obtain that $v_{\e, \d_2}$ is a solution to \eqref{eq:euler_lin} with a small forcing term 
\beq\label{eq:lam_lin}
\pa_t v_{\e, \d_2} + ( \uu \cdot  \na) v_{\e, \d_2} + ( v_{\e, \d_2} \cdot \na) \uu
+ \na q_{\e, \d_2} =  R_{\e, \d_2}.  
\eeq
Moreover, we have the following estimates  
\beq\label{eq:lam_pf5}
\bal
&|| v_{\e, \d_2}(T) ||_{L^p} \geq (1 - \eta) |b(T, x_0, \xi_0)| - C_{\eta ,\d_2} \e , \\
&  || v_{\e, \d_2}(0, \cdot)||_{L^p}  \leq 1 + C_{\eta, \d_2} \e , \quad || R_{\e, \d_2} ||_{L^p} \leq C_{\eta , \d_2} \e, 
\eal
\eeq
where $C_{\eta, \d_2}$ is some constant independent of $\e$. The first two estimates are consequences of the leading order formula of $v_{\e, \eta}$ \eqref{eq:WKB_euler}, \eqref{eq:v_axi_init2}, \eqref{eq:lam_pf4}, and the conservation of $|| \vp(t, \cdot )||_{L^p} =1$, which follows from the fact that $\vp$ is transported by an incompressible flow, see e.g., \eqref{eq:lam_vp2}. See also Appendix \ref{app:WKB} for some formal derivations related to \eqref{eq:lam_lin}-\eqref{eq:lam_pf5}.

\subsubsection{\bf{Symmetrization}}\label{sec:sym}

An important observation is that $v_{\e, \d_2}$ is only supported in the upper half domain $D_1^+ \bsh \Ups$ due to \eqref{eq:lam_supp} and $\supp ( v_{\e, \d_2} ) \subset \supp ( \vp(t, \cdot) )$. 
For the singular solution $\uu$ in Theorem \ref{thm:euler_blowup}, $\om^{\vth}(t)$ is odd and $u^{\vth}(t)$ is even in $z$, which induces the symmetry property \eqref{eq:sym} that $u^z(t)$ is odd and $u^{\vth}(t), u^r(t)$ are even in $z$. 
For vector $f = v_{\e, \d_2}, R_{\e, \d_2}$, we extend it to $D_1^-$ according to the same symmetry 
\[
\bar f^r = f^r(r, z) + f^r (r, -z), \quad \bar f^z = f^z(r, z) - f^z(r, -z), \quad 
\bar f^{\vth} = f^{\vth}(r, z ) + f^{\vth}(r, - z), 
\]
where $ f = f^r e_r+  f^{\vth} e_{\vth} +  f^z e_z, \bar f = \bar f^r e_r+ \bar f^{\vth} e_{\vth} + \bar f^z e_z$. For the pressure $q_{\e, \d_2}$ in \eqref{eq:euler_lin}, we extend it as an even function in $z$ 
\[
\bar q_{\e, \d_2} = q_{\e, \d_2}(r, z) + q_{\e, \d_2}(r, -z).
\]

The above symmetry properties are preserved by  \eqref{eq:euler} and \eqref{eq:euler_lin}. We obtain that $\bar v_{\e, \d_2}$ is a solution to \eqref{eq:euler_lin} with pressure $\bar q_{\e, \d_2}$ and forcing $\bar R_{\e, \d_2}$ and enjoys the symmetry property \eqref{eq:sym}. Since $\supp ( v_{\e, \d_2} ) \in D_1^+$, $\bar v_{\e, \d_2} - v_{\e, \d_2}$ and $v_{\e, \d_2}$ are disjoint, applying \eqref{eq:lam_pf5} yields
\beq\label{eq:lam_pf6}
\bal
&||\bar v_{\e, \d_2}(T) ||_{L^p} \geq 2 (1 - \eta) |b(T, x_0, \xi_0)| - C_{\eta ,\d_2} \e,  \\
&  || \bar v_{\e, \d_2}(0, \cdot)||_{L^p}  \leq 2 + C_{\eta, \d_2} \e , \quad || \bar R_{\e, \d_2} ||_{L^p} \leq C_{\eta , \d_2} \e .
\eal
\eeq
The last inequality on $\bar R_{\e,\d_2}$ follows from the triangle inequality. Let $\bar v(T)$ be the solution to \eqref{eq:euler_lin} with initial data $\bar v_{\e, \d_2}(0)$. Following the argument in \cite{lafleche2021instability,vasseur2020blow}, we obtain 
\beq\label{eq:lam_pf7}
 || \bar v(T) -\bar v_{\e, \d_2}(T) ||_{L^p} \leq C_{\eta, \d_2} \e .
\eeq
Since the symmetry of $\bar v_{\e, \d_2}(0)$ in $z$ is preserved by \eqref{eq:euler_lin}, $v(T)$ satisfies the symmetry \eqref{eq:sym}.

\subsubsection{\bf{Control of $\lam_{p,\s}^{sym}$ for all power $\s$}}

Denote $\chi_2(x) = \one_{ B_{(1,0)}( \f{1}{2}) }(r, z)$. Since $\supp( v_{\e, \d_2}( t, \cdot) )= \supp( \vp(t, \cdot) ) \subset B_{ (1,0)}(1/4) $ \eqref{eq:lam_supp} and $\bar v_{\e, \d_2}$ is the symmetric extension of $ v_{\e, \d_2}$, we get $\chi_2 \bar v_{\e, \d_2} =\bar v_{\e, \d_2}$. Moreover, for $x \in  \supp ( \chi_2) \cap D_1  $, we get $r \in [1/2, 1]$. Then for any $\s \in \R$, using \eqref{eq:lam_pf5}, we obtain 
\[
|| r^{-\s} \bar v_{\e, \d_2}(0, \cdot) ||_{L^p}
 = || r^{-\s} \chi_2 \bar v_{\e, \d_2}(0, \cdot) ||_{L^p} \leq C_{\s} || \bar v_{\e, \d_2}(0, \cdot) ||_{L^p} 
 \leq C_{\s} (2 + C_{\eta, \d_2} \e) .
\]
Applying the above estimate, \eqref{eq:lam_pf4}, \eqref{eq:lam_pf6}, \eqref{eq:lam_pf7} and the definition \eqref{eq:instab1}, we yield
\[
\bal
C_{\s}(2 +& C_{\eta, \d_2}) \lam^{sym}_{p, \s}(T) \geq 
|| r^{-\s} \bar v_{\e, \d_2}(0, \cdot)||_{L^p} \lam^{sym}_{p, \s}(T) \geq || r^{-\s} \bar v(T)||_{L^p}
\geq || r^{-\s} \chi_2  \bar v(T)||_{L^p} \\ 
& \geq \td C_{\s}  ||  \chi_2  \bar v(T)||_{L^p}  
\geq \td C_{\s}  (  || \chi_2 \bar v_{\e, \d_2}(T) ||_{L^p} - || \chi_2 ( \bar v(T) -\bar v_{\e, \d_2}(T) ) ||_{L^p} )    \\
 & \geq  \td C_{\s} ( || \bar v_{\e, \d_2}(T) ||_{L^p} - C_{\eta, \d_2 } \e) 
 \geq \td C_{\s} ( 2 (1-\eta) b(T, x_0, \xi_0) -  C_{\eta, \d_2 } \e)
\geq \td C_{\s} \B( \f{1-\eta}{1 + \eta} \b(T) - C_{\eta, \d_2 } \e  \B).
\eal
\]
Taking $\eta = 1/2$ and letting $\e \to 0$ conclude the proof.
\end{proof}

\begin{proof}[Proof of Theorem \ref{thm:Euler_HL}]
From Theorem \ref{thm:euler_blowup}, we have 
$
\lim_{t \to T^*} ||  \om_p(t) ||_{\inf} = \inf.
$
Combining Propositions \ref{prop:beta} and \ref{prop:lam}, we establish
\[
\liminf_{t \to T_* } \lam^{sym}_{p, s}(t)^2 \geq C_{\s} \liminf_{ t \to T_*} \b^2(T) 
\geq  C_{\s} \lim_{ t \to T^*} \f{ ||  \om_p(t) ||_{\inf} }{ ||  \om_{p,0} ||_{\inf}} = \inf.
\]
We conclude the proof of Theorem \ref{thm:Euler_HL}. 
\end{proof}

\subsection{Proof of Theorem \ref{thm:bous}}

The proof of Theorem \ref{thm:bous} is completely similar to that of Theorem \ref{thm:Euler_HL} and is easier. We follow the arguments in \cite{shao2022instability}. Firstly, we note that there is a sign difference between the Boussinesq equations used in \cite{chen2019finite2} \eqref{eq:bous} and \cite{shao2022instability}. 
In \cite{shao2022instability}, the Boussinesq equations are given by 
\beq\label{eq:bous_v2}
 \th_t + \uu \cdot  \na \th = 0, \quad  \uu_t + \uu \cdot \na \uu + \na p = (0, \th)^T,  \quad \na \cdot \uu = 0.
\eeq
The velocity-density formulation of \eqref{eq:bous} is the above equations with $(0, \th)^T$ replaced by $(0, -\th)^T$. Clearly, \eqref{eq:bous} and \eqref{eq:bous_v2} are equivalent:  $(\uu, \th)$ solves \eqref{eq:bous} if and only if $(\uu, -\th)$ solves \eqref{eq:bous_v2}. The linearized equation of \eqref{eq:bous_v2} around a solution $(\uu, \th)$ of \eqref{eq:bous_v2} is given by 
\beq\label{eq:bous_lin}
\bal
&\pa_t \eta + \uu \cdot \na \eta + \vv \cdot \na \th  = 0, \quad \pa_t \vv + \uu \cdot \na \vv + \vv \cdot \na \uu + \na q =   (0, \eta)^T, \quad  \mathrm{div} \ \vv = 0,  \\
\eal
\eeq
which is also different from \eqref{eq:bous_lin0} with $(0, \eta)^T$ in \eqref{eq:bous_lin} replaced by $(0, -\eta)^T$ in \eqref{eq:bous_lin0}. Given solution $(\uu, \th)$ of \eqref{eq:bous} and $(v, \eta, \uu, \th)$ satisfying \eqref{eq:bous_lin0}, we obtain that $(\uu, -\th )$ is solution of \eqref{eq:bous_v2} and $(v, -\eta, \uu, -\th)$ satisfies \eqref{eq:bous_lin}. To keep the minimal changes of sign and other notations among this paper, \cite{chen2019finite2}, and \cite{shao2022instability}, due to this connection, we use the following setting. 
Given a singular solution $(\uu, - \th)$ of \eqref{eq:bous} in Theorem \ref{thm:bous_blowup}, we obtain the solution $( \uu,  \th)$ of \eqref{eq:bous_v2}, which satisfies the same properties in Theorem \ref{thm:bous_blowup}, e.g., the blowup quantities and the regularity. Then we consider \eqref{eq:bous_v2} and \eqref{eq:bous_lin} in the following discussions so that the derivations 
and notations are consistent with those in \cite{shao2022instability}.

The bicharacteristics-amplitude ODE system of \eqref{eq:bous_v2}  \cite{shao2022instability} read
\begin{align}
\dot \g(t, x_0) & = \uu( t, \g(t, x_0)) , \label{eq:bichar1_bous} \\
\dot \xi(t, x_0, \xi_0) & = - (\pa_x \uu)^T \xi(t, x_0, \xi_0),  \label{eq:bichar2_bous} \\
\dot b( t, x_0, \xi_0) & = - (\pa_x \vec z) b + \LL b + (2 \f{ \vec  \xi^T  ( \pa_x \vec z)  b  }{|\xi|^2}  - \f{ \vec \xi \cdot (\LL b)}{ |\xi|^2} ) \vec \xi,  \label{eq:bichar3_bous}
\end{align}
where $\vec z \teq (\th, \uu)$, $b\in \R^3$, the matrix $\pa_x \vec z$, vector $\vec \xi$, and linear operator $\LL$ are given below
\beq\label{eq:nota_bous}
\pa_x \vec z \teq 
\left(
\begin{array}{ccc}
0 & \pa_1 \th & \pa_2 \th \\
0 & \pa_1 u_1  & \pa_2 u_1 \\
0 & \pa_1 u_2  & \pa_2 u_2 \\
\end{array}
\right), \quad 
\vec \xi \teq 
\left(
\begin{array}{c}
0 \\
\xi_1 \\
\xi_2 \\
\end{array}
\right),
\quad 
\LL b \teq 
\left(
\begin{array}{c}
0 \\
0 \\
b_1 \\
\end{array}
\right).
\eeq
The initial data is given by $ \g|_{t=0} = x_0, \xi |_{t = 0} = \xi_0 \in \R^2 \bsh \{ 0\}$ and $ b|_{t = 0} = b_0 \in \R^3$. 
Denote 
\beq\label{eq:bous_domain}
\Ups_2 \teq \{ (x, y) \in \R^2_+: x = 0 \mathrm{ \ or \ } y = 0\}, \quad 
D = \R^2_+, \quad  D^{\pm} \teq \{ (x, y): y \geq 0, \pm x \geq  0 \}.
\eeq

For the singular solution $(\uu, -\th)$ in Theorem \ref{thm:bous_blowup} (then $(\uu, \th)$ solves \eqref{eq:bous_v2}), since $\om$ is odd, $\th$ is even in $x$, and $v(x, 0) = 0$, we have 
\beq\label{eq:bous_noflow}
\uu \cdot n |_{\Ups_2} = 0,
\eeq
where $n$ is the normal vector of $\Ups_2$. We first generalize Lemma \ref{lem:traj} as follows.
\begin{lem}\label{lem:traj_bous}
Let  $\g_t$ be the solution to \eqref{eq:bichar1_bous} with initial data $x_0$, $T_*$ be the blowup time, $T < T_*$, and $D^{\pm}$ be the domains defined in \eqref{eq:bous_domain}. (a) For any $x_0 \in \Ups_2$ and $ t \in [0, T_*)$, its trajectory $\g_t$ remains in $\Ups_2$; for any $x_0 \in D_1^{\pm} \bsh \Ups_2$ and $t \in [0, T_*)$, we have $\g_t \in D^{\pm} \bsh \Ups_2$. For any $t \in [0, T]$, $\g_t$ is invertible, and $\g_t, \g_t^{-1}$ are Lipschitz in time and the initial value.

(b) For $x_0 \in D^{\pm} \bsh \Ups_2$, there exists $\d(x_0, T) > 0$ depending on $x_0, T$ and a compact set $\Sigma_2$ such that 
\beq\label{eq:traj_bous}
\g_t(  B_{x_0 }(  \d) ) \cup B_{ \g_t(x_0)}( \d)
\subset  \Sigma_2 
\subset D^{\pm} \bsh \Ups_2
.\eeq
As a result, for initial data $z_0 \in B_{x_0}(  \d)$ and any $b_0, \xi_0$, there exista a unique solution $(\g_t, b_t, \xi_t)$ to \eqref{eq:bichar1_bous}-\eqref{eq:bichar3_bous} on $t\in [0, T]$. The functions $(\g_t, b_t, \xi_t)$ are Lipschitz in time and $C^4$ with respect to initial data $z_0 \in B_{x_0}( \d)$ and $b_0, \xi_0$, and $\g_t^{-1}(x)$ is Lipschitz in time and $C^4$ in $x \in \g_t( B_{x_0}(  \d)  ) \cup B_{\g_t(x_0)}(\d)$.
\end{lem}

Unlike Lemma \ref{lem:traj} for the 3D Euler equations, in the above Lemma, since it is in 2D, we do not need to consider the angular variable $\vth$ and the poloidal component $\td x$ \eqref{eq:polo}. 
Moreover, unlike  \eqref{eq:traj_tube3}, we do not restrict the initial data $x_0$ and the trajectory $\g_t(x_0)$ to a domain near the singularity $(0, 0)$ since the velocity $\uu(t)$ in Theorem \ref{thm:bous_blowup} is smooth in any interior compact domain in $\R_2^+$.
The proof of Lemma \ref{lem:traj_bous} follows from the non-penetrated condition \eqref{eq:bous_noflow}, the regularity $\uu, \th \in C^{1,\al}$ and $\uu ,\th \in C^{50}(\Sigma)$ for any compact set $\Sigma \subset D^{\pm} \bsh \Ups_2$ from Theorem \ref{thm:bous_blowup}, and the same argument in the proof of Lemma \ref{lem:traj}.

We adopt the following notation from \cite{shao2022instability}  by replacing the domain $D$ by $D \bsh \Ups_2$
\[
\al(T) \teq \sup_{ |b_0| = 1, |  \xi_0| = 1, x_0 \in D \bsh \Ups_2, b_0 \cdot \vec \xi_0 = 0} | b(T, x_0, \xi_0, b_0)|.
\]
Recall from \eqref{eq:nota_bous} that $b_0, \vec \xi_0 \in \R^3, \xi_0 \in \R^2$.  From Lemma \ref{lem:traj_bous}, for $x_0 \in D\bsh \Ups_2$, $b(T, x_0, \xi_0, b_0)$ and $\al(T)$ are well-defined. We modify Proposition 3.1 from \cite{shao2022instability} as follows.
\begin{prop}\label{prop:al_bous}
Assume that $(\uu, -\th)$ is the singular solution in Theorem \ref{thm:bous_blowup}. 
Then $(\uu, \th)$ is the  singular solution of \eqref{eq:bous_v2}. For any $ t \in (0, T^*)$, we have 
\[
 ||\na \th(T) ||_{\inf} 
 \leq ||  \na \th_0||_{\inf} \al^2(T).
\]
\end{prop}

Note that $\na \th \in C^{\al}$ is continuous, and we can solve \eqref{eq:bichar1_bous}-\eqref{eq:bichar3_bous} for $ x_0 \in D^{\pm} \bsh \Ups_2 $ from Lemma \ref{lem:traj_bous}. The proof 
follows from the proof of Proposition 3.1 in \cite{shao2022instability} with minor modifications similar to those in the proof of Proposition \ref{prop:beta}. Thus, we omit the proof.

We modify Proposition 3.2 from \cite{shao2022instability} as follows.
\begin{prop}\label{prop:lam_bous}
Assume that $(\uu, -\th)$ is the singular solution in Theorem \ref{thm:bous_blowup}. 
Then $(\uu, \th)$ is the  singular solution of \eqref{eq:bous_v2}.
For any $ T \in (0, T^*)$ and $p \in (1, \inf)$, we have 
\[
\al(T) \leq C_p \g^{sym}_p(T). 
\]
\end{prop}

The proof follows from the argument in \cite{shao2022instability} and the argument in the proof of Theorem \ref{prop:lam}.
The key point is that the approximate solution $( \eta_{\e, \d}, v_{\e, \d})$ constructed in \cite{shao2022instability} is similar to \eqref{eq:WKB_euler} and supported in a compact domain  $ \Sigma_2 \subset  D^{\pm} \bsh \Ups_2$. See \eqref{eq:lam_supp} for the case of the 3D Euler equations. The proof is much simpler since we do not need to construct an axisymmetric solution. 

We give a sketch of the proof. We fix $T < T_* $. For any initial data $x_0 \in D^{\pm } \bsh \Ups_2$ and $b_0, \xi_0$ with $b_0 \cdot \vec \xi_0 = 0, |b_0|=1, |\xi_0|=1$, from Lemma \ref{lem:traj_bous}, there exists $\d > 0$ and a compact set $\S_2$ such that \eqref{eq:traj_bous} holds. Without loss of generality, we assume $x_0 \in D^+ \bsh \Ups_2$. We construct a smooth cutoff function $\chi_T$ similar to \eqref{eq:v_axi_u1}  such that 
\[
\chi_T(x) = 1, \quad x \in \S_2, \quad \S_2 \subset \supp(\chi_T) = \S_3 \subset  D^{+} \bsh \Ups_2.
\]
We localize the singular solution $(\uu, \th)$ similar to \eqref{eq:v_axi_u2} as follows 
\beq\label{eq:v_bous1}
\uu_T(t, x) \teq \uu(t, x) \chi_T(x), \quad \th_T(t, x) \teq \th(t, x) \chi_T(x).  
\eeq
From Theorem \ref{thm:bous_blowup}, we get $\uu_T, \th_T \in L^{\inf}([0, T], C^{50}(D)) $. Then we construct $b( t, x ), \xi(t, x), \g(t, x)$ by solving the PDE (Eulerian) form of \eqref{eq:bichar1_bous}-\eqref{eq:bichar3_bous} with $\uu, \th, \vec{z}$ replaced by $\uu_T, \th_T, z_T = (\th_T, \uu_T)$ using the following initial data 
\[
b(0, x) \equiv b_0, \quad \xi(0, x ) = \xi_0, \quad S(0, x) = x \cdot \xi_0.
\]
We choose $\d_2 = \f{\d}{4 (1 + L_{\g})^3}$ similar to \eqref{eq:lam_del} and choose $\vp_T$ that is supported in $B_{\d_2}(\g_T(x_0))$ with $|| \vp_T ||_{L^p}=1$. Then we construct a localized function with properties similar to \eqref{eq:lam_supp}
\[
\vp(t, x) = \vp_T( \g_T \circ \g_t^{-1}(x) ) , \quad \supp( \vp(t, x)) \subset B_{\g_t(x_0)}(\d/ 2) \subset \S_2.
\]
These functions $b(t,x), \xi(t,x), S(t,x), \vp(t,x)$ are at least $C^4$ in the whole domain for $t \in [0, T]$. From \eqref{eq:v_bous1}, we have
\[
\uu_T(t, x) = \uu(t, x), \quad \th_T(t, x) = \th(t, x), \quad x \in \supp(\vp(t)) \subset \S_2, \quad t \in [0, T]. 
\]
Using these functions $b, \xi, S, \vp$, we follow \cite{shao2022instability} to construct the WKB solution, which is supported in $\supp(\vp(t)) \subset \S_2$. Due to the above relation, the localization \eqref{eq:v_bous1} does not change the estimates of the solution. We can further symmetrize the solution using the argument in Section \ref{sec:sym}. The rest of the proof follows \cite{shao2022instability}.


One difference between our settings and those in \cite{shao2022instability} is that our domain $\R_2^+$ has boundary, while the domain in \cite{shao2022instability} is $\R^2$ or $\BT^2$. In the proof of Proposition \ref{prop:lam_bous}, this difference appears only in the elliptic estimate 
\[
\bal
- \D q   = \na \cdot g , \ x \in  \R_2^+, \quad  - \f{\pa q}{\pa n}  =  n \cdot g, \ \mathrm{\ on \ } \pa \R_2^+,
\eal
\]
where $n$ is the unit normal vector. In \cite{shao2022instability}, there is no boundary and the second equation. In $\R_2^+$, the $L^p$ estimate 
\[
|| \na q||_{L^p} \les_p || g||_{L^p}, \quad p \in (1, \inf)
\]
follows from the Poisson's formula for $q$ and the Calderon-Zygmund estimates of the kernel.

Now, we are in a position to prove Theorem \ref{thm:bous}. The proof is simpler than that in \cite{shao2022instability} since we do not require the blowup criterion on $\int_0^T || \na \th ||_{\inf} dt$.
\begin{proof}[Proof of Theorem \ref{thm:bous}]
From Theorem \ref{thm:bous_blowup}, we have 
$
\lim_{t \to T^*} ||  \na \th(t) ||_{\inf} = \inf.
$
Combining Propositions \ref{prop:al_bous} and \ref{prop:lam_bous}, we establish
\[
\liminf_{t \to T_* } \g^{sym}_{p}(t)^2 \geq C_p \liminf_{ t \to T_*} \al^2(T) 
\geq  C_p \lim_{ t \to T^*} \f{ || \na \th(t) ||_{\inf} }{ || \na \th_0 ||_{\inf}} = \inf.
\]
We conclude the proof of Theorem \ref{thm:bous}. 
\end{proof}

\subsection{ Proof of Theorem \ref{thm:Euler_R3}}


For the singular solution \cite{elgindi2019finite}, near the singularity $(r, z) = (0, 0)$, the flow moves down the $z$ axis, and then travel outward in the $r$ direction. See also Remark 2.1 in \cite{elgindi2019finite}. We will use the outward flow to prove Theorem \ref{thm:Euler_R3}. Denote
\beq\label{eq:def_beta_R3}
\bal
\Ups_3 & \teq \{ (r, z) :  r =0 \mathrm{ \ or \ } z = 0 \},  \quad 
\td \b_{\s}(t) &= \sup_{ \substack{ (x_0, b_0, \td \xi_0) \in (\R^3 \bsh \Ups_3) \times \R^3 \times S^1,  \\ b_0 \cdot \xi_0 =0, |b_0| = r_0^{\s} } }   | r_t^{-\s} b_t(x_0, \td \xi_0, b_0)|. 
\eal
\eeq
The definition of $\td \b_{\s}(t)$ modifies \eqref{eq:def_beta0} and  is similar to \eqref{eq:def_beta}.  The velocity $u^r, u^z$ in Theorem \ref{thm:euler_R3_blowup} satisfies 
\beq\label{eq:vel_noflow_R3}
\uu(t) \cdot n \B|_{\Ups_3 } = u^r(t) \cdot n^r + u^z(t) \cdot n^z(t) = 0,
\eeq
and $(u^r, u^z)$ is smooth in $\R^3 \bsh \Ups_3$ and $u^{\vth } = 0$. In particular, $\g_t$ is a bijection from $\R^3 \bsh \Ups_3$ to $\R^3 \bsh \Ups_3$. Hence, we can generalize Lemma \ref{lem:traj} to the current setting, and solve \eqref{eq:bichar1}-\eqref{eq:bichar3} in $ \R^3 \bsh \Ups_3$ with solutions $b_t, \g_t, \xi_t, \g_t^{-1}$ that are $C^4$ on the initial data.

The following result is established in the proof of Proposition 2 in \cite{lafleche2021instability}.
\begin{prop}
For any $(T, x_T)  \in (0, T^*) \times \R^3 \bsh \{ r = 0\}$
and $\s \in \R$, let $x_t$ be the backward solution of \eqref{eq:bichar1} from time $T$ and $x_T$, $\om_0 = \om(0, x_0)$, $\xi_t$ be the solution of \eqref{eq:bichar2} with initial data $\xi_0 \cdot \om_0 = 0, \xi_0 \neq 0, \xi_0 \cdot e_{\vth(x_0)} = 0$, and $b_t$ be a solution of \eqref{eq:bichar3} with initial data $b_0 = r_0^{\s} e_{\vth}$ and $b_0 \cdot \xi_0 = 0$. Then we have 
$ r_0^{\s + 1} \leq r_T |b_T|$. 
\end{prop}

Applying the above result to $x_T \in \R^3 \bsh \Ups_3 \subset \R^3 \bsh \{ r = 0\}$ and using definition \eqref{eq:def_beta_R3}, 
we yield 
\[
 \f{r_0^{\s+1}}{ r_T^{\s+1}} \leq r_T^{-\s} |b_T| \leq \td \b_{\s}(T).
\]
Since $x_T = \g_T(x_0)$ is arbitrary in $\R^3 \bsh \Ups_3$ and $\g_T$ is a bijection from $\R^3 \bsh \Ups_3$ to itself, we derive
$
\sup_{ x_0 \in \R^3 \bsh \Ups_3}   \f{r_0^{\s+1}}{ r_T^{\s+1}}  \leq  \td \b_{\s}(T).
$
Since $\uu \in C^{1,\al}$ and $\g_T(x)$ is Lipschitz in $x$, we get $ \g_T( r_0 , \vth_0, 0) = \lim_{ z \to 0} \g_T(r_0, \vth_0, z)$. Hence, we further obtain 
\beq\label{eq:expand1}
\sup_{ x_0 \in \R^3 \bsh \{ r= 0 \} }   \f{r_0^{\s+1}}{ r_T^{\s+1}}  \leq  \td \b_{\s}(T).
\eeq

We have the following estimate for $r_T / r_0$. The idea is that the outgoing flow in the $r$ direction near $(r, z) = (0, 0)$ generates rapid growth of $r_T / r_0$.
\begin{lem}\label{lem:expand}
Let $\uu$ be the singular solution in Theorem \ref{thm:euler_R3_blowup}. Then for any $T < T_*$, we have 
\[
\sup_{ r_0 \neq 0} \f{r_T}{r_0} \geq \exp\B( \f{1}{2} \int_0^{T} u_r^r( t, 0,0) dt \B).
\]
\end{lem}

\begin{proof}
Note that $u^r( t, 0, z) = 0$. For $T < T_*$, since $u^r(t) \in  C^0( [0, T], C^{1,\al})$ and $u^r_r(t, 0, 0) > 0, t \in [0, T]$, there exists $ \d > 0$, such that 
\beq\label{eq:expand_pf1}
 0 <  \f{1}{2} u^r_r(  t, 0, 0)  
 \leq \f{u^r( t, r, 0 )}{r} \leq 2 u^r_r(  t, 0, 0)  .
\eeq
for all $ r \leq \d, t\in [0, T]$. Since $u^z( t, r, 0) = 0$, solving the $r$ component of the ODE \eqref{eq:bichar1} backward with initial data $x_T = (r_T, 0), r_T = \d/2$, 
we get that the trajectory is on $z = 0$ and 
\[
\f{d}{dt} r_{T-t} = - u^r( T-t, r_{T-t}, 0) 
=  - r_{T-t}  \f{u^r( T-t, r_{T-t} , 0) } { r_{T-t}}
\]
Since $u^r(T-t, r, 0) \geq 0$ on $r \in [0, \d]$, $r_{T-t}$ is decreasing in $t$ and $r_{T-t} \in [0, \d]$. Using the above ODE, \eqref{eq:expand_pf1}, and Gronwall's inequality, we obtain
 \[
 r_0 \leq  \exp( -\f{1}{2}\int_0^T u_r^r( t,0,0) dt)  r_T, \quad  
 r_0 \geq \exp( -2 \int_0^T u_r^r( t,0,0) dt)  r_T > 0 .
 \]
 The desired result follows. 
\end{proof}

For the singular solution in Theorem \ref{thm:euler_R3_blowup} \;, Proposition 3 in \cite{lafleche2021instability} remains true.
\begin{prop}\label{prop:lam_R3}
Let $ t \in (0, T_*), p \in (1, \inf), \s \in (-\f{2}{p^{\prime}}, \f{2}{p} )$ and $\uu$ be the singular solution in Theorem \ref{thm:euler_blowup}.  Then we have $ \td \b_p(T) \les  \lam^{sym}_{p, \s}(T)$.
\end{prop}

From Theorem \ref{thm:euler_R3_blowup}, for any $T < \infty$ and any compact domain $\S \subset \R^3 \bsh \Ups_3$,  we can localize $\uu$ using some cutoff function such that $\uu(t, x) \chi(x) = \uu(t, x)$ for $(x, t) \in \S \times [0, T]$, and $\uu \chi$ is much smoother.
The weighted estimate involving the weight $r^{-\s}$ in Lemma 4.1 in \cite{lafleche2021instability} does not require higher order regularity on $\uu$. Thus the proof follows from \cite{lafleche2021instability} and the proof of proposition \ref{prop:lam}.

Now, we are in a position to prove Theorem \ref{thm:Euler_R3}
\begin{proof}[Proof of Theorem \ref{thm:Euler_R3}]
From Theorem \ref{thm:euler_R3_blowup}, we have $\int_0^{T_*} u_r^r(t, 0,0) dt = \inf$. For $\s < -1$, $-\s - 1 > 0$, combining Lemma \ref{lem:expand} and \eqref{eq:expand1}, we obtain
\[
\lam^{sym}_{p, \s}(T) \geq C \td \b_{\s}(T) \geq C \sup_{r_0 \neq 0} (\f{ r_0}{r_T})^{\s+1}
= C (\sup_{r_0 \neq 0} \f{r_T}{r_0})^{-1-\s}
\geq C\exp( \f{ -1- \s}{2} \int_0^T u_r^r(t, 0, 0) dt).
\]
Letting $T \to T_*$, we complete the proof.
\end{proof}

\section{Properties of the singular solutions to the 2D Boussinesq equations}\label{sec:high}


In this Section, we prove Theorem \ref{thm:bous_blowup} regarding the properties of the singular solutions to the 2D Boussinesq equations \eqref{eq:bous} constructed in \cite{chen2019finite2}. 
In Section \ref{sec:euler_blowup}, we generalize these estimates to the 3D Euler equations with boundary using the connection between the 2D Boussinesq and the 3D Euler equations \cite{majda2002vorticity,chen2019finite2} and the argument in \cite{chen2019finite2}.

We remark that  we will only use the higher-order interior regularity 
in Theorems \ref{thm:bous_blowup}-\ref{thm:euler_R3_blowup}  \textit{qualitatively} to prove Theorems \ref{thm:Euler_HL}-\ref{thm:bous}. These estimates could be established by performing energy estimates of the physical equations directly with a continuation criterion similar to the BKM criterion \cite{beale1984remarks}. Yet, since the singular solution  has only low regularity in the whole domain, we need some delicate weighted estimates. 
Instead, we prove these regularity estimates by generalizing the nonlinear stability estimates in \cite{chen2019finite2} to the higher order and using embedding inequalities. These \textit{quantitative} stability estimates can be useful for future study of the singular solution.


\subsection{Setup for the 2D Boussinesq equations}

Firstly, we recall the setup from \cite{chen2019finite2}.

\subsubsection{Dynamic rescaling formulation}\label{sec:dsform}

The analysis of the singular solutions \cite{chen2019finite2} is based on the dynamic rescaling formulation \cite{mclaughlin1986focusing,landman1988rate}. To distinguish the solutions to \eqref{eq:bous}-\eqref{eq:biot} and the solutions to its dynamic rescaling formulation, we denote by $ \om_{phy}(x, t), \th_{phy}(x,t) , \uu_{phy}(x, t)$
the solutions of \eqref{eq:bous}-\eqref{eq:biot}. Then it is easy to show that 
\beq\label{eq:rescal1}
\bal
  \om(x, \tau) &= C_{\om}(\tau) \om_{phy}(   C_l(\tau) x,  t(\tau) ), \quad   \th(x , \tau) = C_{\th}(\tau)
  \th_{phy}( C_l(\tau) x, t(\tau)),  \\
    \uu(x, \tau) &= C_{\om}(\tau)  C_l(\tau)^{-1} \uu_{phy}(C_l(\tau) x, t(\tau)) , 
\eal
\eeq
are the solutions to the dynamic rescaling equations
 \beq\label{eq:bousdy1}
\bal
\om_{\tau}(x, \tau) + ( c_l(\tau) \xx + \uu ) \cdot \na \om  &=   c_{\om}(\tau) \om + \th_x , \qquad 
\th_{\tau}(x , \tau )+ ( c_l(\tau) \xx + \uu ) \cdot \na \th  = 0,
\eal
\eeq
where $\uu = (u, v)^T = \na^{\perp} (-\D)^{-1} \om $, $\xx = (x, y)^T$, 
\beq\label{eq:rescal2}
\bal
  C_{\om}(\tau) = \exp\lt( \int_0^{\tau} c_{\om} (s)  d \tau\rt), \ C_l(\tau) = \exp\lt( \int_0^{\tau} -c_l(s) ds    \rt) , \  C_{\th}  =  \exp\lt( \int_0^{\tau} c_{\th} (s)  d \tau\rt),
\eal
\eeq
$  t(\tau) = \int_0^{\tau} C_{\om}(\tau) d\tau $ and  the rescaling parameter $c_l(\tau), c_{\th}(\tau), c_{\om}(\tau)$ satisfies 
\[
c_{\th}(\tau) = c_l(\tau ) + 2 c_{\om}(\tau).
\]


We have the freedom to choose the time-dependent scaling parameters $c_l(\tau)$ and $c_{\om}(\tau)$ according to some normalization conditions. After we determine the normalization conditions for $c_l(\tau)$ and $c_{\om}(\tau)$, the dynamic rescaling equation is completely determined and the solution of the dynamic rescaling equation is equivalent to that of the original equation using the scaling relationship described in \eqref{eq:rescal1}-\eqref{eq:rescal2}, as long as $c_l(\tau)$ and $c_{\om}(\tau)$ remain finite. We refer more discussion about this reformulation for the 2D Boussinesq equations to \cite{chen2019finite2}.

To simplify our presentation, we still use $t$ to denote the rescaled time.

\subsubsection{ Change of coordinates and the approximate steady state }

 Consider the polar coordinate in $\R_2^+$
\[
r = \sqrt{x^2 + y^2}, \quad \b = \arctan(y / x), \quad R = r^{\al}.
\]
Let $\om, \th, \psi = (-\D)^{-1} \om$ be the vorticity, density, and the stream function in \eqref{eq:bousdy1}. Denote 
\beq\label{eq:nota_om}
\Om(R, \b, t) = \om(x, y, t), \quad \Psi = \f{1}{r^2} \psi, \quad \eta( R, \b, t) = (\th_x) ( x, y, t), \quad \xi( R, \b, t) = (\th_y)(x, y, t).
\eeq

Using the $(R, \b)$ coordinates and the above new variables, we reformulate \eqref{eq:bousdy1} as follows 
\beq\label{eq:bousdy2}
\bal
\Om_t  + \al c_l R \pa_R \Om + (\uu \cdot \na)  \Om  &=  c_{\om} \Om + \eta , \\
\eta_t + \al c_l R \pa_R \eta + (\uu \cdot \na)  \eta & = (2 c_{\om} -  u_x) \eta - v_x \xi ,\\
\xi_t + \al c_l R \pa_R \xi + (\uu \cdot \na)  \xi & = (2 c_{\om}- v_y) \xi -  u_y \eta .\\
\eal
\eeq
The elliptic equation \eqref{eq:biot} reduces to
\beq\label{eq:elli}
- \al^2 R^2 \pa_{R R} \Psi- \al(4+\al) R \pa_R \Psi - \pa_{\b \b} \Psi - 4 \Psi = \Om , 
\eeq
with  boundary conditions
\beq\label{eq:ellibc}
\Psi(R,0) = \Psi(R, \pi/2) =0, \quad \lim_{R\to \infty} \Psi(R, \b) =  0.
\eeq

Here, instead of working with the equations of $(\om, \th)$ \eqref{eq:bousdy1}, we consider the equations of $(\om, \th_x, \th_y)$ in \eqref{eq:bousdy2} since these variables have similar regularities. 

The above polar coordinates $(R,\b)$ and the change of variables from $\om,\psi$ to $\Om, \Psi$ were first introduced in \cite{elgindi2019finite}. The advantages of these transforms are that for small $\al$, the solution can be expressed as a smooth function of $R$, and the structure of the equations \eqref{eq:bousdy2} and \eqref{eq:elli} are clear under the $(R,\b)$ coordinates. 

The approximate steady state of \eqref{eq:bousdy2} under the coordinate $(R, \b)$ is given by 
\beq\label{eq:profile}
\bal
\bar{\Om}(R, \b)& =  \f{\al}{c} \G(\b) \f{ 3R }{ (1 + R)^2}, \quad \bar{\eta}(R, \b) =  \f{\al}{c} \G(\b) \f{ 6 R }{ (1 + R)^3} , \quad \bar{c}_l  = \f{1}{\al} + 3, \quad \bar{c}_{\om} = -1 ,\\
\G(\b) &= (\cos(\b))^{\al}, \qquad  c =  \f{2}{\pi} \int_0^{\pi/2} \G(\b) \sin(2\b) d\b.
\eal
\eeq


We decompose a solution $ (\hat \Om, \hat \eta, \hat \xi, \hat c_l, \hat c_{\om})$ of \eqref{eq:bousdy2} into the approximate steady state and their perturbations 
\[
\hat \Om = \bar \Om + \Om, \quad  \hat \eta = \bar \eta +  \eta, \quad \hat \xi = \bar \xi + \xi,
\quad \hat c_l = \bar c_l + c_l , \quad \hat c_{\om} = \bar c_{\om} + c_{\om}.
\]

To uniquely determine the dynamic rescaling formulation, we impose the following normalization conditions on the perturbation of the rescaling parameters $c_l(t), c_{\om}(t)$ 
\beq\label{eq:normal}
c_{\om}(t) = -\f{2 }{\pi \al} L_{12}(\Om(t))(0), \quad c_l(t) =  - \f{1-\al}{\al} \f{2}{\pi \al}L_{12}(\Om(t))(0)= \f{1-\al}{\al} c_{\om}(t),
\eeq
where $L_{12}(\cdot)$ is defined below in \eqref{eq:L12}. We use $\Om, \eta, \xi$ to denote the perturbation since we will mainly focus on the analysis of the perturbation in the rest of the paper. The reader should not confuse them with the solution to \eqref{eq:bousdy2}.

\subsubsection{ Linearization }


We introduce 
\beq\label{eq:L12}
\bal
L_{12}(\Om)  & \teq  \int_R^{\infty} \int_0^{\pi/2} \f{ \sin (2\b) \Om(s, \b) }{s} ds d \b,
\quad 
\td{L}_{12}(\Om)(R) \teq L_{12}(\Om)(R) - L_{12}(\Om)(0) , \\
\Psi_*  & \teq  \Psi - \f{ \sin( 2\b)}{ \pi \al} L_{12}(\Om). 
\eal
\eeq
For sufficiently small $\al$, the operator $L_{12}$ captures the leading order term of the Biot-Savart law $\uu = \na^{\perp}(-\D)\psi$ , and $\Psi_*$ is the lower order part in the modified stream function $\Psi$. In particular, the leading order parts of the velocity $u, \bar u$ are given by 
\beq\label{eq:vel_main}
\bal
u&  = -\f{2 r \cos \b }{\pi \al} L_{12}(\Om) + l.o.t., \quad  v  = \f{2r \sin \b}{\pi \al} L_{12}(\Om) + l.o.t., \\
  u_x  & = -v_y   = -\f{2}{\pi \al} L_{12} (\Om) + l.o.t., \quad u_{y}  =  l.o.t. , \quad v_x =l.o.t. .  \\
  &L_{12}(\bar{\Om}) =  \f{\pi}{2}  \f{3 \al }{1 + R} ,  \quad  \bar{\Psi}  =  \f{\sin(2\b)}{2} \f{3}{1 +R} + l.o.t. , \\
&  \bar{u}_x  = -\bar{v}_y = - \f{3}{1+R} + l.o.t. , \quad \bar{u}_y , \  \bar{v}_x = l.o.t. .
\eal
  \eeq
  where  \textit{l.o.t.} denotes the lower order terms and we have used the formula of $\bar \Om$ in \eqref{eq:profile} to derive $L_{12}(\bar \Om)$. The smallness of the lower order terms can be justified using the elliptic estimates in Propositions \ref{prop:key}, \ref{prop:psi}. The formulas of $u, v$ in $(R, \b)$ coordinate are given in \eqref{eq:vel_full}. 
  We refer the complete calculation and the formulas of the lower order terms to Section 8.1 in \cite{chen2019finite2}.

\begin{definition}\label{def:op} We define the differential operators 
\[
D_R = R \pa_R , \quad D_{\b} = \sin(2\b) \pa_{\b}
\]
and the linear operators $\cL_i$
\beq\label{eq:op1}
\bal
\cL_1(\Om, \eta) &\teq - D_R \Om - \f{3}{1 + R} D_{\b} \Om -   \Om +  \eta + c_{\om} ( \bar{\Om} - D_R \bar{\Om}) ,\\
\cL_2(\Om, \eta ) & \teq  - D_R \eta - \f{3}{1 + R} D_{\b} \eta + (-2  + \f{3}{1+R} ) \eta + \f{2}{\pi \al } \td{L}_{12}(\Om)  \bar{\eta}+ c_{\om} ( \bar{\eta} - D_R \bar{\eta})  ,\\
\cL_3(\Om, \xi) & \teq  - D_R \xi - \f{3  }{1 + R} D_{\b} \xi+ (-2  - \f{3}{1+R}) \xi 
- \f{2}{\pi \al} \td{L}_{12}(\Om) \bar{\xi} +c_{\om} ( 3\bar{\xi} - D_R \bar{\xi}  ),
 \eal
\eeq
where $\td{L}_{12}(\Om) $ is defined in \eqref{eq:L12} and $\bar{\Om}, \bar{\eta}$ are defined in \eqref{eq:profile}. 
\end{definition}

With the above notations, in the $(R, \b)$ coordinate, the linearized equations of \eqref{eq:bousdy2} around the approximate steady state $(\bar{\Om}, \bar{\eta}, \bar{\xi}, \bar{c}_l, \bar{c}_{\om})$ read
\beq\label{eq:equiv}
\Om_t  = \cL_1(\Om, \eta) + \cR_{\Om} , \quad  \eta_t = \cL_2(\Om, \eta) + \cR_{\eta} , \quad 
\xi_t = \cL_3(\xi) + \cR_{\xi},
\eeq
where $\Om, \eta, \xi$ are the perturbations. The above definitions of the linearized operators are motivated by the leading order structures of the velocity \eqref{eq:vel_main}, and we only keep the leading order terms in $\cL_i$. The remaining terms $\cR$ contain the lower order terms with a small factor $\al$, the residual error of the approximate steady state, and the nonlinear terms. They are treated as the lower order terms in the energy estimates. The full expansion of $\cR$ is rather lengthy, and we refer the formulas and derivations to Sections 5.1 and 8.1 in \cite{chen2019finite2}.


\subsubsection{ Weights and energy norms}

Recall the following singular weights from \cite{chen2019finite2}.
\begin{definition}\label{def:wg}
Recall $ \G(\b) = \cos^{\al}(\b)$.  Let $\s = \f{99}{100},  \g = 1 + \f{\al}{10}$.
Define $\vp_i, \psi_i, \phi_i$ by 
\beq\label{wg}
\bal
 \psi_0 & \teq  \f{3}{16}\lt( \f{(1+R)^3}{R^4} +\f{3}{2} \f{(1+R)^4}{R^3}  \rt) \G(\b)^{-1} , 
\quad 
 \vp_0  \teq \f{(1+R)^3}{R^3} \sin(2\b),  \\
\vp_1 & \teq \f{(1+R)^4}{R^4}  \sin(2\b)^{ - \s}, \quad 
 \vp_2 \teq \f{(1+R)^4}{R^4}  \sin(2\b)^{ - \g} , \\
\psi_1  &  \teq \f{(1+R)^4}{R^4} (  \sin(\b)\cos(\b) )^{-\s},  \quad
 \psi_2    \teq \f{(1+R)^4}{R^4}  \sin(\b)^{-\s}  \cos(\b)^{-\g} , \\
\phi_1  & \teq \f{1+R}{R}, \quad  \phi_2 \teq 1 + (R \sin(2\b)^{\al})^{- \f{1}{40}} , 
\quad \phi_{ij} = \one_{i \geq 1} \phi_1 + \one_{j\geq 1} \phi_2 . \\
\eal
\eeq
\end{definition}

The special forms of $\psi_0, \vp_0$ are designed carefully to exploit nonlocal cancellations in the linearized equations \eqref{eq:equiv} and are crucial for the linear stability analysis of the weighted $L^2$ part of the energy in \eqref{energy:E1}. The weights $\vp_i, \psi_i$, which are singular near $R=0$, are important to derive the damping effect from the linearized equations \eqref{eq:equiv}. The singular weights $\vp_1, \vp_2$ were first introduced in \cite{elgindi2019finite}.
We define the weighted $H^k$ norms as follows 
\beq\label{norm:Hk}
|| f ||_{\cH^m(\rho)}  \teq \sum_{ 0\leq k \leq m}   ||  \rho_1^{1/2} D^k_R f  ||_{L^2} 
+ \sum_{   \ i+ j \leq m - 1
} ||  \rho_2^{1/2} D^{i}_R D^{j+1}_{\b}  f ||_{L^2} .
\eeq
Choosing $\rho_i = \vp_i$ and $\rho_i = \psi_i, i= 1,2$, we get the $\cH^m(\vp)$ and $\cH^m(\psi)$
norm, respectively. We simplify $\cH^m(\vp)$ as $\cH^m$. The $\cH^m$ norm is used for $\Om, \eta$ and the $\cH^k(\psi)$ norm for $\xi$.

We need the weighted $C^k$ norm to control $\xi$ 
\beq\label{norm:ck}
\bal
||  f ||_{\cC^k} &\teq || f ||_{\infty} + \sum_{1 \leq i \leq k}
( || \phi_1 D_R^i f||_{\inf} + || \phi_2 D_{\b}^i f ||_{\inf} )
+ \sum_{ i, j \geq 1, i+ j \leq k} || ( \phi_1 + \phi_2) D_R^i D_{\b}^j f ||_{\inf}.
\eal
\eeq
We remark that the second weights $\vp_2, \psi_2, \phi_2$ are used to handle the angular derivatives. For mixed derivatives only involving $D_R$, we use the first weights $\vp_1, \psi_1, \phi_1$.

To estimate the velocity of the approximate steady state, we use the  $\cW^{l, \inf}$ norm \cite{chen2019finite2,elgindi2019finite} 
\beq\label{norm:W}
|| f ||_{\cW^{l,\infty}} \teq \sum_{0 \leq k+ j \leq l ,  j \neq 0} \B| \B| 
\sin(2\b)^{- \f{\al}{5}} \f{ D_R^k  D_{\b}^j  }{\f{\al}{10} + \sin(2\b)}   f   \B|\B|_{L^{\infty}}  + \sum_{0 \leq k  \leq l } \B| \B| D_R^k f \B|\B|_{L^{\infty}}.
\eeq

\subsection{ \texorpdfstring{$\cH^k$}{Lg} estimates}


The nonlinear stability analysis in \cite{chen2019finite2} is based on the weighted $L^2$ and $H^1$ linear stability analysis of the following energy with a lower order remaining term $\cR_1$ 
\beq\label{energy:E1}
\bal
E_1^2(\Om, \eta, \xi)
\teq   &
 || \Om \vp_0^{1/2}||_2^2 + || \eta \psi_0^{1/2}||_2^2  + \f{81}{4\pi c} L^2_{12}(\Om)(0 )
 + \mu_1 (  || \Om \vp_1^{1/2}||_2^2  + || \eta \vp_1^{1/2}||_2^2 )  \\
 & +  || \xi \psi_1^{1/2} ||_2^2 + \mu_2(  || \Om \vp_1^{1/2}||_2^2  + || \eta \vp_1^{1/2}||_2^2   )  
 + ||  D_{\b}\xi  \psi_2^{1/2}||_2^2   \\
& +  \mu_3 ( || D_R \Om  \vp_1^{1/2}||_2^2 + || D_R \eta \vp_1^{1/2}||_2   
			+ || D_R\xi  \psi_1^{1/2}||_2^2  ) , \\
\eal
\eeq
\[
\bal
 \cR_1(\Om, \eta, \xi)   \teq & 
\la \cR_{\Om}, \Om \vp_0 \ra + \la \cR_{\eta} ,\eta \psi_0 \ra +  \f{81}{4\pi c}  L_{12}(\Om)(0) \la \cR_{\Om}, \sin(2\b) R^{-1} \ra \\
  &+  \mu_1 ( \la D_{\b} \cR_{\Om}, D_{\b}\Om\vp_2 \ra +  \la D_{\b} \cR_{\eta}, D_{\b}\eta \vp_2 \ra  ) +  \la \cR_{\xi}, \xi   \psi_1 \ra  \\
&+ \mu_2 (  \la \cR_{\Om}, \Om \vp_1 \ra + \la \cR_{\eta}, \eta \vp_1 \ra )
+  \la D_{\b} \cR_{\xi}, ( D_{\b}  \xi) \psi_2 \ra  \\
& +  \mu_3 \B( \la D_R \cR_{\Om}, D_R \Om \vp_1 \ra 
+ \la D_R \cR_{\eta}, D_R \eta \vp_1 \ra + \la D_R \cR_{\xi}, D_R \xi \psi_1 \ra  \B), 
\eal
\]
where the absolute constant $\mu_1, \mu_2, \mu_3 > 0$ are determined in order. Higher order stability analysis is further established inductively for the energy $E_k$ with a remaining term $\cR_k$
\beq\label{energy:Ek}
\bal
E_k^2(\Om, \eta, \xi) & \teq E_1^2 +  \sum_{2 \leq i \leq k} \sum_{0 \leq j \leq i} 
 \mu_{i, j}  \lt( 
 ||  p_j^{1/2}  D_R^{j} D^{i - j}_{\b} \Om ||_2^2  
 +  ||  p_j^{1/2} D_R^j D^{i - j}_{\b} \eta ||_2^2  
 +  || q_j^{1/2}  D_R^j D^{i - j}_{\b} \xi  ||_2^2  
  \rt),  \\
  \cR_k(\Om, \eta,\xi) 
 &\teq \cR_1 + 
 \sum_{ 2 \leq i \leq k} 
\sum_{0\leq j \leq i}
   \mu_{i, j}  \lt( 
\la  D_R^{j} D^{i-j}_{\b} \cR_{\Om},  ( D_R^{ j } D^{ i - j }_{\b} \Om ) p_j \ra   
 +   \la  D_R^j D^{ i - j}_{\b}  \cR_{\eta} , ( D_R^j D^{i - j}_{\b} \eta ) p_j \ra  \rt. \\
 & \lt. +  \la   D_R^j D^{ i - j}_{\b} \cR_{\xi}, ( D_R^j D^{i - j}_{\b} \xi ) q_j \ra  \rt), \\
\eal
\eeq
where the weights $(p_j, q_j) = (\vp_1, \psi_1)$ for $j=0$  and  $(p_j, q_j) = (\vp_2, \psi_2)$ for $j \geq 1$. The absolute constants $\mu_{j, i-j}$ can be determined in the order $(2,0), (1,1), (0, 2), (3,0), (2, 1) ...$.  We apply the first weights $\vp_1, \psi_1$ if the mixed derivatives only contain $D_R$, and $(\vp_2, \psi_2)$ otherwise.


The case $k=1, 2, 3$ of the following estimate has been established in Corollary 6.4 \cite{chen2019finite2}, and its generalization to general $k\geq 3$ is straightforward. 
\begin{prop}\label{prop:Hk}
For any $k \geq 1$, there exists some absolute constants $C_k$ and $\mu_{i, j} > 0 , j \leq i$, which can be determined inductively in the order $(i_1, j_1) \preceq (i_2, j_2)$ if $i_1 < i_2$ or $i_1 = i_2$ and $j_1 \leq j_2$, such that 
\[
\f{1}{2}\f{d}{dt} E_k^2(\Om, \eta,\xi) 
\leq  (- \f{1}{15} + C_k \al ) E_k^2  + \cR_k.
\] 
\end{prop}
We refer the details to Sections 5, 6 in \cite{chen2019finite2} and its arXiv version \cite{chen2019finite2arXiv}. Note that the remaining term $\cR_k$ is a lower order term and can be treated as a small perturbation in the final energy estimates.

We define the inner products on $\cH^k$ and $\cH^k(\psi)$ associated with energy $E_k$ \eqref{energy:E1}, \eqref{energy:Ek}, which are equivalent to $\cH^k, \cH^k(\psi)$
\beq\label{eq:inner}
\bal
\la f, g \ra_{\cH^k}  \teq & \mu_1 \la D_{\b} f , D_{\b} g \vp_2 \ra
+ \mu_2 \la  f ,   g \vp_1 \ra  + \mu_3 \la D_R f , D_R g \vp_1 \ra \\
& + \sum_{2 \leq i \leq k} \sum_{0 \leq j \leq i} 
 \mu_{i, j} \la D_R^j D_{\b}^{i-j} f, D_R^j D_{\b}^{i-j} g p_j \ra , \\
\la f, g \ra_{\cH^k(\psi)}  \teq &   \la D_{\b} f , D_{\b} g \psi_2 \ra
+  \la  f ,   g \psi_1 \ra  + \mu_3 \la D_R f , D_R g \psi_1 \ra  \\
&+ \sum_{2 \leq i \leq k} \sum_{0 \leq j \leq i} 
 \mu_{i, j} \la D_R^j D_{\b}^{i-j} f, D_R^j D_{\b}^{i-j} g q_j \ra , \\
\eal
\eeq
where $ p_j = \vp_1, q_j = \psi_1$ for $j =0$ and $p_j = \vp_2, q_j = \psi_2$ for $j \neq 0$. 
We will use the notations \eqref{eq:inner} in the estimates of the transport term in Proposition \ref{prop:tran1} \;.


To estimate the remaining terms $\cR_k$ \eqref{energy:Ek} in the weighted $H^k$ energy estimates in Proposition \ref{prop:Hk} and perform the weighted $C^k$ estimate in Section \ref{sec:Ck}, we need to generalize the functional inequalities in Section 7.2 in \cite{chen2019finite2} to the higher order. We generalize them one by one. Since most estimates follow from the argument in \cite{chen2019finite2}, for completeness, we present the generalizations and give a sketch of the proof in Appendix \ref{app:high} \;. We further generalize the estimates of the approximate steady state $\bar \Om, \bar \eta, \bar \xi$ from \cite{chen2019finite2} in Appendix \ref{app:profile}\;. The proof of these estimates follows the argument in \cite{chen2019finite2} and thus is omitted. 

The estimates of the remaining term $\cR_3$ in Proposition \ref{prop:Hk} with $k=3$ established in \cite{chen2019finite2} are based on the elliptic estimates in Proposition \ref{prop:key}, the functional inequalities in Appendix \ref{app:high}, the estimates of the approximate steady state in Appendix \ref{app:profile} and the $L_{12}(\cdot)$ operator in Lemma \ref{lem:l12}. 
With the higher order analogs of these estimates, we can follow the derivations and arguments in Section 8 in \cite{chen2019finite2} to establish the following estimate
\[
|\cR_k|
\les \al^{1/2} ( E_k^2 + \al || \xi||^2_{\cC^{k-2}}) + \al^{-3/2} ( E_k + \al^{1/2}||\xi||_{\cC^{k-2}})^3 + \al^2 E_k .
\]
It generalizes the corresponding estimate in the case of $k=3$ in Section 8.4.1 in \cite{chen2019finite2}. Plugging the above estimate in Proposition \ref{prop:Hk}, we complete the $\cH^k$ estimate 
\beq\label{eq:boot1}
\f{1}{2}\f{d}{dt}  E_k^2  \leq - \f{1}{15}    E_k^2  + C_k \al^{1/2} ( E_k^2 + \al || \xi||^2_{\cC^{k-2}}) + C_k \al^{-3/2} ( E_k + \al^{1/2}||\xi||_{\cC^{k-2}})^3 + C_k \al^2 E_k ,
\eeq
which generalizes estimate (8.1) in \cite{chen2019finite2} from $k=3$ to $k \geq 3$.

\subsection{ \texorpdfstring{$\cC^k$}{Lg} estimate of \texorpdfstring{$\xi$}{Lg} }\label{sec:Ck}

To close the $k-th$ order nonlinear estimates, we need to control the $L^{\inf}$ norm of $D_R^i D_{\b}^j \Om,D_R^i D_{\b}^j \eta, D_R^i D_{\b}^j \xi  , i+ j \leq k-2$. 
For $\xi$, however, since the weight $\psi_2$ (see Definition \ref{def:wg}) is less singular in $\b$ for $\b$ close to $0$, the weighted $\cH^k(\psi)$ space associated to $\xi$ is not embedded continuously into $\cC^{k-2}$. Alternatively, we estimate the $\cC^k$ norm of $\xi$. The case $k=1$ has been established in Section 6.3 and 8.5 in \cite{chen2019finite2}. The general case of $k \geq 1$ is not so straightforward and thus we present the estimate. Throughout the rest of the paper, we only consider $k \leq 100$ to avoid tracking absolute constants related to $k$.

Following the derivation in Section 6.3 in \cite{chen2019finite2} (or Section 6.4 in the arXiv version \cite{chen2019finite2arXiv}),  we can rewrite the linearized equation of $\xi$ in \eqref{eq:equiv} as follows 
\beq\label{eq:xi_linf1}
\pa_{t} \xi + \cA(\xi) = (-2 - \f{3}{1+R}) \xi + \Xi_1 + \Xi_2 +  \bar{F}_{\xi} + N_o,
\eeq
where $\cA(\xi)$ denotes the transport term in the $\xi$ equation in \eqref{eq:bousdy2}, including the nonlinear part 
\beq\label{eq:cA0}
\cA (\xi )\teq (1+ 3\al + \al c_l ) D_R \xi +  ( (\bar{\uu} + \uu ) \cdot \na) \xi ,
\eeq
$\Xi_1, \Xi_2$ are lower order terms that contain a small factor 
\beq\label{eq:Xi}
\Xi_1  \teq (\f{3}{1+R} - \bar{v}_y) \xi, \quad  \Xi_2 =  - v_y \bar{\xi} + c_{\om} ( 2 \bar{\xi} - R\pa_R \bar{\xi})
  +( \al c_{\om} R\pa_R - ( \uu \cdot  \na ) ) \bar{\xi} -  (u_y \bar{\eta} +\bar{u}_y \eta ),
\eeq
$N_o$ is some nonlinear term, and $\bar F_{\xi}$ is the error 
\beq\label{eq:non_xio}
 N_o = (2c_{\om} - v_y)\xi   -  u_y \eta , 
 \quad 
 \bar{F}_{\xi}  \teq -  (2 \bar{c}_{\om} -  \bar{v}_y) \bar{\xi} + \bar{u}_y \bar{\eta} +\al \bar{c}_l R \pa_R\bar{\xi} + (\bar{\uu} \cdot \na)  \bar{\xi} .\\
\eeq
Here, we have expanded the remaining term $\cR_{\xi}$ in \eqref{eq:equiv}. The above decomposition and $\Xi_1, \Xi_2$ are motivated by the leading order structure of the velocity \eqref{eq:vel_main} and the properties that $\bar \xi$ is a lower order term of size $\al^2$. 

We introduce $\cT$ \cite{chen2019finite2} to denote the lower order part of the transport operator $\uu \cdot \na$
\beq\label{eq:trans}
\cT(\Om) \teq -\f{ 2\cos(2\b)}{\pi } L_{12}(\Om)  D_R 
- \al \pa_{\b} \Psi_* D_R + \f{  2\Psi_* + \al D_R \Psi}{\sin(2\b)}  D_{\b} .
\eeq

Using this notation, we decompose $  \bar \uu \cdot \na$ into the main term and the lower order term $\cT(\bar \Om)$
\[
\bar \uu \cdot \na = \f{3}{1+ R} D_{\b} + \cT(\bar \Om) ,
\]
which further motivates the following decomposition of \eqref{eq:cA0}
\beq\label{eq:cA}
\cA(\xi) =( (1+ 3\al + \al c_l) D_R \xi +  \f{3 }{ 1+R} D_{\b} \xi )
+ (  ( (\uu + \bar{\uu}) \cdot \na - \f{3}{1+R} D_{\b} )  \xi  ) \teq \cA_1(\xi) + \cA_2(\xi).
\eeq
The main part in $\cA(\xi)$ is captured by $\cA_1(\xi)$.  We refer the detailed derivation to Section 6.4 and Section 8.1 in \cite{chen2019finite2}. 

We need the following simple estimates for the weights.
\begin{lem}\label{lem:dwg}
For $\phi = \phi_1, \phi_2, \phi_{ij}$ defined in \eqref{def:wg}, we have $|D_{R} \phi| \les \phi, |D_{\b} \phi| \les \al \phi, 1 \les \phi_1, 1 \les \phi_2$.
\end{lem}

The estimate for $\phi_1, \phi_2$ follows from a simple calculation. The estimate for $\phi_{ij}$ follows from a triangle inequality.

\subsubsection{ Estimate of $\phi_{ij} D_R^i  D^j_{\b} \xi$}
To establish the $\cC^k$ estimate, we perform $L^{\inf}$ estimate on $\phi_{ij} D_R^i  D^j_{\b} \xi $. 
Denote the commutator and the weighted derivatives 
\beq\label{eq:xi_ij}
[P, Q ] = P Q -  QP, \quad \cP_{ij} \teq \phi_{ij} D_R^i  D^j_{\b} . 
\eeq

In the $L^{\inf}$ estimate, we need to estimate the commutator $[ \cP_{ij}, \cA]$ \eqref{eq:cA}, which is given below.
\begin{lem}\label{lem:Linf_commut}
For $ i+j = k \leq 100$, we have 
\beq\label{eq:commut1}
\bal
  {[}  \cP_{ij},  \cA_1 ] 
 \xi = -(1 + 3 \al + \al c_l)  \f{ D_{R} \phi_{ij} }{ \phi_{ij}} \cP_{ij} \xi 
 + \e_{ij} 
\eal
\eeq
with error terms $\e_{ij}$ satisfying 
\beq\label{eq:commut1_err}
\e_{0 k} \les \al | \cP_{ij} \xi|,  \quad  
|\e_{ij} | \les | \cP_{i-1, j+1} \xi |
+ || \xi||_{\cC^{k-1}} + \al | \cP_{ij} \xi|, \quad  i \geq 1 \;.
\eeq
For $\cA_2$, we have
\beq\label{eq:commut2}
| [\cP_{ij}, \cA_2]  \xi | \les 
 (\al^{-1} || \Om||_{ \cH^{k+2}} + \al) || \xi ||_{\cC^k}. 
\eeq
\end{lem}

The first term on the right hand side of \eqref{eq:commut1} is a damping term so that we do not need to control it. The upper bound of \eqref{eq:commut2} is of order $\al || \xi||_{\cC^k}$, and the second commutator is a lower order term. The above estimate shows that $\cP_{0k}$ almost commutes with $\cA$ up to some lower order terms, which are small or can be controlled by lower order energy. 
Thus we will first perform $L^{\inf}$ estimate on $\cP_{0k} \xi$. Then we estimate other terms $\cP_{i, k-i} \xi$ in the order of $i=1,2,3.., k$ so that in each step, we can control the bad term 
$|\cP_{i-1, j+1}\xi|$ in \eqref{eq:commut1_err} using the previous estimate.

\begin{proof}
To simplify the notation, we denote $\phi = \phi_{ij}$. We decompose  $\cA_1$ \eqref{eq:cA} 
into 
\[
\cA_1 = \cA_{11} + \cA_{12}, \ \cA_{11} =a_{\al} D_R,  \ \cA_{12} = f(R) D_{\b}, \ a_{\al} =  (1 + 3 \al + \al c_l ), \ f(R) = \f{3}{1 + R}. 
\]
Note that $D_{\R}$ and $D_{\b}$ commute. For $\cA_{11}$, a direct computation yields 
\[
\bal
{[} \cP_{ij} , \cA_{11}] \xi 
&= \phi D_R^i D_{\b}^j  \cA_{11} \xi - \cA_{11}( \phi D_R^i D_{\b}^j \xi) 
= a_{\al} ( \phi D_R^{i+1} D_{\b}^j \xi  - D_R( \phi D_R^i D_{\b}^j \xi) ) \\ 
&= - a_{\al}  D_R \phi  D_R^i D_{\b}^j \xi = - a_{\al}  \phi^{-1} D_R{\phi}  \cP_{ij} \xi,
\eal
\]
which gives the first term on the right hand side of \eqref{eq:commut1}. For $\cA_{12}$, similarly, we get 
\[
\bal
{[} \cP_{ij} , \cA_{12}] \xi  
& = \phi D_R^i D_{\b}^j  ( f(R) D_{\b} \xi )
- f(R) D_{\b} ( \phi D_R^i D_{\b}^j \xi) \\
&= \phi \B( D_R^i D_{\b}^j ( f(R) D_{\b} \xi ) - f(R) D_R^i D_{\b}^{j+1} \xi  \B)
- f(R) D_{\b} \phi D_R^i D_{\b}^j \xi \teq I + II .
\eal
\]
Using Lemma \ref{lem:dwg} and $f(R) \les 1$, we yield 
\[
|II| \les \al f(R)  |\phi D_R^i D_{\b}^j \xi|  
\les \al |\cP_{ij} \xi|.
\]

For $I$, since $D_{\b} f(R) = 0$, if $i = 0$, we get $I = 0$. This proves the first inequality in \eqref{eq:commut1_err}. 

If $ i \geq 1$, using the Leibniz rule, we derive 
\beq\label{eq:commut1_est1}
\bal
|I| & =| \phi ( D_R^i ( f(R) D_{\b}^{j +1} \xi ) - f(R) D_R^i D_{\b}^{j+1} \xi)|  \\
&\les \phi D_R  f(R) D_R^{i-1} D_{\b}^{j+1} \xi  
+ \phi \sum_{ l \leq i - 2 } | D_R^{ i -l} f(R) D_R^l D_{\b}^{j+1} \xi|. 
\eal
\eeq

From the definition of $\phi_1, \phi_2, \phi_{ij}$ \eqref{def:wg}, we have 
\[
D_R^l f(R) \les \f{R}{ (1+R)^2} , \  1 \leq l \leq 100, \quad \phi_1 \f{R}{ (1+R)^2} \les 1,  \quad
\phi_{ij} \f{R}{ (1+R)^2}
\les 1 + \one_{j \geq 1} \phi_2
\les \phi_{n, j+1}
\]
for any $ n \geq 0$. Using the above derivations and the definition of $\cC^k$ \eqref{norm:ck}, we yield 
\[
|I| \les |\phi_{i-1, j+1}  D_R^{i-1} D_{\b}^{j+1} \xi|
+ \sum_{ l \leq i-2}  \phi_{ l, j+1}  |D_R^l D_{\b}^{j+1} \xi |
\les |\cP_{i-1, j+1} \xi| + || \xi||_{C^{i+j-1 }}.
\]
This completes the proof of \eqref{eq:commut1} and \eqref{eq:commut1_err}.

For $\cA_2$, we follow the decomposition in Section 8.5.3 in \cite{chen2019finite2}
\[
\bal
\cA_2 (\xi)  = & \f{2}{\pi \al} L_{12}(\Om) D_{\b} \xi + (  \cT(\bar{\Om}) +\cT( \Om) )\xi
= \f{2}{\pi \al} L_{12}(\Om) D_{\b} \xi
 - \f{2}{\pi} \cos(2\b) \B(  L_{12}(\Om) +L_{12}(\bar{\Om}) \B) D_R \xi\\
 &- \al (\pa_{\b} \Psi_* + \pa_{\b}\bar{\Psi}_*)D_R \xi
 +  \f{  2\Psi_* + \al D_R \Psi + 2 \bar{\Psi}_* + \al D_R \bar{\Psi}  }{\sin(2\b)}   D_{\b} \xi \\
   \teq & ( H_1 D_{\b}  + H_2 D_R + H_3 D_{R}+ H_4 D_{\b})\xi,
\eal
\]
where $L_{12}(\Om), \Psi_*$ defined in \eqref{eq:L12} relate to the leading order terms of the velocity,  and $\cT( \cdot)$ \eqref{eq:trans} denotes the lower order part in the transport operator $\uu \cdot \na$. 

Let $(H, D)$ be one of the pairs $(H_1, D_{\b}), (H_2, D_R), (H_3, D_R), (H_4, D_{\b})$ in the above decomposition. Next, we show that
\beq\label{eq:commute2_est1}
| {[} \cP_{ij}, H D ] \xi | \les || H||_{C^k} || \xi||_{C^k}.
\eeq

Since $D_R$ and $D_{\b}$ commute, applying the Leibniz rule, we derive 
\[
\bal
 {[} \cP_{ij}, H D ] \xi
 &= \phi D_R^i D_{\b}^j ( H D \xi) - H D( \phi D_R^i D_{\b}^j  ) \xi 
 = \phi D_R^i D_{\b}^j ( H D \xi) - H \phi  D_R^i D_{\b}^j  D \xi - H D\phi \cdot D_R^i D_{\b}^j \xi \\
&=\phi \sum_{ l \leq i, m \leq j, l + m< i+j} \binom{ i}{l} \binom{ j}{m} 
 (D_R^{i- l } D_{\b}^{ j-m} H ) ( D_R^l D_{\b}^m D \xi) 
 - H D \phi \cdot D_R^i D_{\b}^j \xi  \teq J_1 + J_2.
 \eal
\]

Using Lemma \ref{lem:dwg}, we obtain 
\[
|J_2 | \les || H||_{\inf}  | \phi D_R^i D_{\b}^j \xi | \les 
 || H||_{\inf}  | \cP_{ij} \xi| 
 \les  || H||_{\inf}  || \xi||_{\cC^k} . 
\]

Recall $\phi = \one_{i\geq 1} \phi_1 + \one_{ j \geq 1} \phi_2$. We estimate two weights in $J_1$ separately. If $i=0$, the estimate is trivial since $\one_{i\geq 1} \phi_1 = 0$. If $i \geq 1$, we have $i-l \geq 1$ or $l \geq 1$. In the first case, using the definition of $\cC^k$ \eqref{norm:ck}, Proposition \ref{prop:embed}, and $i + j - l - m  \leq i+ j = k, l+m \leq i+ j =k$, we get 
\[
\bal
|\phi_1 D_R^{i-l} D_{\b}^{j-m} H \cdot D_R^l D_{\b}^m \xi |
&\les | \phi_1 D_R^{i-l} D_{\b}^{j-m} H |_{\inf}  |D_R^l D_{\b}^m \xi |_{\inf} \\
&\les || D_R^{i-l-1} D_{\b}^{j-m} H ||_{\cC^1}  || \xi ||_{\cC^k}
\les || H||_{\cC^k} || \xi||_{\cC^k}.
\eal
\]

The estimates of the other case $l \geq 1$ and that of the second weight are completely similar. Thus, using the triangle inequality, we establish
\[
|J_1| \les || H||_{\cC^k} || \xi ||_{\cC^k},
\]
which along with the estimate of $J_2$ implies \eqref{eq:commute2_est1}.

To further control the $|| H||_{\cC^k}$, we have the following estimates 
\beq\label{eq:commut2_est2}
|| H||_{\cC^k} \les \al^{-1} || \Om||_{\cH^{k+2}} + \al.
\eeq
The case of $k =1$ has been established in Section 8.5.3 in \cite{chen2019finite2}. The general case $k \geq 1$ follows from a similar argument by using the elliptic estimates in Propositions \ref{prop:key}, \ref{prop:psi}, and the embedding in Proposition \ref{prop:embed}. We refer to \cite{chen2019finite2} for more discussions. 

Combining \eqref{eq:commute2_est1} and \eqref{eq:commut2_est2}, we conclude the proof of \eqref{eq:commut2}.
\end{proof}


Now, we are in a position to estimate $\cP_{ij} \xi$. Applying $\cP_{ij}$ on both sides of \eqref{eq:xi_linf1}, we yield 
\beq\label{eq:xi_linf2}
 \pa_{t} \cP_{ij} \xi  +  \cA   \cP_{ij} \xi 
=  \cP_{ij}  ( ( -2 - \f{3}{1 + R} ) \xi ) -  [ \cP_{ij}, \cA ] \xi 
+ \cP_{ij} ( \Xi_1 +  \Xi_2 +  \bar F_{\xi} +  N_o ).
\eeq
The main damping term is $ (-2 -  \f{3}{1+R} )\xi$.

For the first term on the right hand side, applying estimate similar to \eqref{eq:commut1_est1}, we have 
\[
 - \cP_{ij}  ( ( 2 + \f{3}{1 + R} ) \xi ) 
 = - \phi D_R^i D_{\b}^j ( ( 2 + \f{3}{1 + R} ) \xi ) 
 =  - (2 + \f{3}{1 + R}) \cP_{ij} \xi  + \d_{ij} ,
\]
with $\d_{0 k} = 0$ and 
\[
\bal
|\d_{i,j}| & = \B|  \phi \sum_{l \leq  i-1 } \binom{i}{l}  D_R^{ i - l} (2 + \f{3}{1+R}) D_R^l D_{\b}^j \xi \B|
\les \sum_{l \leq i-1} \phi \f{R}{ (1+R)^2} | D_R^l D_{\b}^j \xi |
\les \sum_{l \leq i-1} ( 1 + \phi_{l j} ) | D_R^l D_{\b}^j \xi | \\
& \les || \xi||_{C^{k-1}}
\eal
\] 
for $i \geq 1$.  From the definition of $\phi_1, \phi_2, \phi = \phi_{ij}$ \eqref{def:wg}, we have $D_R \phi_{ij} \leq 0 $. Using the above estimate and Lemma \ref{lem:Linf_commut}, we derive 
\[
\bal
 \cP_{ij} \xi & \cdot 
\B(  \cP_{ij}  ( ( -2 - \f{3}{1 + R} ) \xi ) -  [ \cP_{ij}, \cA ] \xi \B)
 \leq -2 | \cP_{ij} \xi|^2
+ |\cP_{ij} \xi| ( | \e_{ij} | + | \d_{ij} | )  \\
& \leq -2 | \cP_{ij} \xi|^2 + C  | \cP_{ij} \xi | \B( || \xi||_{\cC^{k -1 }} + \al |\cP_{ij} \xi|  + \one_{i\geq1} | \cP_{i-1, j+1} \xi| \B) .
\eal
\]
To derive the first inequality,  we have dropped the first term on the right hand side of \eqref{eq:commut1} since $D_R \phi_{ij} \leq 0$

 Now, multiplying both sides of \eqref{eq:xi_linf2} by $ \cP_{ij}\xi$ and performing the $L^{\inf}$ estimate, we establish
\[
\bal
\f{1}{2} \f{d}{dt} || \cP_{ij} \xi||^2_{\inf}
 \leq  & - (2 - C \al) || \cP_{ij} \xi||_{\inf}^2 + C  || \cP_{ij} \xi ||_{\inf} \B( || \xi||_{\cC^{k -1 }} 
 + \one_{i\geq1} || \cP_{i-1, j+1} \xi||_{\inf} \B) \\
 & + C || \cP_{ij} \xi ||_{\inf}  ||\cP_{ij} ( \Xi_1 +  \Xi_2 +  \bar F_{\xi} +  N_o )||_{\inf} .
 \eal
\]

Notice that for $i=0$, the leading order term $- (2 - C \al) || \cP_{ij} \xi||_{\inf}^2$ is a damping term. Moreover, if $ i+j = k = 0$, the term $|| \xi||_{\cC^{k-1}} $ on the right hand side is $0$. Thus, there exists absolute constants $\nu_{i ,j}>0$ with $i+j \leq k$, which can be determined inductively on $i + j$, such that for 
\beq\label{energy:Ck}
E_{k, \inf}^2( \xi) \teq || \xi||_{\inf}^2 + \sum_{ i+ j \leq k} \nu_{i j}  || \cP_{ij} \xi||_{\inf}^2,
\eeq
we have 
\[
\f{1}{2} \f{d}{dt} E_{k,\inf}^2 \leq -( \f{3}{2} - C\al ) E_{k,\inf}^2
+ C E_{k,\inf} || \Xi_1 +  \Xi_2 +  \bar F_{\xi} +  N_o ||_{\cC^k}.
\]

To further control $\Xi_1 +  \Xi_2 +  \bar F_{\xi} +  N_o$, we have the following estimate 

\begin{lem}
For $ 1 \leq k \leq 100$, we have
\[
\bal
&|| \Xi_1||_{\cC^k} \les  \al || \xi||_{\cC^k} , \quad 
|| \Xi_2||_{\cC^k} \les \al^{1/2} || \Om||_{\cH^{k+2}} + \al^{1/2} || \eta||_{\cH^{k+2}},  \\
& || N_o ||_{\cC^k} \les  \al^{-1} || \xi||_{\cC^k} || \Om||_{\cH^{k+2}} 
+ \al^{-1} || \Om||_{\cH^{k+2} } || \eta||_{\cH^{k+2} } ,
\quad || \bar F_{\xi} ||_{\cC^k} \les \al^2.
\eal
\]
\end{lem}

The case of $k=1$ has been established in Sections 8.5.1 and 8.5.2 in \cite{chen2019finite2}. The general case follows from a similar argument, Proposition \ref{prop:embed}, and the elliptic estimates 
in Propositions \ref{prop:key} and \ref{prop:psi}. 
Note that the $\cC^k$ norm \eqref{norm:ck} and $E_{k, \inf}^2$ are equivalent, we conclude the $\cC^k$ estimate 
\beq\label{eq:boot2}
\bal
\f{1}{2} \f{d}{dt} E_{k,\inf}^2  \leq & - ( \f{3}{2} - C_k \al ) E_{k,\inf}^2
+ C_k E_{k,\inf}  \B( \al^2 
+  \al^{1/2} || \Om||_{\cH^{k+2}} + \al^{1/2} || \eta||_{\cH^{k+2}} \\
& +  \al^{-1} E_{k,\inf} || \Om||_{\cH^{k+2}} 
+ \al^{-1} || \Om||_{\cH^{k+2} } || \eta||_{\cH^{k+2} } \B).
\eal
\eeq

\subsection{Nonlinear stability and finite time blowup}\label{sec:blowup}


We fix $k = 100$ and construct the energy 
\beq\label{eg}
E( \Om, \eta, \xi) =( E_k(\Om, \eta, \xi)^2 + \al E_{k-2,\infty}(\xi)^2)^{1/2}, \quad k = 100.
\eeq

Adding the estimates  \eqref{eq:boot1} and $ \al \times$\eqref{eq:boot2} with $k$ replaced by $k-2$, we have
\[
\f{1}{2} \f{d}{dt} E^2(\Om, \eta,\xi)
\leq  -\f{1}{15} E^2 + K \al^{1/2} E^2 + K \al^{-3/2} E^3 + K \al^2 E ,
\]
for some absolute constant $K$, where we have used the fact that $E_{k,\inf}(\xi)$ \eqref{energy:Ck} is equivalent to $|| \xi||_{\cC^k}$ since $\nu_{ij}$ are absolute constants. Using a standard bootstrap argument, we establish that there exists a small absolute constant $\al_1 < \f{1}{1000}$ and $K_*$, such that if $E(\Om(\cdot, 0), \eta(\cdot ,0), \xi(\cdot, 0)) < K_* \al^2$, we have 
\beq\label{eq:boot}
E(\Om(t), \eta(t), \xi(t) )  < K_* \al^2
\eeq
for all time $t>0$ and $\alpha < \alpha_1$. We refer the detailed bootstrap argument to \cite{chen2019finite2}.

Finally, we consider the regularity of the solutions $ \om + \bar \om, \eta + \bar \eta, \xi + \bar \xi$ in the physical space using the relations \eqref{eq:rescal1}, \eqref{eq:rescal2}. Following the argument in \cite{chen2019finite2}, we obtain that the scaling parameters $c_l(t), c_{\om}(t)$ defined in \eqref{eq:profile}, \eqref{eq:normal}  satisfy
\[
      -\f{3}{2} < c_{\om}(t)  + \bar c_{\om} < -\f{1}{2},  \quad  \f{1}{2\al} + 3 <  c_l + \bar c_l  < \f{2}{\al} + 3,
\]
where $\bar c_{\om}, \bar c_l$ denote the  scaling parameters associated to the approximate steady state, and $c_{\om}, c_l$ are the perturbations. In particular, $C_{\om}(\tau), C_l(\tau)$ defined in \eqref{eq:rescal2} remains finite for any $\tau < +\inf$ with bounds depending on $\tau, \al$ only. 

Next, we show that $ \f{1}{r}( u + \bar u) \in \cC^{k-2}$. Applying $\bar L_{12}(\Om) = \f{\pi}{2} \f{3\al}{1 + R} $ \eqref{eq:vel_main}, the embedding in Proposition \ref{prop:embed}, and Proposition \ref{prop:psi} to $ U( \bar \Psi)$ \eqref{eq:vel_full}, we yield 
\beq\label{eq:U_Ck}
\bal
|| \f{1}{r} U(\bar \Psi)  ||_{\cC^k} 
&\les || \f{1}{1+R} ||_{\cC^k} 
+  ||   \bar \Psi_* ||_{ \cC^k} + || D_R \bar \Psi||_{\cC^k}
+ || \pa_{\b}  \bar \Psi_* ||_{\cC^k} \\
&\les_{\al} 1 + || \f{1+R}{R}  \bar \Psi_* ||_{ \cW^{k, \inf}} + || \f{1+R}{R} D_R \bar \Psi||_{\cW^{k, \inf}}
+ || \f{1+R}{R} \pa_{\b} \bar \Psi_* ||_{\cW^{k, \inf}}
\les_{\al} 1.
\eal
\eeq

For $U(\Psi)$, we first consider $L_{12}(\Om)$ using Lemma \ref{lem:l12}. Let $\chi$ be the radial cutoff function defined in Lemma \ref{lem:l12}, which is constant near $ r= 0$. Using Proposition \ref{prop:embed} and \eqref{eq:boot}, we have 
\[
|| L_{12}(\Om)||_{\cC^{k-2}}
\les || L_{12}( \Om) - \chi_1 L_{12}(\Om)(0)||_{\cC^{k-2}}
+ || \chi_1||_{\cC^{k-2}} | L_{12}(\Om)(0) |
\les_{\al} || \Om||_{\cH^k}\les_{\al} 1.
\]

Applying Propositions \ref{prop:key} and \ref{prop:embed} to control the $\Psi, \Psi_*$ terms in $U(\Psi)$ \eqref{eq:vel_full} , we get
\[
|| \f{1}{r} U(\Psi)||_{ \cC^{k-2}}
\les || L_{12}(\Om)||_{\cC^{k-2}} + || \Om||_{\cH^k} \les_{\al} 1.
\]

Similarly, using the estimates of the approximate steady state in Lemmas \ref{lem:gam}, \ref{lem:xi}, and Proposition \ref{prop:embed}, we obtain 
\beq\label{eq:V_Ck}
|| \f{1}{r} ( V(\Psi + \bar \Psi) ||_{\cC^{k-2}} \les_{\al} 1, \quad  || \bar \Om+  \Om ||_{\cC^{k-2}} \les_{\al} 1,  \quad 
|| \bar \eta + \eta ||_{\cC^{k-2}} \les_{\al} 1, \quad  || \bar \xi + \xi ||_{\cC^{k-2}} \les_{\al} 1.
\eeq

Since the $(R, \b)$ coordinate of $(C_l x, C_l y)$ is $(C_l^{\al} R, \b) $, using the rescaling relation \eqref{eq:rescal1}, \eqref{eq:nota_om}, in $(R,\b)$ coordinate, we obtain 
\[
\Om(R, \b, \tau) = C_{\om}(\tau) \om_{phy}( C_l^{\al}(\tau) R, \b, t(\tau) ), \quad 
\om_{phy}(R, \b , t(\tau) ) = 
C_{\om}^{-1}\Om( C_l^{-\al} R, \b, \tau) .
\]
Similar relations apply for $\th, \uu$. Applying \eqref{eq:U_Ck}, \eqref{eq:V_Ck},  the above relation, and $\rho( \lam R, \b) \les C(\lam) \rho(R, \b)$ for any weight $\rho$ in Definition \ref{def:wg}, we have
\beq\label{eq:phy_Ck}
\bal
&|| \om_{ phy}(t(\tau)) ||_{\cC^{k-2}}  + 
|| \th_{x,phy}(t(\tau) ) ||_{\cC^{k-2}}  +
|| \th_{y,phy}(t(\tau) ) ||_{\cC^{k-2}}  \\
&+ || \f{1}{r}  u_{phy} ||_{\cC^{k-2}}
+ || \f{1}{r}  v_{phy} ||_{\cC^{k-2}}  
 \les  C( C_l( \tau) , C_{\om}(\tau ), \al, \tau )
\les C(\al, \tau) < +\infty .
\eal
\eeq

To further estimate the $C^k$ regularity, we have the following simple embedding. 
\begin{lem}\label{lem:cCk}
Let $S\teq \{ (x, y) : x \neq 0, y > 0\} = \{ (r, \b) : r > 0, \b \in (0, \pi/2) \cup (\pi/2, \pi)\}.$
For any compact domain $\S  \subset S$ and $l \geq 1$, we have 
\[
|| f||_{C^{l-1}(\S)} \les_{l, \al, \S } || f||_{\cC^{l}}.
\]

\end{lem}

\begin{proof}
Recall $D_R = R \pa_R $ and $R = r^{\al}$. Using the chain rule, we yield $ r \pa_r = \al R \pa_R$. 
For any compact domain $\S \subset  S$, $i \geq 0$ and $p \in \R$, since $r \neq 0, \sin(\b) , |\cos \b| \in (0, 1)$, it is easy to obtain that 
\[
  | \pa_r^i r^{ p }| \les_{ i, p, \S} 1, \quad  | \pa_{\b}^i  \sin^p(\b)| 
  \les_{i, p, \S} 1, \quad | \pa_{\b}^i  | \cos(\b)|^p | \les_{i, p, \S} 1.
\]

Recall the relation among $\pa_x, \pa_y, \pa_r, \pa_{\b}$
\[
\pa_x   = \cos (\b) \pa_{r}  -\f{\sin(\b)}{r} \pa_{\b} ,  \quad \pa_y  =  \sin(\b) \pa_{r}  + \f{\cos (\b)}{ r} \pa_{\b} .
\]

Using the Leibniz rule, induction on $l$ and the above estimate, for $i+j \leq l$ and $(x, y) \in \S$, we obtain 
\[
| \pa_x^i \pa_y^j f | \les_{l, \S}  \sum_{ m + n \leq l} | \pa_r^m \pa_{\b}^n f|
\les_{l, \S} \sum_{ m + n \leq l} | (r \pa_r )^m \pa_{\b}^n f|
\les_{l, \al, \S } \sum_{ m + n \leq l} | D_R^m \pa_{\b}^n f| 
\les_{l, \al, \S } || f||_{\cC^{l}}.
\]
It follows $f \in C^{l-1}( \S )$.
\end{proof}

Since $\uu = r \cdot \f{1}{r } \uu$ and $r \in C^{k-2}( \S)$, from \eqref{eq:phy_Ck} and Lemma \ref{lem:cCk}, we further get $ \uu \in C^{k-3}(\S)$ for any compact set $\S \subset \{ (x, y): x \neq 0,  y>0  \}$. Using the above bootstrap estimates, e.g. \eqref{eq:boot}, the regularity estimates, the arguments of localizing the initial data and that for the finite time blowup in \cite{chen2019finite2}, we prove Theorem \ref{thm:bous_blowup}. We refer these two arguments to \cite{chen2019finite2}.

\section{Properties of the singular solution to the 3D Euler equations}\label{sec:euler_blowup}

In \cite{chen2019finite2}, to generalize the blowup results from the 2D Boussinesq equations in $\R_2^+$ to the 3D axisymmetric Euler equations with boundary, we need two additional steps. The first step is to control the support of the solutions in the domain $(r, z) \in [0, 1] \times \BT$ so that it does not touch the symmetry axis $r = 0$. The second step is to generalize the $\cH^3$ elliptic estimates in the Boussinesq equations to the 3D Euler equations. See Section 1.3 and Section 9 in \cite{chen2019finite2}.

For higher order estimates of the singular solutions to the 3D Euler equations, we only need to generalize the $\cH^3$ elliptic estimates to the $\cH^k$ version since the first step does not involve higher order estimates. Note that the $\cH^3$ elliptic estimates in Proposition 9.9 in \cite{chen2019finite2} is proved inductively with the weighted $L^2( \f{(1+R)^4}{R^4})$ elliptic estimate being the based case. Therefore, its generalization to the $\cH^k$ estimate in Proposition \ref{prop:elli_euler} below is straightforward. 

These higher order estimates imply the interior regularity estimates of $\om^{\vth}, (u^{\vth})^2, u^r, u^z$ in Theorem \ref{thm:euler_blowup}. See Section \ref{sec:euler_traj}. The estimate of $u^{\vth}$ does not follow from that of $ (u^{\vth})^2$. In Sections \ref{sec:euler_uth} and \ref{sec:euler_uth_in}, we further estimate $u^{\vth}$.

The proof of Theorem \ref{thm:euler_R3_blowup} is similar and is mostly based on the estimates in \cite{elgindi2019finite,elgindi2019stability}. Thus, we will only sketch the proof.

\subsection{Setup of the 3D axisymmetric Euler equations}

We first review the basic setup of the 3D axisymmetric Euler equations from Section 9 in \cite{chen2019finite2}. Recall the 3D axisymmetric Euler equations from \eqref{eq:euler1}-\eqref{eq:euler21} and the cylindrical coordinates $(r, \vth, z)$ \eqref{eq:polar_basis} in $\R^3$.
 We introduce the following variables 
\beq\label{eq:omth}
\td{\th}(r, z) \teq (r u^{\vth})^2,  \quad \td{\om}(r, z) = \om^{\vth} / r,
\eeq
new coordinates $(x, y)$ centered at $r = 1, z= 0$, and its related polar coordinates
\beq\label{eq:euler_polar}
x =  C_l(\tau)^{-1} z, \quad  y = (1-r) C_l(\tau)^{-1}, \quad  \rho = \sqrt{x^2 + y^2}, \quad \b = \arctan(y/x) , \quad R = \rho^{\al},
\eeq
where $C_l(\tau)$ is defined below \eqref{eq:rescal42}. 
The reader should not confuse the relation $R = \rho^{\al}$ with $R = r^{\al}$ in the 2D Boussinesq. 
Since the domain $D = \{ (r, z): r \leq 1, z \in \BT \}$ of the equations \eqref{eq:euler1}-\eqref{eq:euler21} is periodic in $z$ with period $2$, we focus on one period 
\beq\label{eq:euler_HL_domain}
D_1 \teq \{ (r, z): r \leq 1, |z| \leq 1 \}.
\eeq
In the proof in \cite{chen2019finite2}, the variables $\td \om,  \td \th$ \eqref{eq:omth} are the analog of $(\om, \th)$ in the 2D Boussinesq equations \eqref{eq:bous}. 
The cylindrical coordinate $(r, z)$ for the 3D Euler equations relate to $(y, x)$ in the 2D Boussinesq equations \eqref{eq:bous} via the change of variables \eqref{eq:euler_polar}.

We consider the following dynamic rescaling formulation centered at $r = 1, z= 0$
\beq\label{eq:rescal41}
\bal
\th(x, y, \tau) &= C_{\th}(\tau) \td{\th}( 1 - C_l(\tau) y,   C_l(\tau) x , t(\tau) ) 
= C_{\th}(\tau ) \td \th (r, z, t(\tau))
, \\
  \om(x, y, \tau) &=  C_{\om}(\tau) \td{\om}( 1 - C_l(\tau) y,  C_l(\tau) x  , t(\tau)) 
  =  C_{\om}(\tau) \td \om( r, z, t(\tau)), \\
 \psi(x, y, \tau)  & =  C_{\om}(\tau) C_l(\tau)^{-2} \td{\psi} (1 - C_l(\tau) y,  C_l(\tau) x, t(\tau)) =  C_{\om}(\tau) C_l(\tau)^{-2}  \td \psi(r, z, t(\tau)),
\eal
\eeq
where $C_l(\tau), C_{\th}(\tau), C_{\om}(\tau), t(\tau)$ are given by $C_{\th}(\tau) = C_l^{-1}(\tau) C^2_{\om}(\tau)$, 
\beq\label{eq:rescal42}
\bal
  C_{\om}(\tau) = C_{\om}(0) \exp \lt( \int_0^{\tau} c_{\om} (s)  d \tau \rt), \ C_l(\tau) =C_l(0) \exp\lt( \int_0^{\tau} -c_l(s) ds \rt) , \   t(\tau) = \int_0^{\tau} C_{\om}(\tau) d\tau .
\eal
\eeq

These rescaling relations are similar to those in \eqref{eq:rescal1}-\eqref{eq:rescal2}. 
Denote
\beq\label{eq:nota_om3}
\Psi(R,\b) = \f{1}{\rho^2} \psi(\rho, \b), \quad \Om(R, \b) = \om(\rho, \b), \quad \eta(R, \b) = (\th_x)(\rho, \b),
\quad \xi(R, \b) = (\th_y)(\rho, \b).
\eeq

Since we rescale the cylinder $D_1 = \{ (r, z) : r \leq 1, |z|\leq 1  \}$, the domain for $(x, y)$ is 
\beq\label{eq:rescale_D}
\td D_1 \teq \{ (x, y) :  |x| \leq C_l^{-1}, y \in [0, C_l^{-1}] \}.
\eeq

Using the above change of variables, one can reformulate the elliptic equation \eqref{eq:euler2} as follows 
\beq\label{eq:elli_euler}
\bal
&- \al^2 R^2 \pa_{RR} \Psi- \al(4+\al) R \pa_R \Psi - \pa_{\b \b} \Psi - 4 \Psi \\
&+ \f{ C_l \rho}{r}( \sin(\b) (2  + \al D_R) \Psi + \cos(\b)  \pa_{\b} \Psi ) 
+ \f{C_l^2 \rho^2}{r^2} \Psi = r \Om,
\eal
\eeq
with boundary condition of $\Psi$  (in the sector $R \leq C_l^{-\al}$) given below 
\beq\label{eq:elli_euler2}
\Psi(R,0) = \Psi(R, \pi/2) =0. 
\eeq
See Sections 9.1 and 9.2 \cite{chen2019finite2} for the details.

\begin{definition}\label{def:supp}
We define the size of support of the rescaling variables $(\th, \om)$ \eqref{eq:rescal41}
\[
S(\tau) = \mathrm{ess} \infim  \{\rho :  \th(x, y,\tau) =0, \om(x, y, \tau ) = 0 \mathrm{  \ for \ } x^2 + y^2 \geq \rho^2 \} .
\]
\end{definition}
Obviously, the support of $\Om, \eta$ defined in \eqref{eq:nota_om3} is $S(\tau)^{\al}$. After rescaling the spatial variable, the support of $(\td{\th}, \td{\om})$ \eqref{eq:omth}, \eqref{eq:euler1} satisfies 
\beq\label{eq:omth_supp}
\mathrm{supp} ( \td{\th}(t(\tau)) ),  \ \mathrm{supp} ( \td{\om}(t(\tau)) ) \subset \{ (r, z) :  ( (r-1)^2 + z^2)^{1/2} \leq  C_l(\tau) S(\tau)  \}.
\eeq

\subsection{Localized elliptic estimates}
Let $\chi_1(\cdot) : [0, \infty) \to [0, 1]$ be a smooth cutoff function, such that $\chi_1(R) = 1$ for $R \leq 1$, $\chi_1(R) = 0$ for $R \geq 2$ and $(D_R \chi_1)^2  \les \chi_1$.  This assumption can be satisfied if $\chi_1 = \chi_0^2$ where $\chi_0$ is another smooth cutoff function. Denote 
\beq\label{eq:solu_chi}
\chi_{\lam}(R) = \chi_1(R/ \lam), \quad  \Psi_{\chi_{\lam}} = \Psi \chi_{ \lam } , \quad \Om_{\chi_{\lam}} = \Om  \chi_{\lam}.
\eeq
In Section 9.2.2 in \cite{chen2019finite2}, we showed that the leading order part of $\Psi$ near $0$ is captured by 
\beq\label{eq:l12Z0}
L_{12}(Z_{\chi_{\lam}})(0) =  - L_{12}(\Om)(0)  + 4\al \int_0^{\pi/2} \Psi(0, \b) \sin(2\b) d \b ,
\eeq
when $\lam \geq (S(\tau))^{\al}$.

As discussed at the beginning of Section \ref{sec:euler_blowup}, we can generalize Proposition 9.9 in \cite{chen2019finite2} as follows. 
\begin{prop}\label{prop:elli_euler}
Suppose that $\Psi$ is the solution of \eqref{eq:elli_euler} and $\Om \in \cH^k$. There exists some absolute constant $\al_2$ and constant $\d_k \in (0, 1/4)$, such that if $\al <\al_2$, $\lam = \d_k C_l^{-\al}, \ C_l S <  \al \cdot \d_k^{1/\al +1}$, then we have 
\[
\bal
&\al^2 || R^2 \pa_{RR} \Psi_{\chi_{\lam}} ||_{\cH^k} + 
\al || R \pa_{R \b} \Psi_{\chi_{\lam}}  ||_{\cH^k}  +|| \pa_{\b\b} ( \Psi_{\chi_{\lam}} - \f{\sin(2\b)  }{\al \pi} 
(L_{12}(\Om) + \chi_1 L_{12}(Z_{\chi_{\lam}})(0) )
)   ||_{\cH^k} \les_k  || \Om ||_{\cH^k} , \\
 &| L_{12}(Z_{\chi_{\lam}})(0) |  \les  3^{- \f{1}{\al}}  || \Om \f{1+R}{R}||_{L^2}.
\eal
\]
\end{prop}

In Proposition 9.9 in \cite{chen2019finite2}, we prove the case for $k=3$ with $\d_k = 2^{-13}$. The following results generalize Proposition 9.11 from \cite{chen2019finite2}. The conditions $\lam = \d_k C_l^{-\al}, C_l S< \al \d_k^{1/\al + 1} $ guarantee that $\lam \geq (S(\tau))^{\al}$ in \eqref{eq:l12Z0}. 

\begin{prop}\label{prop:psi2}
Let $\bar{\Psi}_0(t)$ be the solution of \eqref{eq:elli_euler} with source term $\Om = \bar{\Om}_0 = \bar{\Om} \chi(R / \nu)$, and $\al_2, \d_k$ be the constants in Proposition \ref{prop:elli_euler}.
If $\al < \al_2$, $\lam = \d_k C_l^{-\al} ,  \ 
C_l S < \al \d_k^{1/\al + 1}, \ 2 \nu < \lam$, then we have 
\[
\bal
& \al || \f{1+R}{R} D_R^2 \bar{\Psi}_{0,\chi_{\lam}} ||_{\cW^{k, \infty}} + 
\al || \f{1+R}{R} R \pa_{R \b} \bar{\Psi}_{0, \chi_{\lam}}  ||_{\cW^{k, \infty}} \\
&  +||\f{1+R}{R} \pa_{\b\b} ( \bar{\Psi}_{0,\chi_{\lam}} - \f{\sin(2\b)  }{\al \pi} 
(L_{12}( \bar \Om_0) + \chi_1 L_{12}( \bar Z_{\chi_{\lam}})(0) )
)   ||_{\cW^{k,\infty}} \les_k  \al , \\
 &| L_{12}( \bar Z_{\chi_{\lam}})(0) |  \les  3^{- \f{1}{\al}}  ,
\eal
\]
where $ L_{12}( \bar Z_{\chi_{\lam}})(0)$ associated to $\bar \Psi_0$ is defined in \eqref{eq:l12Z0}.
\end{prop}

The case of $k=5$ is Proposition 9.11 in \cite{chen2019finite2}. The general case follows from a similar argument. 

Choosing $k=100$ in Propositions \ref{prop:elli_euler} and \ref{prop:psi2} and using \eqref{eq:solu_chi}, we obtain the elliptic estimates for $ \Psi(R, \b) = \Psi_{\chi_{\lam}}(R, \b) , R \leq \lam = \d_{100} C_l^{-\al}$ 
in the dynamic rescaling equations. Using the relations \eqref{eq:euler_polar} and \eqref{eq:nota_om3} and rescaling the domain, we obtain that $R \leq \lam$ is equivalent to 
\[
\rho \leq  C_l^{-1} \d_{100}^{1/\al},  \quad  \rho C_l \leq \d_{100}^{1/\al}, 
\quad | (r, z) - (1, 0) | \leq \d_{100}^{1/100}. 
\]
Thus, we have $\cH^{100}$ estimate of the stream function $\td \psi$ \eqref{eq:euler1} in the physical domain 
\beq\label{eq:euler_R2al}
B_{ (1, 0)}( R_{2,\al}), \quad R_{2,\al} =  \d_{100}^{1/\al}  < 1/4. 
\eeq


Now, we are in a position to prove Theorem \ref{thm:euler_blowup}.  Denote 
\beq\label{eq:euler_comp}
D_{R_2} \teq \{ (r, z) : r \in (0, 1), z \neq 0  \} \cap B_{(1,0)}(R_{2,\al}) ,
\quad \Ups \teq \{ (r, z) : r = 1 \mathrm{\ or \ } z = 0 \}.
\eeq

\begin{remark}
In later estimates, we will choose $\al$ to be very small and then choose $S(0)$ to be very large. Finally, we choose $C_l(0)$ much smaller than $S(0)^{-1}, \al$. We treat $C_l(0)$ roughly as $0$.
\end{remark}

\subsection{Blowup, control of the trajectory and the interior regularity }\label{sec:euler_traj}

Recall from Definition \ref{def:supp} the size of support $S(t)$  in the dynamic rescaling equations. Then $C_l(t) S(t)$ is the size of the support of the solution in the physical space. In \cite{chen2019finite2}, 
for some small $\al_0> 0$ and any  $0< \al < \al_0$,  we construct a class of $C^{\al}$ singular solutions with the following control of the support and trajectory. For a point within the support of the initial data $(\th_0, \om_0)$ \eqref{eq:rescal41} and with trajectory $(R(t), \b(t))$, $R(t)$ satisfies a uniform estimate 
\beq\label{eq:supp}
 C_l(t) R(t)^{1/\al} \leq C(\al, S(0)) C_l(0) 
\eeq
for some constant $C(\al, S(0))$ up to the blowup time. See Section 9.3.5 in \cite{chen2019finite2}. For initial data with support size $S(0)$, we can pick $C_l(0)$ small enough such that 
\beq\label{eq:supp1}
C_l(t) S(t) \leq  C( \al, S(0)) C_l(0)  < R_{2, \al } / 8 \teq  R_{1, \al} ,
\eeq
where $R_{2,\al}$ is defined in \eqref{eq:euler_R2al}. It follows 
\[
(  S(t))^{\al} < R_{1,\al}^{\al}  C_l^{-\al} < (R_{2,\al} / 2)^{\al} C_l^{-\al} < \d_{100} C_l^{-\al}.
\]
Thus, within the support of the solution, we can apply the high order elliptic estimates $(k=100)$ in Propositions \ref{prop:elli_euler} and \ref{prop:psi2} to estimate $\Psi(R, \b)$. 

As discussed at the beginning of Section \ref{sec:euler_blowup}, using the argument in \cite{chen2019finite2} and the higher order elliptic estimates in Propositions \ref{prop:elli_euler} and \ref{prop:psi2}, we can generalize the blowup results in Theorem \ref{thm:bous_blowup} for the 2D Boussinesq equations to the 3D axisymmetric Euler equations. In particular, we have the control of the support and the trajectory \eqref{eq:supp}-\eqref{eq:supp1} and obtain the following 
generalization of \eqref{eq:U_Ck} and \eqref{eq:V_Ck} for the solution $(\th, \om, \psi)$ in the dynamic rescaling formulation \eqref{eq:rescal41} of \eqref{eq:euler1}
\beq\label{eq:euler_Ck}
|| \na \th(\tau) ||_{\cC^{60}} +  || \om(\tau) ||_{\cC^{60}} + || \f{1}{\rho} \na ( \psi(\tau) \chi_{\lam} ) ||_{\cC^{60}} \les_{\al} 1 .
\eeq

In general,  $\th, \om, \psi$ are only defined in the bounded and rescaled domain \eqref{eq:rescale_D}.  Since $\th, \om, \psi \chi_{\lam}$ have compact support with $ S(t)  < \f{1}{2} C_l^{-1}$ \eqref{eq:supp1} or $(2 \lam)^{1/\al} < \f{1}{2} C_l^{-1}$ (see Lemma \ref{prop:elli_euler}), these variables can be extended naturally to $(x, y) \in \R_+ \times \R$. Then the $\cC^k$ norm \eqref{norm:ck} of these variables are well-defined. From \eqref{eq:omth_supp} and \eqref{eq:supp1}, the solution $\td \th(t, r, z) , \td \om(t,r, z)$ are supported in $B_{(1,0)}(R_{1,\al})\subset B_{(1,0)}(R_{2,\al })$. Since $\chi_{\lam} = 1$ in $B_{(1,0)}(R_{2,\al })$, using \eqref{eq:euler_Ck}, \eqref{eq:euler2}, the rescaling relation \eqref{eq:rescal41}, \eqref{eq:omth}, and estimates similar to those in Lemma \ref{lem:cCk}, we yield 
\beq\label{eq:euler_C50}
 || \td \th(t) ||_{C^{50}(\S)} + || \td \om(t) ||_{C^{50}(\S)} + || u^r(t)||_{C^{50}(\S)}
 + || u^z(t)||_{C^{50}(\S)}  \les C( \al, \S, C_l(\tau), C_{\om}(\tau) )
\eeq
for the compact domain $\S \subset D_{R_2}$ \eqref{eq:euler_comp}. Since $r, \f{1}{r}$ is smooth away from $r=0$, from \eqref{eq:omth}, we yield $ ( u^{\vth})^2 , \om^{\vth} \in C^{50}(\S) $. We prove the estimates for $\om^{\vth}, (u^{\vth})^2, u^r, u^z$ in result (c) in Theorem \ref{thm:euler_blowup}

In the $(r, z)$ coordinate, from \eqref{eq:supp}, \eqref{eq:supp1}, and \eqref{eq:euler_polar}, for $ (r_0, z_0) \in \supp( \td \om_0) \cup \supp (\td \th_0) = \supp (u_0^{\vth}) \cup \supp( \om_0^{\vth})$, we have 
\beq\label{eq:supp2}
\g_t( r_0, z_0) \in B_{(1,0)}( R_{1,\al} ) .
\eeq
This proves result (b) in Theorem \ref{thm:euler_blowup}.

\subsection{Result (a): Blowup of \texorpdfstring{$\om_p$}{Lg}}

Recall the poloidal component of $\om$ from \eqref{eq:polo_w}
\[
 \om_p =\om^r e_r + \om^z e_z , \quad \om^r = - \pa_z u^{\vth}, \quad \om^z = \f{1}{r} \pa_r( ru^{\vth}).
\]
From \eqref{eq:rescal41}, \eqref{eq:rescal42}, and \eqref{eq:omth}, we get $\pa_x \th(x, y, \tau) = C_{\th} C_l \pa_z \td \th = C_{\om}^2 \pa_z( (u^{\vth} / r)^2)$. It follows 
\[
I(\tau) \teq  \int_0^{ t(\tau)} || \pa_z ( \f{  ( u^{\vth}(s) )^2 } {r^2} ) ||_{\inf}  ds
= \int_0^{\tau}  C_{\om}^{-2} || \pa_x \th(x, y, s)||_{\inf} d s .
\]

The nonlinear stability result implies that $|| \pa_x \th(x, y, s)||_{\inf} \approx || \bar \th_x ||_{\inf} \gtr_{\al} 1$ and $C_{\om}(\tau) \leq \exp( - \tau / 2 )$. See Section 9.3.6 in \cite{chen2019finite2} for the derivations. Since $u^{\vth}$ is supported in $B_{(1,0)}(1/4)$ and $ r u^{\vth}(r, z, t)$ is transported \eqref{eq:euler1}, we obtain
\[
|| \pa_z ( \f{  ( u^{\vth}(s) )^2 } {r^2} ) ||_{\inf}
\les || \pa_z u^{\vth}(s)||_{\inf} || r u^{\vth}(s ) ||_{\inf}
\les ||  \om_p(s) ||_{\inf} || r u_0^{\vth} ||_{\inf}.
\]
Therefore, we establish 
\[
\int_0^{\tau} \exp( s ) ds \les_{\al}   I(\tau) 
=  \int_0^{ t(\tau)} || \pa_z ( \f{  ( u^{\vth}(s) )^2 } {r^2} ) ||_{\inf}  ds
\les || r u_0^{\vth} ||_{\inf} \int_0^{ t(\tau)} ||  \om_p(s) ||_{\inf}  ds.
\]
Taking $\tau \to \infty$ yields $\int_0^{T_* } ||  \om_p(s)||_{\inf} ds  = \infty$, where $T_* = t(\infty) < +\infty$ \eqref{eq:rescal42} is the blowup time.

\subsection{Interior regularity of  \texorpdfstring{ $u_0^{\vth}$}{Lg} }\label{sec:euler_uth}
The smoothness of $u^{\vth}$ does not follow from $(u^{\vth})^2$ since $u^{\vth}$ can degenerate. 
In this section, we choose $u_0^{\vth}$ smooth in the interior of the domain. In Section \ref{sec:euler_uth_in}, we show that the regularity can be propagated. 

Let $\S_1$ be any compact domain with 
\beq\label{eq:euler_comp2}
 \S_1 \subset \{ (x, y): x \neq 0, y > 0 \}.
\eeq

\begin{remark}\label{rem:change2}
Recall from Remark \ref{rem:oversight} that we made a minor change of the approximate steady state of the 3D Euler equations in the updated arXiv version of \cite{chen2019finite2}, i.e. \cite{chen2019finite2arXiv}. 
More precisely, in \cite{chen2019finite2arXiv}, we modify $\bar \th_{old}$ used in \cite{chen2019finite2} by $\bar \th $ below 
\beq\label{eq:th_new}
\bar \th_{old} = \int_0^x \bar \th_x( z, y) , \quad \bar \th = 1 + \int_0^x \bar \th_x(z, y) d z, 
\eeq
where $\bar \th_x(x, y) = \bar \eta(R, \b)$ \eqref{eq:profile}. See Eq (A.20) in \cite{chen2019finite2arXiv}.
This modification does not change $\na \bar \th$, i.e. $\na \bar \th = \na \bar \th_{old}$, and we have $  \bar \th \in C^{1,\al}$. We remark that  \cite{chen2019finite2arXiv} and \cite{chen2019finite2} are essentially the same except for this minor change. In the following derivations, we use this new approximate steady state $\bar \th$. 
\end{remark}

The initial data for $\th$ in \cite{chen2019finite2arXiv} (see Eq (9.55) in Sections 9.3.2 and 9.3.6 \cite{chen2019finite2arXiv}) is chosen as 
\[
 \th_0(x, y) = \bar \th_0(x, y) = \chi_1(R / \nu) \bar \th( x, y),
\]
where $\bar \th$ is given in \eqref{eq:th_new} and $\chi_1$ is some smooth cutoff function satisfying that $\chi_1^{1/2}$ is smooth. We have the smoothness of $\chi_1^{1/2}$ by choosing $\chi_1 = \td \chi_1^2$ for another smooth cutoff function $\td \chi_1$. Since $\bar \th_x( x, y ) > 0$ for $x > 0 $,  $\bar \th(0, y) \geq 1$ and $\bar \th$ is even, we get $ \bar \th \geq 1.$
Using induction and the Leibniz rule, we get $\bar \th^{1/2} \in C^{60}(\S_1)$. 
Since $\bar  \th_0^{1/2} = \bar \th^{1/2}  \chi_1^{1/2}(R /\nu), R \in C^{60}(\S_1) $, and $ \chi_1^{1/2}$ is smooth, we further obtain $\th_0^{1/2}(x, y) = \bar \th_0^{1/2}(x, y) \in C^{60}(\S_1)$.

Since $\S_1$ is an arbitrary compact domain with \eqref{eq:euler_comp2}, using the relation among $\th_0, \td \th_0, u^{\vth}_0$ \eqref{eq:omth}, \eqref{eq:rescal41} and the relation between the coordinate $(r, z)$ and $(x, y)$ in \eqref{eq:euler_polar}, we obtain $u_0^{\vth}(r, z) 
= \th_0^{1/2} / r \in C^{60}(\S)$ for any compact domain $\S \subset D_{R_2}$ \eqref{eq:euler_comp}. Moreover, $u_0^{\vth}$ is even in $z$ and this symmetry is preserved by \eqref{eq:euler1}.

\subsection{Propagate the regularity of \texorpdfstring{$u^{\vth}$}{Lg}}\label{sec:euler_uth_in}





In Theorem \ref{thm:euler_blowup}, it remains to prove $u^{\vth}(t) \in L^{\inf}( [0, T], C^{50}(\S))$ for any compact set $\S \subset D_{R_2}$ \eqref{eq:euler_comp}. Recall $\Ups$ from \eqref{eq:euler_comp}.

The idea is that if the domain $\Sigma$ is away from $\supp ( u^{\vth} )$, then $u^{\vth}$ vanishes and it is smooth. Otherwise, the trajectory $g_t$ \eqref{eq:char_rz} through $\Sigma$ can be contained in a compact set in $ ( D_1 \bsh \Ups )  \cap B_{(1,0)}( R_{2,\al})$ and is smooth according to Theorem \ref{thm:euler_blowup}. Since $ r u^{\vth}$ is transported along the trajectory and the initial data $ u^{\vth}_0$ is smooth, we then obtain that $u^{\vth}(t)$ is smooth in $\S$.

\begin{proof}
Recall $D_1, \Ups$ from \eqref{eq:euler_HL_domain}, \eqref{eq:euler_comp}. We fix $T< T_{*}$ and a compact set $\S \subset ( D_1 \bsh \Ups) \cap B_{(1,0)}(R_{2,\al}) $. 
Consider the flow map $g_t: (r, z) \in D_1 \to D_1$ generated by $(u^r, u^z)$
\beq\label{eq:char_rz}
  \f{d}{dt} g_t( r, z ) = ( u^r( g_t(r, z), t) , u^z( g_t(r, z), t ) , \quad g_0(r, z) = (r, z).
\eeq

It is the same as $\td \g_t$ in \eqref{eq:char_polo},\eqref{eq:char_polo2}.
Since $u^r, u^z \in L^{\inf}( [0, T], C^{1,\al}(D_1))$, we can solve the above ODE with $g_t, g_t^{-1}$ being Lipschitz in $(r, z)$. Due to the non-penetrated condition \eqref{eq:bc_vanish}, we obtain that $g_t, g_t^{-1}$ are bijections from $D_1$ to $D_1$ and $D_1 \bsh \Ups$ to $D_1 \bsh \Ups$. One should not confuse \eqref{eq:char_rz} with \eqref{eq:bichar1}. Denote $L_{g}$ by the Lipschitz constant of $g_t, g_t^{-1}$ for $t \in [0, T]$. Recall from \eqref{eq:euler1} that 
\[
\pa_t (r u^{\vth}) + ( u^r \pa_r + u^z \pa_z )  (r u^{\vth}) = 0.
\]
We abuse the notation by denoting $ x = (r, z)$. We get $r_t u^{\vth}( t , g_t(x)) = r_0 u^{\vth}_0( x)$. Inverting $g_t$ yields 
\beq\label{eq:uth_trans}
r u^{\vth}( t, x) = r( g_t^{-1}(x) ) u^{\vth}_0( g_t^{-1}(x)).
\eeq

From result (b) in Theorem \ref{thm:euler_blowup}, we yield 
\beq\label{eq:uth_supp}
 \supp( u^{\vth}(t)) \subset g_t( \supp (u^{\vth}_0) ) \cap B_{(1,0)}(R_{1,\al}), \quad t \in [0, T_* ).
\eeq
Since $\S$ is compact, it suffices to show that for any $x \in \S$, there exists $\d > 0$ such that $u^{\vth}(t) \in C^{50}(B_x(\d))$ with norm uniformly bounded on $[0, T]$. Since $g_t, g_t^{-1}$ are bijections and Lipschitz in $t$ and $x$ and $ g_t^{-1} (\S) \cap \Ups = \emptyset$, we yield 
\beq\label{eq:uth_d1}
\d_1 \teq \min_{t \in [0, T]} \dist( g_t^{-1}(\S), \Ups) > 0 .
\eeq

Now, we define
\beq\label{eq:uth_d}
\bal
&\d = \f{1}{ 4 ( L_{g} + 1) } \min( R_{1, \al}, \d_1), \quad 
\Sigma_2 \teq \{ x:  \dist( x, \Ups)  \geq \d \} \cap \bar B_{(1,0)} ( 4 R_{1,\al} ) \cap \bar D_1,  \\
&S(t, \rho) = \{ x: |x- y| \leq \rho, y \in \supp(u^{\vth}(t)) \} \cap D_1.
\eal
\eeq
The set $S(t, \rho)$ is the $\rho$ neighborhood of $\supp( u^{\vth}(t))$, and  $\S_2$ is a compact set in $D_1 \bsh \Ups \cap B_{(1,0)}( R_{2,\al})$. From result (c) in Theorem \ref{thm:euler_blowup}, we have $u^r, u^z \in L^{\inf} ( [0, T], C^{50}(\S_2) ) $. 


If $x \in \S \bsh S(t, 2 \d)$, 
we get $u^{\vth}(t, x) = 0$ on $B_x( \d )$ and thus $ u^{\vth}(t) \in C^{50} (B_x(\d))$.

If $x \in \S \cap S(t, 2 \d))$, from \eqref{eq:uth_supp}, we have $x = \g_t(x_0) + z, x_0 \in B_{(1,0)}(R_{1,\al}), |z| \leq 2 \d $. Hence, we get $B_x(\d) \subset B_{ \g_t(x_0)}(3\d)$. Next, we show that the trajectory passing through $B_{ \g_t(x_0)}(3\d)$ is contained in $\S_2$.  Recall that $L_{g}$ is the Lipschitz constant of $g_t, g_t^{-1}$ on $[0, T]$. For any $ s \in [0, t]$ and $y = g_t(x_0) + z \in D_1, |z| \leq 3\d $, using \eqref{eq:uth_d1}, \eqref{eq:uth_d}, we get 
\[
\bal
 | g_s^{-1}(y) - g_s^{-1} g_t(x_0)|  &\leq L_{g} | y  - g_t(x_0) | \leq 3 L_{ g} \d,  \\
\dist( g_s^{-1} ( y) , \Ups) 
&\geq \dist( g_s^{-1} g_t(x_0), \Ups) -3 L_{ g} \d
\geq \d_1 - 3 L_{ g} \d > \d,  \\
 | g_s^{-1} (y) - (1,0)|
&\leq   |  g_s^{-1} g_t(x_0) - (1,0)| + 3 L_{ g} \d \leq 3 L_{ g} \d + R_{1, \al} \leq 2 R_{1,\al},
\eal
\]
where we have used $ g_{\tau}( x_0)  \in B_{(1,0)}(R_{1,\al}) $ from Theorem \ref{thm:euler_blowup} for $x_0 \in \supp( u^{\vth}_0)$ and $\tau \in [0, T]$. Hence, we establish
\[
g_s^{-1} B_x(\d) \subset g_s^{-1} B_{\g_t(x_0)}( 3\d) \subset \Sigma_2, \quad s \in [0, t].
\]
Since $u^r, u^z \in L^{\inf}( [0, T], C^{50}(\S_2))$ \eqref{eq:euler_C50}, solving \eqref{eq:char_rz} backward with backward initial data in $B_x(\d)$, we yield $ g_t^{-1} \in C^{50}( B_x(\d))$, with bound depending on $T$ and $\S_2$. Since $ r \in [1/2, 1]$ within the support of $u^{\vth}( \cdot)$, using \eqref{eq:uth_trans}, we prove $u^{\vth}(t) \in C^{50}(B_x(\d))$ with bound depending on $T, \S_2$. 

Combining both cases $x \in \S \bsh S(t, 2\d), x \in \S \cap S(t, 2\d)$, we obtain $u^{\vth} \in L^{\inf}( [0, T], C^{50}( B_x(\d))$. Since $\d$ is uniform for $x \in \S$ and $\S$ can be covered by finite balls with radius $\d$, we obtain $u^{\vth} \in L^{\inf}( [0, T], C^{50}( \S) )$.
\end{proof}

\subsection{Proof of Theorem \ref{thm:euler_R3_blowup}}

The proof of Theorem \ref{thm:euler_R3_blowup} is similar and simpler than that of Theorem \ref{thm:euler_blowup} since we do not need to control the trajectory and estimate the swirl $u^{\vth}$.

\begin{proof}

The first part of the theorem about the blowup result from some $\om_0^{\vth} \in C_c^{\al}$ and $u^{\vth}_0 = 0$ has been proved in Theorems 1, 2 in \cite{elgindi2019stability}. Moreover, higher order estimates of the perturbation in the $\cH^k$ norm for $k\geq 1$ and the profile have been established in Theorem 2 in \cite{elgindi2019stability}. Thus, the interior regularity $\om^{\vth}, u^r, u^z \in L^{\inf}( [0, T], C^{50}(D_2))$ in Theorem \ref{thm:euler_R3_blowup} follows from these higher order estimates and the argument in the proof of Theorem \ref{thm:bous_blowup}.

It remains to estimate $u_r^r(t, 0, 0)$. Let $(r, \vth, z)$ be the cylindrical coordinate in $\R^3$ \eqref{eq:polar_basis}, $\rho, R, \b$ be the modified polar coordinate for $(r, z)$ and $\Om$ be the vorticity in the new coordinate 
\beq\label{eq:polar_wholeR3}
\b = \arctan( z / r), \quad \rho = (r^2 + z^2)^{1/2}, \quad R = \rho^{\al}, \quad \Om(R, \b) = \om^{\vth}(\rho, \b). 
\eeq
Firstly, we show that 
\beq\label{eq:L12_R3}
u_r^r(0, 0) = - \f{1}{2} L(\om^{\vth})(0) = - \f{1}{2\al} L( \Om)(0), \quad 
L(f)( r ) \teq \int_r^{\inf} \int_0^{\pi/2} \f{ f( r, \b) \cos^2(\b) \sin(\b) }{ r } d r d \b.
\eeq
This can be obtained by following the derivations in \cite{elgindi2019finite,elgindi2019stability}. For the sake of completeness, we derive \eqref{eq:L12_R3} in Appendix \ref{app:urr} using the formula  $\uu = \na \times (-\D)^{-1} \om$ in $\R^3$.

In \cite{elgindi2019stability}, it is proved that the blowup solution $\Om$ satisfies
\beq\label{eq:profile_R3}
\bal
& \Om(R, \b, t) = \f{1}{\lam(t)} \Xi( \f{R}{ \lam^{1+\d}}, \b, s ) ,  \quad \f{ds}{dt} = \f{1}{\lam(t)},   \quad || \Xi||_{L^{\inf}} \les_{\al} 1,
 \\
& \Xi =  F + \e(\tau) = F_* + \al^2 g + \e(\tau), \quad \f{1}{\al} L(F)(0) =- 1 + O(\al), \quad L(\e(\tau))(0) \equiv 0,
\eal
\eeq
for some rescaled time $s$ and factor $ \f{T_*}{T_* - t} \lam(t) \to 1$ as $t\to T_*$, where $T_*$ is the blowup time. Here $F = F_* + \al^2 g$ is the time-independent self-similar profile of \eqref{eq:euler} without swirl constructed in \cite{elgindi2019finite}. See Sections 2.3-2.5 in \cite{elgindi2019stability}. In particular, for $\al$ small enough, we get 
\[
\bal
 u_r^r(0, 0, t) & =- \f{1}{2 \al}  L(\Om)(0)
 =  - \f{1}{2 \al \lam(t)}  L( \Xi)(0)  
 = - \f{1}{2 \al \lam(t)} L( F)(0) > 0,  \\
    u_r^r(t,0,0) & \gtr_{\al} || \Om(t) ||_{L^{\inf}}
 = || \om(t)||_{L^{\inf}}. 
 \eal
\]
The last inequality is a consequence of that $u_r^r(t, 0,0)$ and $||\om^{\vth}||_{L^{\inf}} = || \om||_{L^{\inf}}$ have the same scaling and that the blowup is asymptotically self-similar.  It follows $\int_0^{T_*} u_r^r(t, 0,0 )d t = \inf$.
\end{proof}

\begin{remark}
In \cite{elgindi2019finite}, the setup of the 3D axisymmetric Euler equations is not conventional and differs from \eqref{eq:euler1}-\eqref{eq:euler21} by a negative sign. See Section 2 in \cite{elgindi2019finite} for this difference. Therefore, in the current setting, the profile $F$ for the vorticity is negative, and $\f{1}{\al}L(F) (0) = -1 + O(\al)$, while the profile  $F$ is positive in \cite{elgindi2019finite,elgindi2019stability}. These changes do not affect the positive sign of $u^r_r(0,0,t)$. 
\end{remark}

\noindent
{\bf Acknowledgments}. The research was in part supported by NSF Grants DMS-1907977 and DMS- 1912654, and the Choi Family Gift Fund. We would like to thank Professors Tarek Elgindi, Sasha Kiselev, Vladimir Sverak and Yao Yao for their constructive comments on an earlier version of our paper and for bringing to our attention several relevant references. 

\appendix 

\section{Review of the construction of unstable solutions}\label{app:WKB}


We provide a brief review of the construction of the unstable solution in \cite{lafleche2021instability,vasseur2020blow} via a WKB expansion and explain the connections among the WKB expansion, the bicharacteristics-amplitude ODE system \eqref{eq:bichar1}-\eqref{eq:bichar3}, and the growth of the unstable solution.

\subsection{Construction of the approximate solution}
Suppose that $\uu(t, x)$ is a singular solution of \eqref{eq:euler}. Denote by $\g_t(x)$ the flow map
\beq\label{eq:ODE_flow}
 \f{d}{dt} \g_t(x) =  \uu( t, \g_t(x)), \quad \g_0(x) = x.
\eeq
The main idea in \cite{vasseur2020blow} is to construct an approximate solution  to \eqref{eq:euler_lin} using a WKB expansion 
\beq\label{eq:WKB0}
v(t, x) \approx b(t, x) \exp( \f{ i S(t, x)}{\e} )
\eeq
for sufficiently small $\e$ and the characteristics of the flow, where $b(t, x) \in \R^3$ and $S$ is a scalar. Plugging the above ansatz into \eqref{eq:euler_lin}, we obtain 
\[
R_{\e} = ( \pa_t  + \uu \cdot \na + \na \uu )  v  = 
\f{i}{\e}( \pa_t + \uu \cdot \na ) S  \cdot b e^{ iS / \e} 
+  ( \pa_t + \uu \cdot \na   + \na \uu) b  \cdot  e^{i S / \e},
\]
where $ (\na \uu) f = f \cdot \na \uu =  f_j \pa_j u_i  e_i $. To eliminate the $O(\e^{-1})$ term, one requires 
\beq\label{eq:WKB1}
(\pa_t + \uu \cdot \na ) S =0.
\eeq
Then we can rewrite $R_{\e}$ as follows 
\beq\label{eq:WKB12}
R_{\e} = (\pa_t + \uu \cdot \na + \na \uu) b  \cdot e^{ i S/ \e} 
\teq F(t, x) \cdot e^{ i S/ \e} , \quad  F \teq  (\pa_t + \uu \cdot \na + \na \uu) b .
\eeq

An important observation in \cite{vasseur2020blow} is that for high frequency oscillation, i.e. small $\e$, the pressure term in \eqref{eq:euler_lin} is almost local. 
We would like to construct $(v, Q)$ such that 
\[
R_{\e} = F(t, x) e^{i S/ \e} = \na Q + E_{\e} ,
\]
where $E_{\e}$ is a small error term. 
This is possible since $Q$ is one order more regular than a highly oscillatory function $F(t, x) e^{i S / \e}$. By integration and exploiting the cancellation, $Q$ can be of order $O(\epsilon)$. In fact, taking $\na \times$ on both sides, we obtain 
\[
\na \times R_{\e} = ( \na \times F)  e^{i S / \e} + \f{i}{\e} (\na S \times F) e^{i S / \e} 
=  \na \times ( \na Q + E_{\e} ) = \na \times E_{\e}\;.
\] 

To eliminate the $O(\e^{-1})$ term, we require $\na S \times F = 0$, which implies $F = c(t, x) \na S $ for some scalar $c(t, x)$. In this case, one can construct the pressure $Q$ as follows 
\[
Q = - i \e  c(t, x)   e^{i S / e} .
\]
As a result, the error is given by 
\beq\label{eq:WKB_error}
E_{\e} = R_{\e} - \na Q = c \na S e^{i S / \e} + i \e \cdot \na c \cdot e^{i S / \e}
+ i \e c \f{ i \na S}{\e} e^{iS / \e} = i \e \cdot \na c \cdot e^{i S / \e}.
\eeq
Suppose that $c$ is smooth, then the $L^p$ norm of the error $E_{\e}$ is small as $\e \to 0$. 

From $F = c(t, x) \na S$ and \eqref{eq:WKB12}, we yield
\[
( \pa_t + \uu \cdot \na  + \na \uu) b = F(t, x) = c(t, x) (\na S)( t, x).
\]

Using the Lagrangian coordinates and the flow map $\g_t$ \eqref{eq:ODE_flow}, we get 
\[
\pa_t b(t, \g_t(x)) = - ( \na \uu ) b(t, \g_t(x)) + c(t, x) (\na S)(t,  \g_t(x)) .
\]

Denote 
\beq\label{eq:WKB2}
\xi_t( x) \teq   (\na S)( t, \g_t(x)) , \quad b_t(x) \teq b(t, \g_t(x)) \;.
\eeq
The above equation reduces to 
\beq\label{eq:ODE_bt}
 \f{d}{dt} b_t = - (\na \uu) b_t + c(t, x) \xi_t .
\eeq

Next, we determine the equations for $b, \xi$. In order for $v(t, x)$ to be incompressible, from the ansatz \eqref{eq:WKB0} and 
\[
 \na \cdot v(t, x) = ( \na \cdot b ) e^{i S / \e}  + \f{i}{\e} b \cdot \na S e^{ i S /\e} \; ,
\]
we require $ b(t, x) \cdot (\na S)(t, x) = 0$ to eliminate the $O(\e^{-1})$ term. In the Lagrangian coordinates, this condition is equivalent to enforcing
\beq\label{eq:WKB3}
b(t, \g_t(x)) \cdot (\na S)( t, \g_t(x)) = b_t(x) \cdot \xi_t(x) = 0.
\eeq

Taking the gradient in the transport equation \eqref{eq:WKB1}, we get 
\[
 ( \pa_t + \uu \cdot \na  ) \na S =  - (\na \uu)^T \na S.
\]

Using the Lagrangian coordinates and \eqref{eq:WKB2}, we derive
\beq\label{eq:ODE_xit}
 \f{d}{d t} \xi_t = \f{d}{dt} (\na S)(t, \g_t(x)) = - (\na \uu )^T (\na S )( t ,\g_t(x))
 =  - (\na \uu )^T  \xi_t.
\eeq

The incompressible condition \eqref{eq:WKB3} implies $ \f{d}{dt} (b_t \cdot \xi_t ) = 0$. Thus, from \eqref{eq:ODE_bt} and \eqref{eq:ODE_xit}, we get 
\[
\la c(t, x) \xi_t ,  \xi_t \ra - \la  (\na u) b_t ,  \xi_t \ra - \la  (\na \uu)^T \xi_t ,  b_t  \ra = 0,
\]
where $\la p, q \ra = p_i q_i $. It follows that
\[
c(t, x) = 2 \f{ \xi_t^T ( \na \uu ) b_t}{ | \xi_t|^2} .
\]
Thus, from \eqref{eq:ODE_flow},\eqref{eq:ODE_bt},\eqref{eq:ODE_xit}, $\g_t, \xi_t, b_t$ satisfy the bicharacteristics-amplitude ODE system \eqref{eq:bichar1}-\eqref{eq:bichar3} of \eqref{eq:euler} \cite{lafleche2021instability,vasseur2020blow}

The above derivation reveals the main idea behind the construction of an approximate solution to \eqref{eq:euler_lin} in \cite{vasseur2020blow} and the relationship between the WKB expansion \eqref{eq:WKB0} and the bicharacteristics-amplitude ODEs \eqref{eq:bichar1}-\eqref{eq:bichar3}. The last step is to localize the solution $v(t, x)$ to some trajectory and add a correction to $v(t,x)$ \eqref{eq:WKB0} so that it is incompressible. We refer to  \cite{vasseur2020blow} for the details.

\subsection{Growth of the solution}

The solution $v(t, x)$ satisfies \eqref{eq:euler_lin} up to an error similar to \eqref{eq:WKB_error}. Since $E_{\e}$ contains the highly oscillatory phase $e^{i S / \e}$, the error may not be small in $C^{k, \al}$ or $H^s$ norm. In \cite{vasseur2020blow}, based on  the WKB construction \eqref{eq:WKB0} and using the smallness of the error in the $L^p$ norm, 
the authors constructed an approximate solution to \eqref{eq:euler_lin} with error controlled by $\e$. To prove the instability, they further showed the growth of $v(t, x)$. From \eqref{eq:WKB0}, the growth of $|| v||_p$ is due to $|| b_t||_p$. The authors showed that 
if the velocity $\uu(t, x)$ is smooth, the system \eqref{eq:bichar1}-\eqref{eq:bichar3} satisfies the following conservations along the characteristic $\g_t(x)$
\[
\bal
 \om(t, \g_t(x) ) \cdot \xi_t & = \om_0(x) \cdot \xi_0 ,  \quad 
b_t \cdot \xi_t   = \td b_t \cdot \xi_t , \quad 
( b_t \times \td b_t)  \cdot \xi_t = (b_0 \times \td b_0) \cdot \xi_0 ,
 \eal
\]
where $\om = \na \times \uu$ is the vorticity of the blowup solution $\uu$, $ \xi_t, b_t , \td b_t$ are the solution to \eqref{eq:bichar1}-\eqref{eq:bichar3} with  initial data $x_0, \xi_0, b_0, \td b_0$, $ b_0 \cdot \xi_0 = \td b_0 \cdot \xi_0 = 0$ and $ b_0, \td b_0, \xi_0$ being linearly independent. 

From the first and the third identity, formally, $ b_t \times \td b_t$ plays a role similar to $\om(t, \g_t(x))$. Indeed, using the above conservations, 
 the authors further proved 
\beq\label{eq:growth}
 || \om(t, \cdot )||_{\inf} \leq || \om_0||_{L^{\inf}} \B( \sup_{ |b_0| = |\xi_0| = 1, x_0 \in \Om, b_0 \cdot \xi_0 = 0} | b_t(x_0, \xi_0, b_0)| \B)^2. 
\eeq
 According to the BKM blowup criterion, $|| \om(t)||_{\inf}$ must blowup, which leads to the growth of $b_t$ and $|| v(t)||_{L^p}$ and implies linear instability.

\section{Embedding inequalities and estimates of nonlinear terms}\label{app:high}

\paragraph{\textbf{Notation}}
We use the notation $A \asymp B$ if there are some absolute constants $C_1, C_2>0$ with $A \leq C_1 B, B \leq C_2 A$.


We have the following equivalence, which allows us to generalize the lower order nonlinear estimates in \cite{chen2019finite2}, which are based on the $\cC^1$ and $\cH^3$ norms, to the higher order easily.

\begin{prop}\label{prop:equiv}
Let $\cH^k(\rho)$ and $\cC^k$ be the norms defined in \eqref{norm:Hk} and \eqref{norm:ck} with $\rho_1 \les \rho_2$. For $ i + j + k \leq m$, we have 
\beq\label{eq:norm_include}
 || D_R^i D_{\b}^j f ||_{\cC^k } \les || f||_{\cC^m}, \quad 
  || D_R^i D_{\b}^j f ||_{\cH^k(\rho) } \les || f||_{\cH^m(\rho)}. 
\eeq
For $k \geq 1 $, we have 
\beq\label{eq:equiv_Ck}
|| f||_{\cC^k } \asymp \sum_{i+j \leq k-1} || D_R^i D_{\b}^j f ||_{\cC^1}.
\eeq
For $k\geq 3$, we have 
\beq\label{eq:equiv_Hk}
|| f||_{\cH^k(\rho)} \asymp \sum_{ i + j \leq k-3} || D_R^i D_{\b}^j f ||_{\cH^3(\rho)}.
\eeq
\end{prop}

\begin{proof}[Proof of Proposition \ref{prop:equiv}]
The inequalities \eqref{eq:norm_include} follow from the definitions of the norms \eqref{norm:Hk}, \eqref{norm:ck}. The key point is that in the definitions  \eqref{norm:Hk},\eqref{norm:ck}, the weight associated with the mixed derivatives $D_R^i D_{\b}^j, i , j \neq 0$ is larger than that with $D_R^i$ or $D_{\b}^j$. 

For \eqref{eq:equiv_Ck} and \eqref{eq:equiv_Hk}, the $\gtr$ side of the inequality follows from \eqref{eq:norm_include}. For the $\les$ side of \eqref{eq:equiv_Ck}, we have 
\[
\bal
|| f||_{\cC^k} & \les || f||_{\inf} + \sum_{i+j \leq k} || ( \phi_1 \one_{i\geq 1 } 
+ \phi_2 \one_{ j \geq 1} ) D_{R}^i D_{\b}^j  ||_{\inf}  \\
&\les || f||_{\inf} + \sum_{ i+j \leq k-1} ( ||  \phi_1 D_R^{i+1} D_{\b}^j  f||_{\inf}
+ ||  \phi_2 D_R^{i} D_{\b}^{j+1}  f||_{\inf} ) \les  \sum_{i+j \leq k-1} ||  D_R^{i} D_{\b}^j  f||_{\cC^1}.
\eal
\]

Denote by $A, B$ the left and right hand side of \eqref{eq:equiv_Hk}, respectively. From \eqref{norm:Hk}, we get 
\[
A = \sum_{ i \leq k} || \rho_1^{1/2} D_R^i f ||_2  
+ \sum_{ j \geq 1, i+ j \leq k} || \rho_2^{1/2} D_R^i D_{\b}^j f ||_2.
\]
We remark that the $\rho_1$ weight only applies to $D_R^i f$ terms, and $\rho_2$ applies to other derivatives. We have 
\[
|| \rho_1^{1/2} D_R^i f ||_2  \leq \one_{ i \leq 3} || f||_{\cH^3(\rho)}
+ \one_{ i \geq 3} || D_R^{i-3} f||_{\cH^3(\rho)} \les B.
\]

Denote $j_1 = \min(3, j) \geq 1,  i_1 = 3 - j_1$. When $ i + j \geq 3$, we get $i - i_1 = i + j_1 - 3 = \min( i , i+j - 3) \geq 0$. Since $j \geq 1$, we yield 
\[
\bal
|| \rho_2^{1/2} D_R^i D_{\b}^j f ||_2
&\leq \one_{i+j \leq 3} || f||_{\cH^3(\rho) } + \one_{i+j \geq 3} 
|| \rho_2^{1/2} D_R^{i_1} D_{\b}^{j_1}  ( D_R^{i-i_1} D_{\b}^{j-j_1} f) ||_{\cH^3(\rho)} \\
&\les || f||_{\cH^3(\rho)} + || D_R^{i-i_1} D_{\b}^{j-j_1} f||_{\cH^3(\rho) } \les B.
\eal 
\]
We conclude the proof. 
\end{proof}

\subsection{Higher order embedding Lemmas}

We have the following estimates for different norms. The first and last inequality generalize Proposition 7.6 in \cite{chen2019finite2}. The second inequality is exactly Proposition 7.7 in \cite{chen2019finite2}. The third inequality in \eqref{eq:norm_embed} generalizes Lemma 7.11 in \cite{chen2019finite2}. 
\begin{prop}\label{prop:embed}
Let $\cC^k$ and $\cW^{k,\inf}$ be the norms defined in \eqref{norm:ck} and \eqref{norm:W}. 
For $ k \geq 1$, 
\beq\label{eq:norm_embed}
\bal
|| f g||_{\cC^k} & \les || f||_{\cC^k} || g||_{\cC^k},  \quad
|| f g ||_{\cW^{k,\inf}}  \les || f ||_{\cW^{k,\inf}} || g||_{\cW^{k, \inf}} , \\
  || f||_{\cC^{ k}} & \les \al^{-1/2} || f||_{ \cH^{k+2}},  \quad || f||_{\cC^k} \les || \f{1+R}{R} f ||_{\cW^{k, \inf}}.
\eal
\eeq
\end{prop}

\begin{proof}
The first inequality follows from the Leibniz rule. The second inequality has been proved in \cite{chen2019finite2arXiv}. For the third inequality, the case $k=1$ has been proved in Lemma 7.11 in \cite{chen2019finite2}
\beq\label{eq:H3_C1}
 || f ||_{\cC^1} \les \al^{-1/2} || f||_{\cH^3}.
\eeq
For $k \geq 2$, using \eqref{eq:H3_C1} and the equivalences \eqref{eq:equiv_Hk}, \eqref{eq:equiv_Ck}, we obtain 
\[
|| f||_{\cC^k} \les \sum_{i+j \leq k-1} || D_R^i D_{\b}^j f||_{\cC^1}
\les \al^{-1/2} \sum_{i+j \leq k-1}  || D_R^i D_{\b}^j f||_{\cH^3} 
\les \al^{-1/2} || f||_{\cH^{k+2}}.
\]

Next, we consider the last inequality in \eqref{eq:norm_embed}. By the triangle inequality and the Leibniz rule, it is not difficult to obtain the equivalence 
\beq\label{eq:embed_pf1}
\bal
|| \f{1 + R}{R} f ||_{\cW^{k, \inf}}  \asymp 
 \sum_{0 \leq i + j \leq k ,  j \neq 0} \B| \B| 
\f{1+R}{R} \sin(2\b)^{- \f{\al}{5}} \f{ D_R^i  D_{\b}^j  }{\f{\al}{10} + \sin(2\b)}   f   \B|\B|_{L^{\infty}}  + \sum_{0 \leq i  \leq k } \B| \B|\f{1+R}{R}  D_R^i f \B|\B|_{L^{\infty}}.
\eal
\eeq

By definition of $\cC^k$ \eqref{norm:ck} , it suffices to show 
\[
||  (1 + \one_{i\geq 1} \phi_1 + \one_{ j \geq 1} \phi_2  ) D_R^i D_{\b}^j f ||_{\inf} \les || \f{1 + R}{R} f ||_{\cW^{k, \inf}} ,
\]
for $i+j \leq k$. The estimate is trivial if $ j =0$. If $ j \geq 1$, we compare the weights. Since 
\[
\bal
&\f{1+R}{R}\sin(2\b)^{- \f{\al}{5} }  (\f{ \al}{10 } + \sin(2\b))^{-1}
\gtr \f{1+R}{R}\sin(2\b)^{-\al/5 } ,   \\
& \sin(2\b)^{-\al/5 } \gtr 1 + \sin(2\b)^{-\al / 40},  \quad \f{1+R}{R}  \gtr 1,
\eal
\]
the weight $1 + \one_{i\geq 1} \phi_1 + \one_{ j \geq 1} \phi_2$ with $j \geq 1$ can be bounded by the corresponding weight in \eqref{eq:embed_pf1}. We conclude the proof. 
\end{proof}

We have the following elliptic estimates for the stream function \eqref{eq:L12}. 
\begin{prop}\label{prop:key}
Assume that $\al \leq \f{1}{4}$ and $\Om \in \cH^k, k \geq 3$. Let $\Psi$ be the 
solution to \eqref{eq:elli} with boundary condition \eqref{eq:ellibc}. Then we have
\[
\bal
&\al^2 || R^2 \pa_{RR} \Psi ||_{\cH^k} + 
\al || R \pa_{R \b} \Psi ||_{\cH^k}  +|| \pa_{\b\b} (\Psi - \f{1}{\al \pi} \sin(2\b) L_{12}(\Om)) ||_{\cH^k} \les_k  || \Om ||_{\cH^k}.
\eal
\]
\end{prop}
The above estimate with $k =3 $ has been established in \cite{chen2019finite2}. The general case $k \geq 3$ can be proved similarly. See also \cite{elgindi2019finite}.

We have the following estimates for the velocity $\bar u$ of the approximate steady state. 
\begin{prop}\label{prop:psi}
For $\al \leq \f{1}{4}$ and $k \geq 5$, we have 
\[
\bal
|| \f{1+R}{ R} \pa_{\b \b} ( \bar{\Psi}  -  \f{\sin(2\b)}{\pi \al}  L_{12}(\bar{\Om})) ||_{\cW^{k+2,\infty}} \les \al , \qquad || L_{12}(\bar{\Om})||_{\cW^{k+2,\infty}} \les \al , \\
\al || \f{1+R}{R}  D_R^2 \bar{\Psi} ||_{\cW^{k,\infty}} +  \al||  \f{1+R}{R} \pa_{\b} D_R \bar{\Psi} ||_{\cW^{k,\infty}}  + || \f{1+R}{ R} \pa_{\b \b} ( \bar{\Psi}  -  \f{\sin(2\b)}{\pi \al}  L_{12}(\bar{\Om})) ||_{\cW^{k,\infty}} \les  \al  \notag. 
\eal
\]
\end{prop}
The case of $k = 5$ has been proved in Proposition 7.8 \cite{chen2019finite2}. The general case $k \geq 5$ follows from a similar argument. See also \cite{elgindi2019finite}.

We generalize Proposition 7.9 in  \cite{chen2019finite2} as follows.
\begin{prop}\label{prop:W2}
Assume that $ \f{(1+R)^3}{R^2} f \in \cW^{k,\infty}$, then we have $f \in \cH^k$ and 
\[
|| f ||_{\cH^k} \les || \f{(1+R)^3}{R^2} f ||_{\cW^{k,\infty}}.
\]
\end{prop}
The proof with $k = 3$ is Proposition 7.9 in \cite{chen2019finite2} with proof given in its arXiv version \cite{chen2019finite2arXiv}. 

\begin{proof}
Denote $g(R) = \f{ (1+R)^3}{R^2}$. Note that for $i \geq 0$
\[
|| D_R^i g(R) | =  | D_R^i  \f{ (1+R)^3}{R^2} | \les \f{ (1+R)^3}{R^2} = g(R).
\]
From this estimate and using induction, it is not difficult to obtain 
\[
|| g(R) f||_{ \cW^{k, \inf}}
\asymp_k   \sum_{0 \leq i + j \leq k ,  j \neq 0} \B| \B| 
g(R) \sin(2\b)^{- \f{\al}{5}}  \f{ D_R^i D_{\b}^j}{\f{\al}{10} + \sin(2\b)}   f   \B|\B|_{L^{\infty}}  + \sum_{0 \leq i  \leq k } \B| \B| g(R)  D_R^i f \B|\B|_{L^{\infty}}.
\]

Now applying the equivalence \eqref{eq:equiv_Hk} in Proposition \ref{prop:equiv}, Proposition \ref{prop:W2} with $k=3$ proved in \cite{chen2019finite2} and the above equivalence on $ \cW^{k,\inf}$, we obtain 
\[
\bal
|| f||_{\cH^k} &\les \sum_{ i+ j \leq k-3} || D_R^i D_{\b}^j f||_{\cH^3}
\les \sum_{ i+ j \leq k-3} || g(R) D_R^i D_{\b}^j f||_{\cW^{3,\inf}}  \\
&\les 
\sum_{i + j \leq k-3}
\B( \sum_{ m + n \leq 3 , n \neq 0} 
|| g(R) \sin(2\b)^{-\al / 5} \f{ D_R^{i+m} D_{\b}^{ j+n}}{ \f{\al}{10 } + \sin(2\b) } f||_{\inf}
+ \sum_{ m \leq 3} || g(R) D_R^{i+m} D_{\b}^j f||_{\inf} \B) \\
&\les || g(R) f||_{\cW{l, \inf}},
\eal
\]
where we have used the fact that $1\les \sin(2\b)^{-\al /5} (\al / 10 + \sin(2\b))^{-1} $ to bound the term $g(R) D_R^{i+m} D_{\b}^j f $ by $|| g(R) f||_{\cW{l, \inf}}$.
\end{proof}

To control the remaining term $\la \cR_{\eta} ,\eta \psi_0 \ra$ in \eqref{energy:E1}, \eqref{energy:Ek}, we need Lemma 7.10 from \cite{chen2019finite2} for the decay estimate of $\xi$. We do not need to generalize it since we only apply it to estimate $\la \cR_{\eta} ,\eta \psi_0 \ra$.
\begin{lem}\label{lem:xi_decay}
Suppose that $\xi \in \cH^2(\psi)$, we have 
\[
|| R^{1/2} \sin(2\b)^{1/4} \xi||_{L^{\infty}} \les || \xi||_{\cH^2(\psi)}.
\]
\end{lem}

\subsubsection{The product rules}
In this subsection, we generalize the estimates of nonlinear terms and the transport terms established in \cite{chen2019finite2} to higher order. 

Denote the sum space $ X_k \teq \cH^k \oplus \cW^{k+2, \infty}$ with sum norm 
\beq\label{norm:X}
|| f||_{X_k} \teq \inf \{ || g||_{\cH^k} + || h||_{\cW^{k+2,\infty}} :  f = g + h  \}.
\eeq

We generalize the $\cH^3$ product rules in Proposition 7.12 \cite{chen2019finite2} as follows.
\begin{prop}\label{prop:prod1}
For all $ f \in X,  g \in \cH^3,  \xi \in \cH^3(\psi) \cap \cC^1$, we have 
\beq\label{eq:prod1}
\bal
|| f g  ||_{\cH^k} & \les \al^{-1/2}  || f||_{X_k} || g  ||_{\cH^k} ,  \\
|| f \xi||_{\cH^k(\psi)} & \les  \al^{-1/2}  || f ||_{X_k} ( \al^{1/2} ||\xi||_{\cC^{k-2}} + || \xi||_{\cH^k(\psi)}) . 
\eal
\eeq
\end{prop}
The case $k=3$ has been established in \cite{chen2019finite2}, which will be used in the following proof.

\begin{proof}
We focus on the inequality for $\cH^k(\psi)$, which is more difficult. Using the equivalence \eqref{eq:equiv_Hk} and the Leibniz rule, we yield 
\[
S \teq || f \xi ||_{\cH^k( \psi)}
\les \sum_{ i+j \leq k-3} || D_R^i D_{\b}^j ( f \xi) ||_{\cH^3(\psi)}
\les \sum_{ i+j \leq k- 3}
\sum_{ p \leq i, q \leq j}
 || D_R^p D_{\b}^q f D_R^{i-p} D_{\b}^{j - q} \xi ||_{\cH^3(\psi)}.
\]
Applying the above Proposition \ref{prop:prod1} with $k=3$ and Proposition \ref{prop:equiv}, we prove
\[
\bal
S &\les  \al^{-1/2} \sum_{ i+j \leq k- 3}
\sum_{ p \leq i, q \leq j} || D_R^p D_{\b}^q f ||_{X_3}
( \al^{1/2} || D_R^{i-p} D_{\b}^{j-q} \xi||_{\cC^{1}} + || 
D_R^{i-p} D_{\b}^{j-q} \xi||_{\cH^3(\psi)})   \\
&\les \al^{-1/2} || f||_{X^k} ( \al^{1/2} || \xi||_{\cC^{k-2}} 
+ || \xi||_{\cH^k(\psi)} ) .
\eal
\]
The desired inequality follows. The proof of the first inequality in \eqref{eq:prod1} is similar.
\end{proof}

Recall the inner products $\la \cdot, \cdot \ra_{\cH^k}, \la \cdot, \cdot \ra_{\cH^k(\psi)}$ from \eqref{eq:inner}. We generalize the $\cH^3$ estimates of the transport terms in Propositions 7.13, 7.14, 7.15 \cite{chen2019finite2} to the following Propositions \ref{prop:tran1}-\ref{prop:tran3}.


\begin{prop}\label{prop:tran1}
Assume that $u, \pa_{\b}u, D_R u \in \cH^k$ and $\Om \in \cH^k, \xi \in \cH^k(\psi) \cap \cC^{k-2}$ we have 
\[
\bal
| \la \Om , u D_R \Om \ra_{\cH^k} |& \les \al^{-\f{1}{2}} \lt( ||u||_{\cH^k} + || \pa_{\b} u ||_{\cH^k} 
+ || D_R u ||_{\cH^k}   \rt) || \Om ||^2_{\cH^k} , \\
| \la \xi , u D_R \xi\ra_{\cH^k(\psi)} | & \les \al^{-\f{1}{2}} \lt( ||u||_{\cH^k} + || \pa_{\b} u ||_{\cH^k} + || D_R u ||_{\cH^k}   \rt) ( || \xi ||_{\cH^k(\psi)} + \al^{1/2} ||\xi||_{\cC^{k-2} }  )^2.
\eal
\]
Moreover, for all $u, D_R u \in X_k = \cH^k \oplus \cW^{k+2,\infty}$ and $\Om \in \cH^k, \xi \in \cH^k(\psi)
\cap \cC^{k-2}$, we have 
\[
\bal
| \la \Om , u D_{\b} \Om \ra_{\cH^k} | &\les \al^{-1/2} \lt( || u||_{X_k} + || D_R u ||_{X_k} ) \rt)  
|| \Om ||^2_{\cH^k} ,\\
| \la \xi , u D_{\b} \xi \ra_{\cH^k(\psi)} | & \les \al^{-1/2} \lt( || u||_{X_k} +
  || D_R u ||_{X_k}) \rt)  ( || \xi ||_{\cH^k (\psi)} + \al^{1/2} ||\xi||_{\cC^{k-2}} )^2.
\eal
\]
\end{prop}


\begin{prop}\label{prop:tran2}
Let $\cH^k(\rho)$ be either $\cH^k$ or $\cH^k(\psi)$ defined in \eqref{norm:Hk}.  For all $g \in \cH^k(\rho)$, $u$ with $ || D_R^i u ||_{L^{\infty}} < \infty $ for $i \leq k$ and $|| D_R^i D^j_{\b} \pa_{\b} u ||_{L^{\infty}} < \infty$ for $ i+ j \leq k-1$, we have 
\[
\bal
| \la g , u D_R g \ra_{\cH^k(\rho)} | & \les \al^{-1/2} ( 
\sum_{ 0 \leq i \leq k} || D_R^i u ||_{L^{\infty}} + \sum_{i + j \leq k-1}|| D_R^i D^j_{\b} \pa_{\b} u ||_{L^{\infty}}  ) || g ||^2_{\cH^k(\rho)} ,\\
\eal
\]
\end{prop}

\begin{prop}\label{prop:tran3}
 Let $\Psi$ be a solution of \eqref{eq:elli}. Suppose that $g, \Om \in \cH^k, \xi \in \cH^k(\psi) \cap \cC^{k-2}$. We have 
\[
\bal
| \la g ,  \f{1}{\sin(2\b)} D_R \Psi D_{\b} g \ra_{\cH^k} | &\les \al^{-3/2} || \Om ||_{\cH^k}   || g ||^2_{\cH^k } ,\\
| \la \xi , \f{1}{\sin(2\b)} D_R \Psi D_{\b}  \xi \ra_{\cH^k(\psi)} | & \les \al^{-3/2} 
|| \Om ||_{\cH^k }  ( || \xi ||_{\cH^k(\psi)} + \al^{1/2} || \xi||_{\cC^{k-2}} )^2.
\eal
\]
\end{prop}

The case of $k=3$ in Propositions \ref{prop:tran1}-\ref{prop:tran3} has been proved in \cite{chen2019finite2}.  The ideas of the proof of the above propositions are simple. To estimate a typical term 
\[
S = \la  D_R^i D_{\b}^j,  \ \rho_j \cdot D_R^i D_{\b}^j( f Dg ) \ra
\]
in the expansion of $\la g, f D g \ra_{\cH^k(\rho)}$ with $\cH^k(\rho) = \cH^k, \rho_j = \vp_1 \one_{j= 0} + \vp_2 \one_{ j \neq 0 }$ or $\cH^{\psi}, \rho = \one_{j=0} \psi_1 + \one_{j \neq 0} \psi_2$ \eqref{eq:inner}, we perform integration by parts if all the derivatives $D_R^i D_{\b}^j$ falls on $Dg$
\[
|\la  D_R^i D_{\b}^j g, \ \rho_j f \cdot D_R^i D_{\b}^j Dg ) \ra|
\les || \rho_j^{-1} D( \rho_j f )  ||_{\inf} || \rho_j^{1/2} D_R^i D_{\b}^j g ||_2^2,
\]
which can be further bounded by the desired upper bound. In other cases, we estimate 
\[
\B|\B\la  D_R^i D_{\b}^j g, \rho_j D_R^m D_{\b}^nf  \cdot D_R^{i-m} D_{\b}^{j-n} Dg  \B\ra\B|
\les || \rho_j^{1/2}  D_R^i D_{\b}^j g ||_2^2 \cdot
|| \rho_j^{1/2} D_R^m D_{\b}^nf  \cdot D_R^{i-m} D_{\b}^{j-n} Dg ) ||_2^2,
\]
for some $m + n \geq 1, m \leq i, n \leq j$. Since the number of derivatives on $f$ is less than $m+n \leq k$, and that on $ g$ is less than $i-m + j-n+1 \leq  i+j\leq k$, the estimate of $|| \rho_j^{1/2} D_R^m D_{\b}^nf  \cdot D_R^{i-m} D_{\b}^{j-n} Dg ) ||_2$ follows from the same method of estimating the product in Proposition \ref{prop:prod1}. We refer to \cite{chen2019finite2} for the estimates in the case of $k=3$, which can be generalized to $k\geq 3$ in a straightforward manner. 



We generalize the $\cH^3$ estimates in Proposition 7.16 in \cite{chen2019finite2} to the $\cH^k$ estimate. 
\begin{prop}\label{prop:prod3}
Let $\Psi, \bar{\Psi}$ be a solution of \eqref{eq:elli} with source term $\Om, \bar{\Om}$, respectively, and $V_1(\Psi)$ be the operator associated to $v_x$ 
\beq\label{eq:vel_vx}
\bal
V_1(\Psi)   =& \al (1 + 2\cos^2 \b) D_R \Psi - \al D_R D_{\b}\Psi -  D_{\b} \Psi_* + 2 \Psi_*  + \sin^2 (\b) \pa^2_{\b} \Psi_*  \\
& + \al^2 \cos^2(\b) D_R^2 \Psi \teq A(\Psi)  + \al^2 \cos^2(\b) D_R^2 \Psi.
\eal
\eeq

Assume that $\xi \in \cH^k(\psi) \cap \cC^{k-2}, \Om \in \cH^k$. We have 
\beq\label{eq:prod3}
\bal
|| V_1(\Psi) \xi||_{\cH^k}  &\les  \al^{-1/2}  || \Om||_{\cH^k} ( 
\al^{1/2} || \xi||_{\cC^{k-2} } + || \xi||_{\cH^k(\psi)}),  \\
\quad || V_1(\bar{\Psi}) \xi ||_{\cH^k }  &\les \al^{1/2}  || \xi||_{\cH^k(\psi)}.
\eal
\eeq
\end{prop}

We refer the derivation of \eqref{eq:vel_vx} to Section 8.1 in \cite{chen2019finite2}. The difficulty is due to the fact that $\cH^k(\psi)$ is weaker than $\cH^k$ (see Lemma \ref{lem:H2H2}). Moreover, it is more difficult to control the singular weight $\phi_2$ in the $\cH^k$ norm, which is singular in $\b$. Thus we cannot apply Proposition \ref{prop:prod1} directly to estimate $v_x \xi$. 

\begin{lem}\label{lem:H2H2}
Let $\g  = 1 + \f{\al}{10} , \s = \f{99}{100}$ be the parameter given in Definition \ref{def:wg}. 
For $  \f{\g - \s}{2} \leq \lam \leq \f{1}{2}$ and $m \geq  1$, we have 
\beq\label{eq:H2H2}
|| f||_{\cH^m(\psi)}  \les  || f||_{\cH^m} , \quad  || \sin(\b)^{\lam} f ||_{\cH^m}  \les || f||_{\cH^m(\psi)} .
\eeq
\end{lem}
The case of $m \leq 3$ has been proved in \cite{chen2019finite2}. The general case follows from the same argument, which compares the corresponding weights in $|| \sin(\b)^{\lam} f ||_{\cH^m}$ and $|| f||_{\cH^m(\psi)}$.

\begin{proof}[Proof of Proposition \ref{prop:prod3}]
The proof of the second inequality in \eqref{eq:prod3} follows from the product rules in Proposition \ref{prop:prod1}, the elliptic estimates in Proposition \ref{prop:psi}, and the argument in \cite{chen2019finite2}. 

 The proof of the first inequality in \eqref{eq:prod3} also follows from the argument in \cite{chen2019finite2} for the special case $k=3$. Note that $V_1(\Psi)$ vanishes on $\b = 0$. Thus, we have $ \sin(2\b)^{-1/2} V_1( \Psi)  \in \cH^k, \sin(2\b)^{1/2} \xi \in \cH^k$ for $\xi \in \cH^k(\psi)$, which allows us to apply the product rules similar to Proposition \ref{prop:prod1}. We only give a sketch and refer related details to \cite{chen2019finite2}.


Recall the decomposition \eqref{eq:vel_vx}. From Propositions \ref{prop:key} and Hardy's inequality, we get $\sin(2\b)^{-1/2} A(\Psi ) \in \cH^k$. Using Propositions \ref{prop:key}, \ref{prop:prod1} and Lemma \ref{lem:H2H2}, we yield 
\[
|| A(\Psi) \xi ||_{\cH^k}  \les \al^{-1/2} || \sin(2\b)^{-1/2} A(\Psi) ||_{\cH^k}
|| \sin(2\b)^{1/2} \xi||_{\cH^k} 
\les \al^{-1/2} || \Om||_{\cH^k} || \xi||_{\cH^k(\psi)}. 
\]


To estimate $ ||\al^2 D_R^2\Psi \xi||_{\cH^k}$, from \eqref{norm:Hk}, we need to estimate two types of terms 
\[
I =  || \vp_1^{1/2} D_R^i ( \al^2  D_R^2 \Psi \cdot \xi ) ||_{2}, \quad 
II = || \vp_2^{1/2} D_R^m D_{\b}^n D_{\b} ( \al^2  D_R^2 \Psi \cdot \xi ) ||_{2}
\]
for some $i\leq k, m+ n \leq k-1$. Since $\vp_1 \asymp \psi_1$ \eqref{wg}, using definition of $\cH^k(\psi)$ in \eqref{norm:Hk} and Proposition \ref{prop:prod1}, we get 
\[
 |I| \les || \al^2 D_R^2 \Psi \cdot \xi ||_{\cH^k(\psi)} 
 \les  \al^{-1/2}  || \Om||_{\cH^k} (  \al^{1/2} || \xi||_{\cC^{k-2} } + || \xi||_{\cH^k(\psi)}).
\]

For $II$, it contains at least one $D_{\b}$ derivative. We perform the following decomposition 
\[
\bal
D_{\b} ( \al^2 D_R^2 \Psi \cdot \xi)
&=  \sin(2\b)^{1/4} \B( \al^2 D_R^2 \pa_{\b} \Psi \cdot \sin (2\b)^{3/4} \xi 
+ \sin(2\b)^{-1/2} \al^2 D_R^2  \Psi \cdot \sin(2\b)^{1/4} D_{\b} \xi \B) \\
&\teq \sin(2\b)^{1/4} ( J_1 \cdot J_2 + J_3 \cdot J_4) \teq \sin(2\b)^{1/4} J.
\eal
\]

Since $m+n \leq k-1$ and $ \sin(2\b)^{1/4} \vp_2^{1/2} \les  \vp_1^{1/2} $, using the triangle inequalities and Propositions \ref{prop:prod1}, we yield 
\[
II = || \vp_2^{1/2} D_R^m D_{\b}^n ( \sin(2\b)^{1/4} J) ||_2 
\les 
 \sum_{ l\leq n} 
|| \vp_2^{1/2} \sin(2\b)^{1/4} D_R^m D_{\b}^l J ||_2
\les \sum_{ l \leq n} || \vp_1^{1/2} D_R^m D_{\b}^l J ||_2
\les || J||_{\cH^{k-1}}.
\]

Applying Propositions \ref{prop:prod1}, \ref{prop:key}, and Lemma \ref{lem:H2H2}, we get
\[
\bal
|II| 
&\les \al^{-1/2} ( || J_1||_{\cH^{k-1}} || J_2||_{\cH^{k-1}} 
+ || J_3||_{\cH^{k-1}} || J_4||_{\cH^{k-1}}  ) \\
&\les \al^{-1/2} || \Om||_{\cH^k} ( || \sin(2\b)^{3/4}  \xi||_{\cH^{k-1} } 
+  || \sin(2\b)^{1/4} D_{\b} \xi||_{\cH^{k-1} } )
\les \al^{-1/2}  || \Om||_{\cH^k} || \xi||_{\cH^k(\psi)}.
\eal
\]
We conclude the proof.
\end{proof}

\section{Estimate of the approximate steady state}\label{app:profile}

Recall from \eqref{eq:nota_om} that $\bar \Om, \bar \eta, \bar \xi$ denote the approximate steady state $\bar \om, \bar\th_x, \bar \th_y$ under the coordinate $(R, \b)$, and  the formula of $\bar{\Om}, \bar{\eta}$ in \eqref{eq:profile}. 
\beq\label{eq:bar0}
\bar{\Om} = \f{\al}{c}  \f{3 R \G(\b) }{(1+R)^2},  \quad
\bar{\eta} = \f{\al}{c} \f{6 R \G(\b)  }{(1+R)^3} .
\eeq


We generalize Lemma A.6 in \cite{chen2019finite2} from $k \leq 3$ to any $k$ below.
\begin{lem}\label{lem:bar}
The following results apply to any $ k \geq 0, 0 \leq i + j \leq k, j \neq 0$.
(a) For $f = \bar{\Om}, \bar{\eta}, \bar{\Om} - D_R \bar{\Om}, \bar{\eta} - D_R \bar{\eta}$, we have
\beq\label{eq:bar}
| D_R^k f | \les_k f , \quad  |D_R^i D^j_{\b} f | \les_k \al \sin(\b) f.
\eeq
(b) Let $\vp_i$ be the weights defined in \eqref{wg}. For $g = \bar{\Om}, \bar{\eta}$, we have
\beq\label{eq:bar_ux}
\int_0^{\pi /2} R^2 (D_R^k g )^2 \vp_1 d \b \les_k \al^2,  \quad
\int_0^{\pi /2} R^2 (D_R^i D^j_{\b} g )^2 \vp_2  d \b \les_k \al^3, 
\eeq
uniformly in $R$ and 
\beq\label{eq:bar_ing}
\la  (D^k_R  ( g - D_R g   ) )^2 , \vp_1 \ra  \les_k \al^2, \quad
\la  (D^i_R D^j_{\b} ( g - D_R g   ) )^2 , \vp_2  \ra  \les_k \al^3.
\eeq
\end{lem}

We generalize Lemma A.7 in \cite{chen2019finite2} from $k = 7$ to any $k \geq 7$ below.
\begin{lem}\label{lem:gam} 
For any $k \geq 7$, it holds true that
$\G(\b) ,\bar{\Om}, \bar{\eta} \in \cW^{k, \infty}$ with
\[
\bal
&|| \G(\b)||_{\cW^{k,\infty}} \les_k 1, \quad
||\f{(1+R)^2}{R} \bar{\Om}||_{\cW^{k,\infty}}  + || \f{(1+R)^2}{R}\bar{\eta} ||_{\cW^{k,\infty}} \les_k \al ,  \\
&|| D_{\b} \bar{\Om}||_{\cW^{k,\infty}}  + || D_{\b}\bar{\eta} ||_{\cW^{k,\infty}} \les_k \al^2 .
\eal
\]
\end{lem}

We generalize Lemma A.8 in \cite{chen2019finite2} from $k= 5$ to any $k \geq 5$ below.
\begin{lem}\label{lem:xi}
Assume that $0\leq \al \leq \f{1}{1000}$. For $R \geq0, \b \in [0, \pi/2], k \geq 1$ and $0 \leq i + j \leq k$,
we have
\begin{align}
    & | D^i_R D^j_{\b} \bar{\xi} | \les_k  - \bar{\xi},  \quad  | D^i_R D^j_{\b} (3\bar{\xi} - R\pa_R \bar{\xi}) | \les_k -\bar{\xi}, \label{eq:xi0} \\
 & |\bar{\xi} |  \les  \f{\al^2R^2}{1+R} \lt( \one_{\b < \pi /4} \f{ \sin^{\al}(\b)}{ (1+R \sin^{\al}(\b) )^3} 
+ \one_{\b \geq \pi/4} \f{\cos^{\al+1}(\b)}{(1+R)^3}  \rt) 
\label{eq:xi} , \\
  & -\bar{\xi } \les \al^2 \cos(\b) , \quad   || \bar{\xi} ||_{\cC^k} \les  || \f{1+R}{R} 
( 1 + ( R \sin( 2 \b)^{\al} )^{-\f{1}{40} } ) \bar{\xi} ||_{L^{\infty}} \les \al^2, \notag
\end{align}
where $||\cdot ||_{\cC^k}$ is defined in \eqref{norm:ck}. Let $\psi_1, \psi_2$ be the weights defined in \eqref{wg}. We have 
\beq\label{eq:xi_ux}
\int_0^{\pi /2} R^2 (D^i_R D^j_{\b} \bar{\xi} )^2 \psi_k d \b  
\les \al^4  
\eeq
uniformly in $R$, and 
\beq\label{eq:xi_cw}
\la  (D^i_R D^j_{\b} ( 3\bar{\xi} - R \pa_R \bar{\xi}  ) )^2 , \psi_k \ra  
\les \al^4, \quad  \la  (D^i_R D^j_{\b}  \bar{\xi})^2 , \psi_k \ra  
\les \la \bar{\xi}^2 , \psi_k \ra \les \al^4,
\eeq
where $ (D^i_R D^j_{\b},  \psi_k )$ represents 
$ ( D^i_R, \psi_1)$ for $0\leq i \leq k$,  and $ (D^i_R D^j_{\b},  \psi_2 )$ for $i+j \leq k, j \geq 1$.
\end{lem}

The proofs of Lemmas \ref{lem:bar}-\ref{lem:xi} follows from the argument in \cite{chen2019finite2}, and thus are omitted.

For the $L_{12}$ operator \eqref{eq:L12}, we generalize Lemma A.4 in \cite{chen2019finite2} from $\cH^3$ to its $\cH^k$ version. The proof follows from a similar argument.

\begin{lem}\label{lem:l12}
Let $\chi(\cdot) : [0, \infty) \to [0, 1]$ be a smooth cutoff function, such that $\chi(R) = 1$ for $R \leq 1$ and $\chi(R) = 0$ for $R \geq 2$. For $k=1,2$, we have
\beq\label{eq:l12}
\bal
& 
 || L_{12}(\Om) ||_{L^{\infty}}\les || \f{1+R}{R} \Om||_{L^2}, \quad || \td{L}_{12}(\Om) (R^{-2} + R^{-3})^{1/2} ||_{L^2(R)}  \les || \Om \f{(1+R)^2}{R^2} ||_{L^2} ,
 \\
& || L_{12}(\Om)||_2 \les || \Om||_2, \quad 
|| \f{(1+R)^k}{R^k} ( L_{12}(\Om) - L_{12}(\Om)(0) \chi ) ||_{L^2(R)} \les || \f{(1+R)^k}{R^k} \Om||_{L^2},
\eal
\eeq
provided that the right hand side is bounded. 
Moreover, if $\Om \in \cH^n$, then for $ 0 \leq k \leq n, 0\leq l \leq n-1, n \geq 3$, we have 
\beq\label{eq:l12X}
\bal
& ||   L_{12}(\Om) - L_{12}(\Om)(0) \chi ||_{\cH^n} + ||  D_R( L_{12}(\Om) - L_{12}(\Om)(0) \chi) ||_{\cH^n} \les_n || \Om ||_{\cH^n},   \\
& 
|| D^k_R L_{12}(\Om) ||_{\infty} 
+ || D^k_R ( L_{12}(\Om) -\chi L_{12}(\Om)(0)) ||_{\infty} \les_n || \Om||_{\cH^n},  \\
&
|| (1 +R) \pa_R D^l_R L_{12}(\Om) ||_{\infty} 
+ || (1+R) \pa_R D^l_R ( L_{12}(\Om) -\chi L_{12}(\Om)(0)) ||_{\infty} \les_n || \Om||_{\cH^n},  \\
& || L_{12}(\Om) ||_{X_n} + || D_R L_{12}(\Om) ||_{X_n} \les_n  || \Om||_{\cH^n},
\eal
\eeq
where $X_n \teq \cH^n \oplus \cW^{n+2,\infty}$ is defined in \eqref{norm:X}.
\end{lem}


\section{Some derivations}

The following formulas of velocity in the $(R, \b)$ coordinate are derived in Section 8.1 in \cite{chen2019finite2}
\beq\label{eq:vel_full}
\bal
U(\Psi) &= -\f{2r\cos(\b)}{\pi \al} L_{12}(\Om) - 2 r \sin(\b) \Psi_*  - \al r  \sin \b D_R \Psi  -  r\cos \b  \pa_{\b}\Psi_*, \\
V(\Psi) &= \f{2r \sin(\b)}{\pi \al} L_{12}(\Om) + 2 r \cos \b \Psi_* + \al r  \cos \b D_R \Psi - r \sin \b \pa_{\b} \Psi_*, 
\quad \Psi_* = \Psi - \f{\sin(2 \b)}{\pi \al } L_{12}(\Om),
\\ 
\eal
\eeq
where $\Psi$ is the solution of \eqref{eq:elli}, and  $L_{12}(\cdot), \Psi_*$ are defined in \eqref{eq:L12}.

\subsection{Derivation of \texorpdfstring{$u_r^r( 0,0)$}{Lg} }\label{app:urr}

We derive the formula \eqref{eq:L12_R3} for $u_r^r(0, 0)$ using the formula 
\[
\uu(x) = \na \times (-\D)^{-1} \om =  \f{1}{4\pi} \int_{\R^3} \f{\om(y) \times (x- y)}{ |x-y|^3} dy .
\]

Recall the coordinates and change of variables \eqref{eq:polar_wholeR3}
\[
\b = \arctan( z / r), \quad \rho = (r^2 + z^2)^{1/2}, \quad R = \rho^{\al}, \quad \Om(R, \b) = \om^{\vth}(\rho, \b),
\]
where $(r,  \vth, z)$ is the cylindrical coordinate in $\R^3$ \eqref{eq:polar_basis}. Note that $u_r^r(0,0) = -\f{1}{2} u_z^z(0,0)$ \eqref{eq:euler2}, we compute $u_z^z(0, 0)$. Since there is no swirl $u^{\vth} \equiv 0$, we get
\[
\om = \om^{\vth} e_{\vth} = ( - \om^{\vth}  \sin \vth, \om^{\vth} \cos \vth, 0) , \quad 
( \om \times (x - y) )_3 
= - \om^{\vth} \sin(\vth ) (x_2 - y_2 ) - \om^{\vth} \cos(\vth ) (x_1 - y_1). 
\]
Since the above formula is independent of $z= x_3$ and $\om^{\vth}(y)$ is odd in $y_3$, we yield 
\[
\pa_3 u^3 = \f{1}{4\pi} \int_{\R^3} (\om \times (x - y))_3  \pa_{x_3} \f{1}{|x-y|^3} dy 
=  \f{1}{4\pi} \int_{\R^3} (\om \times (x - y))_3  \f{-3( x_3 - y_3)}{  |x-y|^5} dy .
\]
Evaluating at $x = 0$ and using 
\[
 ( \om \times ( - y) )_3 =  \om^{\vth}( y) \sin(\vth) y_2 + \om^{\vth} \cos (\vth) y_1,
 =  \om^{\vth}(y) r
\]
and $r = \rho \cos \b, z = \rho \sin \b, \b \in [-\pi/2, \pi/2]$, we obtain
\[
\bal
\pa_3 u^3(0, 0) &= \f{3 }{4\pi} \int_{\R^3} \f{\om^{\vth}(y) r y_3}{ |y|^5} d y
=  \f{3}{4\pi} \int_0^{\inf}\int_0^{2\pi}\int_{\R} \f{\om^{\vth}(y) r z }{ |y|^5} r d r d \vth d z  
=\f{3}{2} \int_{\R_+\times \R} \f{\om^{\vth}(y) r^2 z }{ \rho^5} d r  d z  
 \\ 
&= \f{3}{2} \int_0^{\inf} \int_{-\pi/2}^{\pi/2} \f{\om^{\vth}(\rho, \b) \cos^2(\b) \sin(\b) }{ \rho}  d \rho d \b 
= 3 \int_0^{\inf} \int_{0}^{\pi/2} \f{\om^{\vth}(\rho, \b) \cos^2( \b) \sin(\b) }{ \rho}  d \rho d \b  .
\eal
\]
Using $u_r^r(0,0) = -\f{1}{2}u_z^z(0,0)$ \eqref{eq:euler2} and $\f{ d \rho}{\rho} = \f{1}{\al} \f{dR}{R}$, we prove \eqref{eq:L12_R3}.

\bibliographystyle{plain}
\bibliography{selfsimilar}

\end{document}